\documentclass[a4paper,10pt,twoside]{article}

\usepackage[utf8]{inputenc}
\usepackage{amsmath, amssymb, amsthm}
\usepackage{amstext, amsfonts, a4}
\usepackage{dsfont}
\usepackage{latexsym}
\usepackage{mathrsfs}
\usepackage[all,cmtip]{xy}
\usepackage{graphicx}
\usepackage{enumerate}
\usepackage{hyperref,color}
\usepackage{stmaryrd}
\usepackage[shortlabels]{enumitem}
\usepackage{bbm}
\usepackage{tikz-cd}
\usepackage{adjustbox}
\usepackage{mathtools}

\usepackage{url}
\usepackage[pdftex,left=2cm,right=2cm,top=3cm,bottom=3cm]{geometry}

\def\lra{\rightarrow}

\def\lra{\longrightarrow}

\def\lmapsto{\longmapsto}

 \def\sS{\mathscr{S}}

\def\bbC{\mathbb{C}}
\def\bbF{\mathbb{F}}

\def\bbQ{\mathbb{Q}}\def\bbR{\mathbb{R}}

\def\bbZ{\mathbb{Z}}

\def\cC{\mathcal{C}}

\def\cK{\mathcal{K}}\def\cL{\mathcal{L}}
\def\cO{\mathcal{O}}\def\cP{\mathcal{P}}
\def\cS{\mathcal{S}}

\def\bfB{\mathbf{B}}\def\bfC{\mathbf{C}}
\def\bfG{\mathbf{G}}
\def\bfK{\mathbf{K}}
\def\bfM{\mathbf{M}}\def\bfN{\mathbf{N}}\def\bfP{\mathbf{P}}
\def\bfT{\mathbf{T}}
\def\bfU{\mathbf{U}}
\def\bfZ{\mathbf{Z}}

\def\ff{\mathfrak{f}}
\def\fn{\mathfrak{n}}
\def\ft{\mathfrak{t}}

\def\fN{\mathfrak{N}}

\newcommand{\nD}{{\textnormal{D}}}

\newcommand{\nG}{{\textnormal{G}}}

\newcommand{\nL}{{\textnormal{L}}}

\newcommand{\nR}{{\textnormal{R}}}
\newcommand{\nS}{{\textnormal{S}}}

\DeclareMathOperator{\charr}{char}

\DeclareMathOperator{\cind}{c-ind}

\DeclareMathOperator{\cd}{cd}

\DeclareMathOperator{\End}{End}

\DeclareMathOperator{\Ext}{Ext}

\DeclareMathOperator{\Fil}{Fil}

\DeclareMathOperator{\fin}{finite}

\DeclareMathOperator{\Gal}{Gal}
\DeclareMathOperator{\gr}{gr}

\DeclareMathOperator{\Hom}{Hom}
\DeclareMathOperator{\hgt}{\textnormal{ht}}

\DeclareMathOperator{\Id}{Id}

\DeclareMathOperator{\Ind}{Ind}
\DeclareMathOperator{\ind}{ind}

\DeclareMathOperator{\Rep}{Rep}
\DeclareMathOperator{\Res}{Res}

\DeclareMathOperator{\RHom}{RHom}

\DeclareMathOperator{\uRHom}{\underline{RHom}}

\def\lan{\langle}

\def\lran{\rangle}

\newtheorem{counter}[subsection]{$\!\!$}

\newtheorem{counter*}[subsubsection]{$\!\!$}
\newenvironment{Def*}{\begin{counter*} {\bf Definition.}}{\end{counter*}}
\newenvironment{Not*}{\begin{counter*} \rm {\bf Notation.}}{\end{counter*}}
\newenvironment{Notss*}{\begin{counter*} \rm {\bf Notations.}}{\end{counter*}}
\newenvironment{DefNot*}{\begin{counter*} \rm {\bf Definition-Notation.}}{\end{counter*}}
\newenvironment{Nots*}{\begin{counter*} \rm {\bf Notations.}}{\end{counter*}}
\newenvironment{Prop*}{\begin{counter*} {\bf Proposition.}}{\end{counter*}}
\newenvironment{Lem*}{\begin{counter*} {\bf Lemma.}}{\end{counter*}}
\newenvironment{Cor*}{\begin{counter*} {\bf Corollary.}}{\end{counter*}}
\newenvironment{Th*}{\begin{counter*} {\bf Theorem.}}{\end{counter*}}
\newenvironment{Rem*}{\begin{counter*} \rm {\bf Remark.}}{\end{counter*}}
\newenvironment{Ex*}{\begin{counter*} \rm {\bf Example.}}{\end{counter*}}
\newenvironment{Exs*}{\begin{counter*} \rm {\bf Examples.}}{\end{counter*}}
\newenvironment{Pt*}{\begin{counter*} \rm}{\end{counter*}}
\newenvironment{Q*}{\begin{counter*} \rm {\bf Question.}}{\end{counter*}}
\newenvironment{Qs*}{\begin{counter*} \rm {\bf Questions.}}{\end{counter*}}
\newenvironment{Assn*}{\begin{counter*} {\bf Assumption.}}{\end{counter*}}
\newenvironment{Const*}{\begin{counter*} \rm {\bf Construction.}}{\end{counter*}}

\newtheorem*{Claim}{Claim}

\newcommand*{\longhookrightarrow}{\ensuremath{\lhook\joinrel\relbar\joinrel\rightarrow}}
\newcommand*{\longtwoheadrightarrow}{\ensuremath{\relbar\joinrel\twoheadrightarrow}}

\title{\textbf{Derived Satake morphisms for $p$-small weights in characteristic $p$}}

\author{Karol Kozio{\l} and C\'{e}dric P\'{e}pin}
\date{\today}

\begin{document}

\maketitle

%\subjclass[2010]{22E50 (primary), 20C08, 20J06, 11G80 (secondary)}

\abstract{Let $F$ be a finite unramified extension of $\mathbb{Q}_p$ with ring of integers $\mathcal{O}_F$, and let $\mathbf{G}$ denote a split, connected reductive group over $\mathcal{O}_F$.  We fix a Borel subgroup $\mathbf{B} = \mathbf{T}\mathbf{U}$ with maximal torus $\mathbf{T}$ and unipotent radical $\mathbf{U}$, and let $L(\lambda)$ denote an irreducible representation of $\mathbf{G}(\mathcal{O}_F)$ with coefficients in a sufficiently large field of characteristic $p$.  

Under the assumption that $\lambda$ is a $p$-small and sufficiently regular character and that $p$ is greater than 1 plus the Coxeter number of $\mathbf{G}$, we show that the complex $L(\mathbf{U}(F),\textrm{c-ind}_{\mathbf{G}(\mathcal{O}_F)}^{\mathbf{G}(F)}(L(\lambda)))$ splits as the orthogonal direct sum of its cohomology objects in the derived category of smooth $\mathbf{T}(F)$-representations in characteristic $p$.  (Here $L(\mathbf{U}(F), -)$ denotes Heyer's left adjoint of parabolic induction, from the derived category of smooth $\mathbf{G}(F)$-representations to the derived category of smooth $\mathbf{T}(F)$-representations.)  Consequently, this gives rise to a collection of morphisms of graded spherical Hecke algebras
\begin{flushleft}
$\displaystyle{\bigoplus_{i \in \mathbb{Z}}\textrm{Ext}_{\mathbf{G}(F)}^{i}\left(\textrm{c-ind}_{\mathbf{G}(\mathcal{O}_F)}^{\mathbf{G}(F)}(L(\lambda)),~\textrm{c-ind}_{\mathbf{G}(\mathcal{O}_F)}^{\mathbf{G}(F)}(L(\lambda))\right)}$
\end{flushleft}
\begin{flushright}
$\displaystyle{\longrightarrow \bigoplus_{i \in \mathbb{Z}}\textrm{Ext}_{\mathbf{T}(F)}^{i}\left(\textrm{c-ind}_{\mathbf{T}(\mathcal{O}_F)}^{\mathbf{T}(F)}(L^n(\mathbf{U}(\mathcal{O}_F),L(\lambda))),~\textrm{c-ind}_{\mathbf{T}(\mathcal{O}_F)}^{\mathbf{T}(F)}(L^n(\mathbf{U}(\mathcal{O}_F),L(\lambda)))\right)}$
\end{flushright}
indexed by $n=-[F:\mathbb{Q}_p]\dim(\mathbf{U}), \ldots, 0$, which we refer to as derived Satake morphisms. For $\lambda=0$ and $n=0$, this recovers the graded mod $p$ Satake homomorphism constructed by Ronchetti.

We also give some partial results for general standard parabolic subgroups $\mathbf{P} = \mathbf{M}\mathbf{N} \subset \mathbf{G}$.
 }

\tableofcontents

\section{Introduction}

\subsection{Setting and results}

\begin{Pt*}
This article is motivated by the mod $p$ Local Langlands Program.  Specifically, we study the mod $p$ representation theory of $p$-adic reductive groups, through the viewpoint of derived Satake morphisms.  To describe the setting, let $F$ be a locally compact nonarchimedean field of residual characteristic $p$, with ring of integers $\cO_F$. We let $\bfG$ be a split connected reductive group over $\cO_F$, write $G:=\bfG(F)$ for the locally profinite group of its $F$-points, and $G_0\subset G$ for the compact open subgroup $\bfG(\cO_F)\subset \bfG(F)$. Further, we fix a Borel subgroup $\bfB\subset \bfG$, with unipotent radical $\bfU$, and maximal torus $\bfT$ satisfying $\bfB = \bfT \bfU$; write  $B=TU$ and $B_0=T_0U_0$ for the corresponding groups of rational and integral points.

The classical Satake morphism is the map
\begin{eqnarray*}
\mathbb{C}[G_0\backslash G/ G_0] & \lra & \mathbb{C}[T/T_0] \\
\kappa & \lmapsto & \bigg(t\longmapsto \delta_B(t)^{\frac{1}{2}}\sum_{u\in U/U_0} \kappa(tu)\bigg)
\end{eqnarray*}
where $\delta_B:B\longrightarrow \mathbb{R}_+^{\times}$ is the modulus character. The source and the target are the $\mathbb{C}$-vector spaces freely generated by the displayed cosets; they are $\mathbb{C}$-algebras with respect to the convolution product, and the Satake map is a morphism of $\mathbb{C}$-algebras.  This morphism is fundamental in constructing \textit{unramified} instances of the Local Langlands Correspondence (see \cite[\S 4.2]{C79}).

For a coefficient field $k$ other than $\mathbb{C}$, the above formula still makes sense after omitting the character $\delta_B^{\frac{1}{2}}$; in fact, the map
\begin{eqnarray*}
\mathbb{Z}[G_0\backslash G/ G_0] & \lra & \mathbb{Z}[T/T_0] \\
\kappa & \lmapsto & \bigg(t\longmapsto \sum_{u\in U/U_0} \kappa(tu)\bigg)
\end{eqnarray*}
is a morphism of rings. This was first observed by Herzig \cite{herzig:satake}, who studied the morphism of $\overline{\bbF}_p$-algebras resulting from the scalar extension $\bbZ\longrightarrow\overline{\bbF}_p$, in the case $\charr(F)=0$, and called it the \emph{mod $p$ Satake morphism}. The morphism for an arbitrary coefficient field $k$ of characteristic $p$, and with $\charr(F)=p$ allowed, was studied by Henniart--Vignéras \cite{HV15}. 
\end{Pt*}

\begin{Pt*}
Given a coefficient field $k$, the interest in the convolution algebra $k[G_0\backslash G/ G_0]$ is that it realizes all the endomorphisms of the smooth representation $\cind_{G_0}^G(1_{G_0})$ of $G$, compactly induced from the trivial representation $1_{G_0}$ of $G_0$ over $k$.  Namely, evaluation on the characteristic function of $G_0$ defines an isomorphism $(\End_G(\cind_{G_0}^G(1_{G_0})),\circ) \stackrel{\sim}{\longrightarrow} (k[G_0\backslash G/ G_0],\star)$.  On the other hand, if $k$ has characteristic $p$, then the $G$-representation $\cind_{G_0}^G(1_{G_0})$ admits \emph{higher} endomorphisms: if $\charr(F)=0$ and $G_0$ is $p$-torsion-free, the graded $k$-algebra
$$
\Ext_G^{\bullet}\left(\cind_{G_0}^G(1_{G_0}),~\cind_{G_0}^G(1_{G_0})\right) := \bigoplus_{i\in\bbZ}\Ext_G^i\left(\cind_{G_0}^G(1_{G_0}),~\cind_{G_0}^G(1_{G_0})\right),
$$
equipped with the Yoneda product, is concentrated in degrees $[0,[F:\bbQ_p]\dim(\bfG)]$ and nonzero in degree $[F:\bbQ_p]\dim(\bfG)$. 
%follows from the fact that $G_0$ is a Poincaré group at p in the sense of \cite[Chap. 3 \S 7]{NSW}
\end{Pt*}

\begin{Pt*}
\emph{From now on, we fix a coefficient field $k$ of characteristic $p$}.

In the case $\bfG=\bfT$ is a torus, the graded $k$-algebra above is completely understood, namely
$$
\Ext_T^{\bullet}\left(\cind_{T_0}^T(1_{T_0}),~\cind_{T_0}^T(1_{T_0})\right) \cong k[T/T_0]\otimes_k \sideset{}{^{\bullet}_k}{\bigwedge}\Hom_{\textnormal{cts}}(T_1,k)
$$
where $T_1$ is the pro-$p$-Sylow subgroup of $T_0$. Consequently, for general $G$, it is desirable to relate the graded algebra for $G$ to the one for $T$, that is, to extend the mod $p$ Satake morphism  
$$
\Ext_G^0\left(\cind_{G_0}^G(1_{G_0}),~\cind_{G_0}^G(1_{G_0})\right)\lra \Ext_T^0\left(\cind_{T_0}^T(1_{T_0}),~\cind_{T_0}^T(1_{T_0})\right)
$$
in degree $0$ to a morphism between the full graded algebras.

By definition, the map in degree $0$
\begin{eqnarray*}
k[G_0\backslash G/ G_0] & \lra & k[T/T_0] \\
\kappa & \lmapsto & \bigg(t\longmapsto \sum_{u\in U/U_0} \kappa(tu)\bigg)
\end{eqnarray*}
is summing the $G_0$-bi-invariant \emph{compactly supported kernels} on $G$ along the $U_0$-orbits in $B=TU$ relative to $T$. Hence, in terms of representations, it involves the \emph{homology} functor of the locally profinite group $U$. Now, the relevant left derivatives of this right exact functor have been constructed by Heyer \cite{H22}, when $\charr(F)=0$. Namely, denote by $\nD(H)$ the (triangulated) unbounded derived category of smooth $k$-representations of a locally profinite group $H$. Then:

\begin{Th*} \textbf{\emph{(Heyer, \cite{H22})}}
Let $F/\bbQ_p$ be a finite extension. The $t$-exact parabolic induction functor $\Ind_B^G:\nD(T)\longrightarrow \nD(G)$ admits a left adjoint 
$$
L(U,-):\nD(G) \lra \nD(T).
$$
Moreover, there is a natural isomorphism $h^0(L(U,\cind_{G_0}^G(1_{G_0})))\cong \cind_{T_0}^T(1_{T_0})$, and the composed morphism of 
$k$-algebras
$$
\End_G\left(\cind_{G_0}^G(1_{G_0})\right) \xrightarrow{L(U,-)} \End_{\nD(T)}\left(L(U,\cind_{G_0}^G(1_{G_0}))\right)  \xrightarrow{h^0(-)} \End_T\left(\cind_{T_0}^T(1_{T_0})\right)
$$
coincides with the mod $p$ Satake morphism of Herzig \cite{herzig:satake}.
\end{Th*}
\end{Pt*}

By applying the functors $h^n(-):\nD(T)\longrightarrow \nD(T)^{\heartsuit}$ for all $n\in\bbZ$ instead of only $n=0$, one can even get a collection of morphisms 
$$
\End_G\left(\cind_{G_0}^G(1_{G_0})\right) \lra \End_T\left(\cind_{T_0}^T\big(h^n(L(U_0,1_{G_0}))\big)\right),
$$
where $L(U_0,-):\nD(G_0) \lra \nD(T_0)$ is the left adjoint of $\Ind_{B_0}^{G_0}$, the parabolic induction at compact level. However, the source of these morphisms is still the classical degree $0$ endomorphism algebra (and the target is still a degree $0$ algebra).

\begin{Pt*}
Set $L^n(U,-) := h^n\circ L(U,-):\nD(G)\longrightarrow \nD(T)^{\heartsuit}$ for any $n\in\bbZ$. Our basic result is the following.  (In the body of the article, we prove more general versions with non-trivial weights, and with larger Levi subgroups; see Subsections \ref{intro:weight} and \ref{intro:Levi} below.)

\begin{Th*}\label{basic}
Assume:
    \begin{itemize}[$\diamond$]
    \item $F$ is unramified over $\bbQ_p$, and
    \item $p > h + 1$, where $h$ denotes the maximum of the Coxeter numbers of the irreducible components of the root system of $\bfG$.
    \end{itemize}
Then the following orthogonality relations hold:  for all $m,n\in \mathbb{Z}$, $m\neq n$, we have
$$
\RHom_T\left(L^m(U,\cind_{G_0}^G(1_{G_0})),~L^n(U,\cind_{G_0}^G(1_{G_0}))\right)=0.
$$
Consequently, $L(U,\cind_{G_0}^G(1_{G_0}))\in \nD(T)$ decomposes as
$$
L\big(U,\cind_{G_0}^G(1_{G_0})\big) \cong \bigoplus_{n=-[F:\bbQ_p]\dim(\bfU)}^0L^n\big(U,\cind_{G_0}^G(1_{G_0})\big)[-n],
$$
and
\begin{flushleft}
$\displaystyle{\RHom_{T}\left(L(U,\cind_{G_0}^G(1_{G_0})),~L(U,\cind_{G_0}^G(1_{G_0}))\right)}$
\end{flushleft} 
\begin{flushright}
$\displaystyle{\cong \bigoplus_{n=-[F:\bbQ_p]\dim(\bfU)}^0\RHom_{T}\left(L^n(U,\cind_{G_0}^G(1_{G_0})),~L^n(U,\cind_{G_0}^G(1_{G_0}))\right).}$
\end{flushright}
\end{Th*}

In particular, under the assumptions of the Theorem, there is a canonical projection 
$$
\RHom_{T}\left(L(U,\cind_{G_0}^G(1_{G_0})),~L(U,\cind_{G_0}^G(1_{G_0}))\right)\lra \RHom_{T}\left(L^n(U,\cind_{G_0}^G(1_{G_0})),~L^n(U,\cind_{G_0}^G(1_{G_0}))\right)
$$
for each $n$. Precomposing with the map 
$$
\RHom_{G}\left(\cind_{G_0}^G(1_{G_0}),~\cind_{G_0}^G(1_{G_0})\right)\lra \RHom_{T}\left(L(U,\cind_{G_0}^G(1_{G_0})),~L(U,\cind_{G_0}^G(1_{G_0}))\right)
$$
induced by the functor $L(U,-):\nD(G)\longrightarrow\nD(T)$, one gets maps
$$
\RHom_{G}\left(\cind_{G_0}^G(1_{G_0}),~\cind_{G_0}^G(1_{G_0})\right)\lra \RHom_{T}\left(L^n(U,\cind_{G_0}^G(1_{G_0})),~L^n(U,\cind_{G_0}^GL(1_{G_0}))\right).
$$
Finally, again one can rewrite $L^n(U,\cind_{G_0}^G(1_{G_0}))$ as $\cind_{T_0}^T(L^n(U_0,1_{G_0}))$, and $L^0(U_0,1_{G_0})=1_{T_0}$. Whence the following definition.

\begin{Def*}\label{basic:def}
Suppose assumptions of Theorem \ref{basic} hold. 

The morphism of graded $k$-algebras
$$
\Ext_{G}^{\bullet}\left(\cind_{G_0}^G(1_{G_0}),~\cind_{G_0}^G(1_{G_0})\right)\lra \Ext_{T}^{\bullet}\left(\cind_{T_0}^T(1_{T_0}),~\cind_{T_0}^T(1_{T_0})\right)
$$
induced by $L(U,-):\nD(G)\longrightarrow\nD(T)$ is called the \emph{graded mod $p$ Satake morphism}.

More generally, for any $n\in [-[F:\bbQ_p]\dim(\bfU),0]$, the morphism of graded $k$-algebras
$$
\Ext_{G}^{\bullet}\left(\cind_{G_0}^G(1_{G_0}),~\cind_{G_0}^G(1_{G_0})\right)\lra \Ext_{T}^{\bullet}\left(\cind_{T_0}^T(L^n(U_0,1_{G_0})),~\cind_{T_0}^T(L^n(U_0,1_{G_0}))\right)
$$
induced by $L(U,-):\nD(G)\longrightarrow\nD(T)$ is called the \emph{$n^{\textrm{th}}$ graded mod $p$ Satake morphism}.
\end{Def*}

In the case $n=0$, a similar morphism has been constructed by Ronchetti \cite{R19b}, and the morphism of the above definition coincides with that of \textit{op. cit.} (see Subsection \ref{link-Ronchetti}). 
\end{Pt*}

\begin{Pt*}
The study of the complex  $L(U,\cind_{G_0}^G(1_{G_0}))\in \nD(T)$ reduces to the compact level $U_0\subset U$, thanks to the quasi-isomorphism $L(U,\cind_{G_0}^G(1_{G_0}))\cong\cind_{T_0}^T(L(U_0,1_{G_0}))$. Then the functor $L(U_0,-):\nD(G_0)\longrightarrow \nD(T_0)$ can be computed using \emph{co}homology of the profinte group $U_0$; more precisely, we have:
$$
L(U_0,-) \cong \textnormal{R}H^0(U_0,-)\big[[F:\bbQ_p]\dim(\bfU)\big] \otimes_k \delta_{B_0}.
$$
The character $\delta_{B_0}:B_0\longrightarrow \bbF_p^\times\subset k^{\times}$ is the \emph{mod $p$ modulus character}, given by $N_{k_F/\bbF_p}(\overline{2\rho})$ where $k_F/\bbF_p$ is the residue field extension of $F/\bbQ_p$ and $\overline{2\rho}$ denotes the mod $p$ reduction of the algebraic character given by the sum of the positive roots relative to $\mathbf{B}$ (see Subsection \ref{formula-spherical}).
%Remark: here this is directly the *continuous cohomology* of $U_0$ which appears, i.e. \emph{a priori} we do not have to \emph{use} the result of Fust saying that the latter coincides with *smooth cohomology* (computed by Ext^i in \nD(U_0)), cf HM; but it is worth to have the latter result in mind, which maybe was implicitly used in the paper...? (note that it does not appear in Heyer's one, however)

The key decomposition property of the complex  $L(U,\cind_{G_0}^G(1_{G_0}))\in \nD(T)$ already holds for the complex  
$\textnormal{R}H^0(U_0,k)\in \nD(T_0)$. To prove the latter, we actually prove that the $T_0$-representation 
$\bigoplus_{n\in\bbZ}H^n(U_0,k)$ is a multiplicity-free sum of characters. By contrast, we note that for a \emph{ramified} finite extension $F/\bbQ_p$, the $T_0$-representation $H^1(U_0,k)$ may be a non-trivial extension of two copies of the same character (see Remark \ref{ramified}). 
\end{Pt*}

\begin{Pt*}
The profinite group $U_0$ is pro-$p$; more precisely, any choice of a total order on the set $\Phi^+$ of positive roots determines a homeomorphism
$$
\prod_{\alpha \in \Phi^+}\bfU_\alpha(\cO_F) \stackrel{\sim}{\longrightarrow} U_0,
$$
induced by the morphisms $u_\alpha: \bfG_a \stackrel{\sim}{\longrightarrow} \bfU_\alpha$. By choosing such an order which is compatible with the height function $\hgt:\Phi^+ \longrightarrow \bbZ_{\geq 0}$, the group $U_0$ can be equipped with a $p$-valuation $\omega:U_0\longrightarrow \bbR_{> 0} \cup \{\infty\}$ in the sense of Lazard \cite{lazard}. Then $(U_0,\omega)$ is $p$-saturated, though in general it is \emph{not equi-$p$-valued}: even in the case $F=\bbQ_p$ (so that $\bfU_\alpha(\cO_F)\cong\bbZ_p$ for each $\alpha$), the basis elements $u_{\alpha}(1)\in U_0$, $\alpha\in\Phi^+$, may have distinct $p$-valuations. In particular, Lazard's general calculation of the mod $p$ cohomology of equi-$p$-valuable groups does not apply, and indeed $H^{\bullet}(U_0,k)$ is different from the exterior algebra of the dual of the mod $p$ Lie algebra associated to $(U_0,\omega)$. 

The mod $p$ cohomology and Lie algebra cohomology of $(U_0,\omega)$ are still related, however, which for general $p$-valued groups has been formalized by Sorensen \cite{sorensen:hochschild} into a spectral sequence. In the case of $(U_0,\omega)$, the sequence is $T_0$-equivariant, and we prove that it degenerates. Further, we determine the $\Res_{k_F/\bbF_p}(\bfT_{k_F})$-representations $H^n(\textnormal{Lie}(\Res_{\cO_F/\bbZ_p}(\bfU))\otimes_{\bbZ_p}k_F,k_F)$, generalizing a strategy of Polo--Tilouine \cite{polotilouine}: at least under the assumption $p > h + 1$, these $k_F$-representations of the algebraic torus $\Res_{\cO_F/\bbZ_p}(\bfT)$ have the same structure as their generic fiber $H^n(\textnormal{Lie}(\Res_{\cO_F/\bbZ_p}(\bfU))\otimes_{\bbZ_p}F,F)$, which was determined by Kostant. At this point, the proof of Theorem \ref{basic} is complete.
\end{Pt*}

\begin{Pt*}\label{intro:weight}
The computation of Polo--Tilouine holds for more general coefficients than the trivial one. Let us review it in the case $F=\bbQ_p$ for simplicity (note, however, that our results below are valid for an unramified extension $F/\bbQ_p$). Let $\lambda$ be any character of $\bfT$ which is \emph{dominant and $p$-small}, i.e. which satisfies
$$\langle \rho, \alpha^\vee \rangle\leq \langle \lambda + \rho, \alpha^\vee \rangle \leq p \quad \textnormal{for all} \quad \alpha \in \Phi^+.$$
Then the irreducible algebraic representation of $\bfG_{\overline{\bbF}_p}$ of highest weight $\lambda$ is defined over $\bbF_p$ and lifts to $\bbZ_p$, and we denote this representation by $L(\lambda)$.  We have
$$
H^n(\textnormal{Lie}(\bfU),L(\lambda)\otimes_{\bbZ_p}R)\cong \bigoplus_{\substack{w\in W\\ \ell(w)=n}}R(w\cdot\lambda)
$$
both for $R=\bbF_p$ and $R=\bbQ_p$, where $W$ is the Weyl group of $(\bfG,\bfT)$ and $\cdot$ is the dot action on the character lattice $X^*(\bfT)$. From this, we deduce a similar description of the $T_0$-representation $H^n(U_0,L(\lambda)_{k})$, namely as the sum of the smooth characters $w\cdot\lambda:T_0\longrightarrow \bbF_p^{\times}\subset k^{\times}$ obtained from the algebraic ones by restriction to $\bbZ_p$-points and reduction mod $p$. In general, while the algebraic characters $w\cdot\lambda$, $w\in W$, are always pairwise distinct, some of the resulting mod $p$ characters of $T_0$ may coincide, and the same character may occur in different degrees of the $T_0$-representation $\bigoplus_{n\in\bbZ}H^n(U_0,L(\lambda)_{k})$. However, this phenomenon occurs only for a few $\lambda$ (which can be detected by the root system when the center of $\bfG$ is connected, see Lemma \ref{split:M=T}).  For all the other $\lambda$, the complex $L(U,\cind_{G_0}^G(L(\lambda)_{k}))\in \nD(T)$ is the orthogonal direct sum of its shifted cohomology objects, as in Theorem \ref{basic} (which is the case $\lambda=0$), giving rise to $\dim(\bfU)+1$ morphisms of graded $k$-algebras as in Definition \ref{basic:def}.
\end{Pt*}

\begin{Pt*}\label{intro:Levi}
Up to here, we have discussed Satake morphisms related to a Borel subgroup $\bfB\subset\bfG$. However, the classical Satake morphism over $\mathbb{C}$ and its mod $p$ variant are naturally defined when starting more generally with a standard \emph{parabolic} subgroup with a Levi decomposition $\bfP=\bfM\bfN$. Correspondingly, in the derived mod $p$ context, Heyer has constructed the left adjoint $L(N,-)$ of the parabolic induction functor $\Ind_P^G: \nD(M)\longrightarrow\nD(G)$. In particular, we can study the complexes $L(N,\cind_{G_0}^G(L(\lambda)_{k}))\in\nD(M)$ (for $k$ equipped with a fixed embedding $k_F\longhookrightarrow k$).

The strategy explained above to compute the $T$-representations $L^n(U,\cind_{G_0}^G(L(\lambda)_{k}))$ still applies to compute the $M$-representations $L^n(N,\cind_{G_0}^G(L(\lambda)_{k}))$. That is, for $\lambda$ dominant and $p$-small, Sorensen's spectral sequence degenerates at the first page, so that the $M_0$-representations $H^n\left(N_0,L(\lambda)_{k}\right)$ can be described as the restriction to $M_0$ of the corresponding $\Res_{\cO_F/\bbZ_p}(\bfM)$-representations on Lie algebra cohomology (at least after semisimplication), which in turn are known from (the generalization of) Polo--Tilouine's computation; see Theorem \ref{groupcoh} for the precise statement. Here again the only assumptions are that $F/\bbQ_p$ is unramified and $p>h+1$.

The question of orthogonality between $L^m(N,\cind_{G_0}^G(L(\lambda)_{k}))$ and $L^n(N,\cind_{G_0}^G(L(\lambda)_{k}))$ for  $m\neq n$ is more delicate, already in the case $\lambda=0$.  In Assumption \ref{assn-centralchar}, we give a criterion for when this orthogonality occurs, in terms of the restriction of the characters $w\cdot \lambda$ to $C_{M,0}$, the integral points of the connected center of $\bfM$.  Under stricter bounds on $p$, we prove in Lemmas \ref{split:bruhat} and \ref{split:abelian} that Assumption \ref{assn-centralchar} holds for $\lambda = 0$ and several classes of Levi subgroups $\bfM$.  In particular, for all these examples, the complex $L(N,\cind_{G_0}^G(1_{G_0}))\in\nD(M)$ splits as the orthogonal direct sum of its shifted cohomology objects, giving rise to $[F:\bbQ_p]\dim(\bfN)+1$ Satake morphisms of graded $k$-algebras
$$
\Ext_{G}^{\bullet}\left(\cind_{G_0}^G(1_{G_0}),~\cind_{G_0}^G(1_{G_0})\right)\lra \Ext_{M}^{\bullet}\left(\cind_{M_0}^M(L^n(N_0,1_{G_0})),~\cind_{M_0}^M(L^n(N_0,1_{G_0}))\right).
$$
\end{Pt*}

\begin{Pt*}
The article is organized as follows. Section \ref{section-coh-unipotent} is devoted to calculating cohomology.  In Subsection \ref{p-val}, we review the elements of Lazard's theory of $p$-valued groups that we will use, and develop how the pro-$p$ group $N_0$ enters into this theory. In Subsection \ref{Alg-and-Lie}, we study in detail the $p$-small weights of the algebraic group 
$\Res_{k_F/\bbF_p}(\bfM_{k_F})$ and the corresponding adjoint representations.  In Subsections \ref{coh-n} and \ref{coh-N_0} we compute the Jordan--Hölder constitutents of the $M_0$-representations $H^n(N_0,L(\lambda)_k)$. Then, in Subsection \ref{split:subsection}, we extract sufficient conditions for the complex 
$
\textnormal{R}H^0(N_0,L(\lambda)_k)\in \nD(M_0)
$
to satisfy the orthogonality relations $\RHom_{M_0}(H^m(N_0,L(\lambda)_k),H^n(N_0,L(\lambda)_k))=0$ for all $m\neq n$; we do this by examining its image under the restriction functor $\nD(M_0)\longrightarrow \nD(C_{M,0})$.

Section \ref{LN_0:section} is devoted to the study of the relation between the two functors $\textnormal{R}H^0(N_0,-)$ and $L(N_0,-)$ from 
$\nD(G_0)$ to $\nD(M_0)$, which in particular gives orthogonal properties for the latter. Note that here we have to work with the profinite group $G_0$ which is $p$-torsion-free but not pro-$p$; as a consequence, we use the fact that  $G_0$ is a \emph{Poincaré group at $p$} in the sense of \cite[Ch. III]{NSW}.

In Section \ref{Satake:section}, we construct the $[F:\bbQ_p]\dim(\bfN)+1$ graded Satake morphisms 
$$
\Ext_{G}^{\bullet}\left(\cind_{G_0}^G(X_0),~\cind_{G_0}^G(X_0)\right)\lra \Ext_{M}^{\bullet}\left(\cind_{M_0}^M(L^n(N_0,X_0)),~\cind_{M_0}^M(L^n(N_0,X_0))\right)
$$
for any $X_0\in\nD(G_0)^{\heartsuit}$ such that $\RHom_M(L^m(N,\cind_{G_0}^G(X_0)),L^n(N,\cind_{G_0}^G(X_0)))$ vanishes for all $m\neq n$. In Remark \ref{link-Ronchetti}, we check that for $\bfM=\bfT$ and $X_0=1_{G_0}$, the graded Satake morphism for $n=0$
% (and $\Ext$-degrees $\leq 1$) 
coincides with the one of \cite{R19b} constructed using universal unramified principal series. (Regarding principal series representations, we also include in Subsection \ref{psr:subsection} the computation of the $G_0$-cohomology with coefficients in the principal series of any mod $p$ smooth character of $T$.)

\end{Pt*}

\begin{Pt*} \textbf{Acknowledgements.}
This project began during a visit of KK to Universit\'e Sorbonne Paris Nord, and he would like to thank Stefano Morra for the invitation.  Both authors would like to thank Claudius Heyer and Stefano Morra for several useful comments.

During the preparation of this article, KK was supported by NSF grant DMS-2101836 and a PSC-CUNY Trad B award, and CP was supported by the ANR project COLOSS ANR-19-CE40-0015.
\end{Pt*}

\subsection{Notation}
We set some notation which will be in force throughout the article.  As above, we let $F$ denote a finite extension of $\bbQ_p$ with ring of integers $\cO_F$, and residue field $k_F$.  Let $\bfG$ denote a split, connected, reductive group over $F$.  Since $\bfG$ is split, we can choose an $\cO_F$-model for $\bfG$ (\cite[Thm. 1.2]{conrad:nonsplitZ}), and we use the same letter to denote this $\cO_F$-model.  

Let 
$$\bfG \supset \bfB \supset \bfU$$
denote a fixed choice of Borel subgroup and its unipotent radical, respectively, and fix a maximal torus $\bfT$ satisfying $\bfB = \bfT \bfU$.  Analogously, let $\bfP = \bfM\bfN$ denote a fixed standard parabolic subgroup of $\bfG$, so that $\bfP$ contains $\bfB$, $\bfM$ contains $\bfT$, and $\bfN$ is the unipotent radical of $\bfP$. We use the same letters to denote the $\cO_F$-models of all these subgroups in the $\cO_F$-group scheme $\bfG$.  

We denote by italicized Roman letters the groups of $F$-points of these groups, so that $G = \bfG(F), B = \bfB(F)$, etc., and by a subscript ``$0$'' the groups of $\cO_F$-points, so that $G_0 = \bfG(\cO_F)$, $B_0 = \bfB(\cO_F)$, etc.  In particular $G_0$ is a hyperspecial compact open subgroup of $G$.

We let $\Phi\subset X^*(\bfT)$ denote the set of roots of $\bfT$ in $\textnormal{Lie}(\bfG)$, and let $\Phi \supset \Phi^+ \supset \Delta$ denote the subsets of positive and simple roots determined by $\bfB$.  We also let $\hgt:\Phi \longrightarrow \bbZ$ denote the height function relative to $\Delta$: if we write $\beta \in \Phi$ as $\beta = \sum_{\alpha \in \Delta} n_\alpha \alpha$, then $\hgt(\beta) = \sum_{\alpha \in \Delta} n_\alpha$.  Given $\alpha \in \Phi$, we let $\bfU_\alpha \subset \bfU$ denote the associated root subgroup, and fix a root isomorphism $u_\alpha:\bfG_{a/\cO_F} \stackrel{\sim}{\longrightarrow} \bfU_\alpha$.  In particular, for $t \in \bfT$ and $x \in \bfG_a$, we have $tu_\alpha(x)t^{-1} = u_\alpha(\alpha(t)x)$.  The parabolic subgroup $\bfP$ corresponds to a subset $J \subset \Delta$, and we let $\Phi_J$ denote the sub-root system generated by $J$, and let $\Phi_J^+ := \Phi_J \cap \Phi^+$.

We fix once and for all a total order on $\Phi$ compatible with the height function $\hgt$.  The morphisms $u_\alpha$ induce an isomorphism of $\cO_F$-schemes (see \cite[Thm. 5.1.16]{conrad:sga3})
\begin{equation}
\label{unip-isom}
\prod_{\alpha \in \Phi^+ - \Phi_J^+}\bfU_\alpha \stackrel{\sim}{\longrightarrow} \bfN,
\end{equation}
the product on the left-hand side being given by the fixed total order.  This isomorphism induces a homeomorphism
\begin{equation}
\label{unip-homeo}
\prod_{\alpha \in \Phi^+ - \Phi^+_J}\bfU_\alpha(\cO_F) \stackrel{\sim}{\longrightarrow} N_0.    
\end{equation}

We let $k$ denote a field of characteristic $p$.  If $H$ is $p$-adic Lie group, we use both the symbols $k$ and $1_H$ to denote the trivial, one-dimensional $H$-representation over $k$.  We let $\textnormal{Rep}(H)$ denote the Grothendieck abelian category of smooth $H$-representations on $k$-vector spaces, and let $\nD(H)$ denote the unbounded derived category of $\textnormal{Rep}(H)$.  The category $\nD(H)$ is a closed symmetric monoidal category (see \cite[Cor. 3.3]{SS22}), such that the functors $-\otimes_k X$ are triangulated for any $X$ in $\nD(H)$.  We let $\underline{\RHom}_H$ denote the internal Hom functor on $\nD(H)$.  Further, $\nD(H)$ has a natural t-structure, and we let $\nD(H)^\heartsuit$ denote its heart.

\section{Mod $p$ cohomology of unipotent subgroups}\label{section-coh-unipotent}

The goal of Section \ref{section-coh-unipotent} is to understand the continuous group cohomology $H^i(N_0,V)$, and the action of $M_0$ on it.  Various versions of this calculation have already been considered in the literature (see \cite[\S 7]{grosseklonne:spherical}, \cite{polotilouine}, \cite{R19a}, \cite[Ch. 2]{kongsgaard:thesis}), but for the sake of completeness we include the arguments below.

\subsection{$p$-valuations and Lie algebras} \label{p-val}
In order to carry out our calculation, we will use Lazard's theory of $p$-valued groups.  We make the following assumptions from this point onwards.

\begin{Assn*}\hfill
    \label{assn1}
    \begin{itemize}[$\diamond$]
    \item $F$ is unramified over $\bbQ_p$, and
    \item $p > h + 1$, where $h$ denotes the maximum of the Coxeter numbers of the irreducible components of the root system of $\bfG$.
    \end{itemize}
\end{Assn*}

%\vspace{5pt}

\begin{Pt*}
Using the homeomorphism \eqref{unip-homeo}, we define a function $\omega': N_0 \longrightarrow \bbR_{> 0} \cup \{\infty\}$ by the formula
\begin{equation}
\label{omega'-def}
\omega'\left(\prod_{\alpha \in \Phi^+ - \Phi_J^+} u_\alpha(x_\alpha)\right) := \min_{\alpha \in \Phi^+ - \Phi_J^+}\left\{\textnormal{val}_p(x_\alpha) + \frac{\hgt(\alpha)}{h(\alpha)}\right\},
\end{equation}
where each $x_\alpha \in \cO_F$, where $\textnormal{val}_p$ denotes the valuation on $F$ which satisfies $\textnormal{val}_p(p) = 1$, and where $h(\alpha)$ denotes the Coxeter number of the irreducible root system to which $\alpha$ belongs.  We note that $\omega'$ is valued in $\frac{1}{h'}\bbZ_{\geq 1} \cup \{\infty\}$, where $h'$ denotes the least common multiple of the $h(\alpha)$.  By \cite[Prop. 3.5]{lahirisorensen}, Assumption \ref{assn1} guarantees that the function $\omega'$ defines a $p$-valuation on $N_0$.  (For details about $p$-valuations and the constructions which follow, see \cite[\S\S II.1, III.2]{lazard} or \cite[\S 23]{schneider:padicliegroups}.)

We define a modified function $\omega:N_0 \longrightarrow \bbR_{> 0} \cup \{\infty\}$ by
\begin{equation}
\label{omega-def}
\omega(n) := \inf_{m \in M_0}\left\{\omega'(mnm^{-1})\right\}.
\end{equation}
Then $\omega$ is also a $p$-valuation on $N_0$ (\cite[\S III.2.1.2]{lazard}), which moreover satisfies $\omega(mnm^{-1}) = \omega(n)$ for all $m \in M_0, n \in N_0$.  

We have the following:
\begin{Lem*}
\label{psatN0}
The $p$-valued group $(N_0, \omega)$ is $p$-saturated (see \cite[Defs. III.2.1.5, III.2.1.6]{lazard}).
\end{Lem*}

\begin{proof}
Let us choose a basis $\{x_i\}_{i = 1}^{[F:\bbQ_p]}$ for $\cO_F$ over $\bbZ_p$, so that the mod $p$ reductions $\{\overline{x_i}\}_{i = 1}^{[F:\bbQ_p]}$ give a basis for $k_F$ over $\bbF_p$.  The homeomorphism \eqref{unip-homeo} implies that the set 
\begin{equation}
\label{N0-ordered-basis}
\left\{u_{\alpha}(x_i)\right\}_{\alpha \in \Phi^+ - \Phi_J^+, 1 \leq i \leq [F:\bbQ_p]}
\end{equation}
gives an ordered basis for $N_0$, in the sense of \cite[Def. III.2.2.4]{lazard} (for the ordering determined by \eqref{unip-homeo}).  We have
$$\frac{1}{p - 1} < \omega(u_{\alpha}(x_i)) \leq \omega'(u_\alpha(x_i)) = \frac{\hgt(\alpha)}{h(\alpha)} < 1 < \frac{p}{p - 1}.$$
By \cite[Prop. III.2.2.7]{lazard}, we conclude that $(N_0,\omega)$ is $p$-saturated.
\end{proof}
\end{Pt*}

\begin{Pt*}
We now consider Lie algebras.  The $p$-valuation $\omega$ induces a filtration on $N_0$, and we let $\textnormal{gr}(N_0)$ denote the associated graded group.  Recall that it is defined as
$$\textnormal{gr}(N_0) := \bigoplus_{\nu \in \bbR_{>0}} \textnormal{gr}_\nu(N_0) := \bigoplus_{\nu \in \bbR_{>0}} N_{0,\nu}/N_{0,\nu+},$$
where 
$$N_{0,\nu} := \{n \in N_0: \omega(n) \geq \nu\}, \qquad N_{0,\nu+} := \{n \in N_0: \omega(n) > \nu\}.$$
Note that each $\nu$ for which $\textnormal{gr}_\nu(N_0) \neq 1$ satisfies $\nu \in \frac{1}{h'}\bbZ_{\geq 1}$.  The graded group $\textnormal{gr}(N_0)$ has the structure of a graded Lie algebra over $\bbF_p$, and moreover is a Lie algebra over the ring $\bbF_p[\pi]$: the Lie bracket is induced by the commutator, and $\pi$ denotes the $\bbF_p$-linear map defined by
\begin{eqnarray*}
\pi: \textnormal{gr}_\nu(N_0) & \longrightarrow & \textnormal{gr}_{\nu + 1}(N_0) \\
nN_{0,\nu+} & \longmapsto & n^pN_{0,(\nu + 1)+}.
\end{eqnarray*}
Finally, we define
$$\fn_\omega := \textnormal{gr}(N_0) \otimes_{\bbF_p[\pi]} \bbF_p.$$
Thus $\fn_\omega$ is a graded Lie algebra over $\bbF_p$ equipped with a conjugation action of $M_0$, which factors through $\bfM(k_F)$.

On the other hand, one can construct $\textnormal{Lie}(\bfN)$, the Lie algebra functor of the $\cO_F$-group scheme $\bfN$; we set
$$\fn := \textnormal{Lie}(\bfN)(k_F).$$
Thus $\fn$ is a Lie algebra over $k_F$ equipped with a conjugation action of $\bfM(k_F)$.  
\end{Pt*}

The comparison between the two Lie algebras is given by the following.

\begin{Lem*}
    \label{liealgcompn}
    There exists an isomorphism of $\bbF_p$-Lie algebras
    $$\fn_\omega \cong \fn,$$
    after forgetting about the grading on $\fn_\omega$ and the $k_F$-linear structure on $\fn$.  In particular, this implies that the $\bbF_p$-vector space 
    $\fn_\omega$ has a $k_F$-linear structure, and that the Lie bracket on $\fn_\omega$ is $k_F$-bilinear.  Moreover, this isomorphism is $M_0$-equivariant.
\end{Lem*}

\begin{proof}
    We prove that $\textnormal{gr}(N_0) \cong \fn\otimes_{\bbF_p} \bbF_p[\pi]$ as Lie algebras over $\bbF_p[\pi]$.  Applying the base change $-\otimes_{\bbF_p[\pi]} \bbF_p$ (which maps $\pi \longmapsto 0$) will then give the result.
    
    Let us choose a basis $\{x_i\}_{i = 1}^{[F:\bbQ_p]}$ for $\cO_F$ over $\bbZ_p$ as in \ref{psatN0}.
    %so that the mod $p$ reductions $\{\overline{x_i}\}_{i = 1}^{[F:\bbQ_p]}$  give a basis for $k_F$ over $\bbF_p$. 
    The set 
    $$\left\{\overline{u_\alpha(p^jx_i)}\right\}_{\alpha \in \Phi^+ - \Phi^+_J,1 \leq i \leq [F:\bbQ_p], j \geq 0}$$
    is then a $\bbF_p$-basis for $\textnormal{gr}(N_0)$, and 
    \begin{equation}\label{n-omega-basis}
    \left\{\overline{u_\alpha(x_i)}\right\}_{\alpha \in \Phi^+ - \Phi^+_J,1 \leq i \leq [F:\bbQ_p]}
    \end{equation}
    is an $\bbF_p[\pi]$ basis for $\textnormal{gr}(N_0)$.  (For $n\in N_{0,\nu}$, we denote by $\overline{n}$ the image of $n$ in $\textnormal{gr}_\nu(N_0) \subset \textnormal{gr}(N_0)$.)

    On the other hand, the Lie algebra functor $\textnormal{Lie}(\bfN)$ on $\cO_F$-algebras $R$ is defined by
    $$\textnormal{Lie}(\bfN)(R) := \ker\Big(\bfN(R[\varepsilon]/(\varepsilon^2)) \longrightarrow \bfN(R)\Big).$$
    In particular
    $$\fn = \ker\Big(\bfN(k_F[\varepsilon]/(\varepsilon^2)) \longrightarrow \bfN(k_F)\Big).$$
    Thus, the map 
    \begin{eqnarray}
        \ff: \fn\otimes_{\bbF_p}\bbF_p[\pi] & \stackrel{\sim}{\longrightarrow} &\textnormal{gr}(N_0) \label{def-of-ff}\\
        u_\alpha(\overline{x_i}\varepsilon)\otimes f(\pi) & \longmapsto & f(\pi)\cdot \overline{u_\alpha(x_i)} \notag
    \end{eqnarray}
    gives an $\bbF_p[\pi]$-linear bijection between the two Lie algebras, which is moreover equivariant for the conjugation action of $M_0$.  It only remains to verify that this isomorphism preserves the Lie bracket.

    By \cite[Prop. 5.1.14]{conrad:sga3}, given distinct roots $\alpha,\beta \in \Phi^+ - \Phi_J^+$, the root morphisms $u_\alpha, u_\beta$ satisfy 
    \begin{equation}
    \label{commutator}
    u_\alpha(x)u_\beta(y)u_\alpha(x)^{-1}u_\beta(y)^{-1} = \prod_{\substack{i,j> 0 \\ i\alpha + j\beta \in \Phi^+ - \Phi_J^+}}u_{i\alpha + j\beta}(c_{\alpha,\beta; i,j}x^iy^j),
    \end{equation}
    for some fixed $c_{\alpha,\beta;i,j} \in \bbZ_p$.  In particular, if we let $\varepsilon'$ denote another indeterminate satisfying $(\varepsilon')^2 = 0$, we have (assuming $\alpha + \beta \in \Phi^+$)
    $$u_\alpha(x\varepsilon)u_\beta(y\varepsilon')u_\alpha(x\varepsilon)^{-1}u_\beta(y\varepsilon')^{-1} = u_{\alpha + \beta}(c_{\alpha,\beta; 1,1}xy\varepsilon\varepsilon').$$
    Hence, by \cite[Lem. A.7.4, Prop. A.7.5]{CGP}, the Lie bracket on $\fn$ is given by
    \begin{equation}
    \label{comm1}
    [u_\alpha(x\varepsilon), u_\beta(y\varepsilon)] = u_{\alpha + \beta}(c_{\alpha,\beta;1,1}xy\varepsilon).
    \end{equation}
    On the other hand, by \cite[\S II.1.1.7]{lazard} the Lie bracket on $\textnormal{gr}(N_0)$ is induced by the commutator \eqref{commutator}.  Thus, in the graded group $\textnormal{gr}(N_0)$, we have
    \begin{equation}
    \label{comm2}
    [\overline{u_\alpha(x)},\overline{u_\beta(y)}] = \overline{u_\alpha(x)u_\beta(y)u_\alpha(x)^{-1}u_\beta(y)^{-1}} = \overline{u_{\alpha + \beta}(c_{\alpha,\beta; 1,1}xy)},    
    \end{equation}
    since any other root $i\alpha + j\beta$ appearing in \eqref{commutator} will be of strictly greater height (and thus $u_{i\alpha + j\beta}(c_{\alpha,\beta;i,j}x^iy^j)$ will have strictly greater $p$-valuation).  
      %Sketch: since $u_\alpha$ and $u_\beta$ commute if $\alpha$ and $\beta$ lie in different irreducible components of $\Phi$, it suffices to assume $\Phi$ is irreducible.  Consider the space $(\mathbb{R}\alpha + \mathbb{R}\beta) \cap \Phi$.  If this root system is of type $A_1 \times A_1$ or $A_2$, then there is at most one term appearing in the right-hand side of \eqref{commutator}, and the claim follows.  If the root system is of type $G_2$, then $\Phi$ itself must be of type $G_2$, and an analysis of all cases verifies the claim.  Finally, if the root system is of type $B_2$, then the only case where more than one term appears in \eqref{commutator} is when $\alpha$ and $\beta$ are simple roots.  Assume without loss of generality that $\alpha$ is short and $\beta$ is long.  Then the terms appearing in \eqref{commutator} are $u_{\alpha + \beta}$ and $u_{2\alpha + \beta}$.  Pick $w\in W_M$ which minimizes the height of $w(2\alpha + \beta)$.  Then $w(\alpha)$ and $w(\beta)$ are still positive, so $\hgt(w(\alpha + \beta)) < \hgt(w(2\alpha + \beta))$.  Therefore $\omega(u_{\alpha + \beta}) \leq \omega'(u_{w(\alpha + \beta)}) < \omega'(u_{w(2\alpha + \beta)}) = \omega(u_{2\alpha + \beta})$, and the result follows.
    Comparing \eqref{comm1} and \eqref{comm2}, we see that the isomorphism $\ff$ preserves the Lie bracket.
\end{proof}

\begin{Pt*}
We recall one more construction.  Let $\cO_F\llbracket N_0 \rrbracket$ denote the completed group algebra of $N_0$ over $\cO_F$, and let $w: \cO_F\llbracket N_0 \rrbracket \longrightarrow \bbR_{>0} \cup \{\infty\}$ be the valuation associated to $\omega$.  (For the definition of the latter, see \cite[Def. III.2.3.1.2]{lazard}; note that the ``filtration'' $w$ on $\cO_F[N_0]$ defined in \emph{loc. cit.} is in fact a valuation (as defined in \cite[Def. I.2.2.1]{lazard}), and $\cO_F\llbracket N_0 \rrbracket$ is the completion of $\cO_F[N_0]$ for $w$ \cite[Thm. III.2.3.3, Cor. III.2.3.4]{lazard}.  For an explicit description of $w$, see \cite[\S III.2.3.8]{lazard}.  We note also that the constructions of \emph{loc. cit.} carry over readily from the case of $\bbZ_p$-coefficients to the case of $\cO_F$-coefficients.)  We let $\textnormal{gr}(\cO_F\llbracket N_0 \rrbracket)$ denote the associated graded ring. Recall that it is defined as 
$$\textnormal{gr}(\cO_F\llbracket N_0 \rrbracket) := \bigoplus_{\nu \in \bbR_{\geq 0}} \textnormal{gr}_\nu(\cO_F\llbracket N_0 \rrbracket) := \bigoplus_{\nu \in \bbR_{\geq 0}}\cO_F\llbracket N_0 \rrbracket_{\nu}/\cO_F\llbracket N_0 \rrbracket_{\nu+},$$
where 
$$\cO_F\llbracket N_0 \rrbracket_\nu := \{\xi \in \cO_F\llbracket N_0 \rrbracket: w(\xi) \geq \nu\},\qquad \cO_F\llbracket N_0 \rrbracket_{\nu+} := \{\xi \in \cO_F\llbracket N_0 \rrbracket: w(\xi) > \nu\}.$$
The graded ring $\textnormal{gr}(\cO_F\llbracket N_0 \rrbracket)$ has the structure of a graded algebra over $k_F[\pi]$, where $\pi$ acts by 
\begin{eqnarray*}
\pi: \textnormal{gr}_\nu(\cO_F\llbracket N_0 \rrbracket) & \longrightarrow & \textnormal{gr}_{\nu + 1}(\cO_F\llbracket N_0 \rrbracket) \\
\xi + \cO_F\llbracket N_0 \rrbracket_{\nu + } & \longmapsto & p\xi + \cO_F\llbracket N_0 \rrbracket_{(\nu + 1)+}.
\end{eqnarray*}
The maps (defined in \cite[II.1.1.9]{lazard})
\begin{eqnarray}
\textnormal{gr}_\nu(N_0) & \longrightarrow & \textnormal{gr}_\nu(\cO_F\llbracket N_0 \rrbracket)  \label{lie-alg-to-gr}\\
nN_{0,\nu+} & \longmapsto & (n - 1) + \cO_F\llbracket N_0 \rrbracket_{\nu+} \notag
\end{eqnarray}
assemble to a map $\gr(N_0) \longrightarrow \gr(\cO_F\llbracket N_0 \rrbracket)$.  
\end{Pt*}

We have the following result on the structure of $\gr(\cO_F\llbracket N_0 \rrbracket)$.

\begin{Th*} \emph{\textbf{\cite[Thm. III.2.3.3]{lazard}}}
The maps \eqref{lie-alg-to-gr} induce an isomorphism of graded $k_F[\pi]$-modules
$$U_{\bbF_p[\pi]}(\gr(N_0)) \otimes_{\bbF_p[\pi]}  k_F[\pi]\stackrel{\sim}{\longrightarrow} \gr(\cO_F\llbracket N_0 \rrbracket),$$
where $U_{\bbF_p[\pi]}$ denotes the universal enveloping algebra over $\bbF_p[\pi]$.
\end{Th*}

\begin{Pt*}
Finally let $k_F\llbracket N_0 \rrbracket$ denote the completed group algebra of $N_0$ over $k_F$, and note that $\cO_F\llbracket N_0 \rrbracket \otimes_{\cO_F} k_F \cong k_F\llbracket N_0 \rrbracket$.  We define $k_F\llbracket N_0 \rrbracket_{\nu}$ (resp., $k_F\llbracket N_0 \rrbracket_{\nu+}$) as the image of $\cO_F\llbracket N_0 \rrbracket_{\nu}$ (resp., $\cO_F\llbracket N_0 \rrbracket_{\nu+}$) in $k_F\llbracket N_0 \rrbracket$. 
\end{Pt*}

Applying the base change $-\otimes_{k_F[\pi]}k_F$ (mapping $\pi \longmapsto 0$) to the previous result gives the following.

\begin{Cor*}\label{lazardcor}
We have an isomorphism of graded $k_F$-algebras
$$U_{k_F}(\fn_\omega\otimes_{\bbF_p}k_F) = U_{\bbF_p[\pi]}(\gr(N_0))\otimes_{\bbF_p[\pi]}k_F \stackrel{\sim}{\longrightarrow} \gr(k_F\llbracket N_0 \rrbracket).$$
\end{Cor*}

\subsection{Algebraic representations and Lie algebra actions}\label{Alg-and-Lie}

We now discuss the actions of the Lie algebras $\fn_\omega$ and $\fn$ on various algebraic representations.  
%In particular Assumption \ref{assn1} remains in effect throughout this subsection.  

\begin{Pt*}
To allow for more flexibility, we put ourselves in the following setting.  We let $\underline{\bfG}$ denote the Weil restriction from $k_F$ to $\bbF_p$ of the special fiber of the $\cO_F$-group scheme $\bfG$:
$$\underline{\bfG} := \textnormal{Res}_{k_F/\bbF_p}(\bfG_{k_F}).$$
We use the analogous notation $\underline{\bfB}, \underline{\bfT}$, etc., for the other group schemes introduced above.  We have
\begin{equation}
\label{resGdecomp}
\underline{\bfG}_{k_F} := \underline{\bfG}\times_{\bbF_p} k_F \cong \prod_{\varsigma:k_F \rightarrow k_F} \bfG_{k_F}\times_{k_F,\varsigma}k_F,
\end{equation}
where $\varsigma$ runs over all the field homomorphisms $k_F \longrightarrow k_F$ over $\bbF_p$.  We have similar decompositions and notations for the other underlined groups, as well as an isomorphism of character groups
\begin{equation}
\label{resTdecomp}
X^*(\underline{\bfT}_{k_F}) \cong \bigoplus_{\varsigma:k_F \rightarrow k_F}X^*(\bfT_{k_F} \times_{k_F, \varsigma} k_F).
\end{equation}
To make notation lighter, we set $X^*(\underline{\bfT}):=X^*(\underline{\bfT}_{k_F})$\footnote{In this paper, we will never consider 
$X^*(\underline{\bfT})$ as a $\Gal(k_F/\bbF_p)$-module, i.e., we only regard it as an abstract abelian group.}.
Accordingly, we denote the set of roots of $\underline{\bfT}_{k_F}$ in $\textnormal{Lie}(\underline{\bfG}_{k_F})$ by $\underline{\Phi}$; by \eqref{resTdecomp} and the decomposition
$$
\textnormal{Lie}(\underline{\bfG}_{k_F})(k_F) \cong \bigoplus_{\varsigma:k_F \rightarrow k_F} \textnormal{Lie}(\bfG)(k_F) \otimes_{k_F,\varsigma} k_F,
$$
we can identify $\underline{\Phi}$ with $[k_F:\bbF_p]$ copies of $\Phi$ indexed by the field homomorphisms $\varsigma: k_F \longrightarrow k_F$ over $\bbF_p$.  We use analogous notation $\underline{\Phi}^+, \underline{\Delta}$, etc., for the other root-theoretic data.  In particular, the subset $J \subset \Delta$ (corresponding to $\bfP = \bfM\bfN$) gives a subset $\underline{J} \subset \underline{\Delta}$ (corresponding to $\underline{\bfP} = \underline{\bfM}\underline{\bfN}$).  Finally, we let $W = W(\underline{\bfG},\underline{\bfT})$ denote the Weyl group of $\underline{\bfG}_{k_F}$ relative to $\underline{\bfT}_{k_F}$.  
\end{Pt*}

\begin{Pt*}
We let 
$$X^*(\underline{\bfT})_+ := \{\lambda \in X^*(\underline{\bfT}): \langle \lambda, \alpha^\vee \rangle \geq 0 \quad \textnormal{for all}\quad \alpha \in \underline{\Phi}^+\}$$
denote the set of dominant characters of $\underline{\bfT}$ (with respect to $\underline{\Phi}^+$).  For future reference, we also define
$$\rho := \frac{1}{2}\sum_{\alpha \in \underline{\Phi}^+} \alpha \in X^*(\underline{\bfT}) \otimes_{\bbZ}\bbQ,$$
and recall that $2\rho \in \bbZ\underline{\Phi} \subset X^*(\underline{\bfT})$, and that $\langle \rho, \alpha^\vee \rangle = 1$ for all $\alpha \in \underline{\Delta}$.

Given $\lambda \in X^*(\underline{\bfT})_+$, we let $L(\lambda)$ denote the irreducible representation of the split connected reductive group 
$\underline{\bfG}_{k_F}$ of highest weight $\lambda$ (see \cite[Part II, \S 2.4]{jantzen}).  We will often also use the notation $L(\lambda)$ to denote the $k_F$-points of this representation, which we view as a representation of the finite group $\bfG(k_F)$ over $k_F$ via the following sequence of morphisms:
$$\bfG(k_F) \cong \underline{\bfG}(\bbF_p) \longhookrightarrow  \underline{\bfG}(k_F) \curvearrowright  L(\lambda).$$
Then $L(\lambda)$ defines a representation of $G_0=\bfG(\cO_F)$ via inflation.  

Given a subset $J \subset \Delta$ (corresponding to $\bfM$) and a character $\lambda \in X^*(\underline{\bfT})$ which is dominant with respect to $\underline{\Phi}_J^+$, we use the notation $L_J(\lambda)$ to denote the irreducible representation of $\underline{\bfM}_{k_F}$ of highest weight 
$\lambda$.  
\end{Pt*}

\begin{Def*}
We say $\lambda \in X^*(\underline{\bfT})_+$ is \emph{$p$-small} (with respect to $\underline{\Phi}^+$) if 
$$\langle \lambda + \rho, \alpha^\vee \rangle \leq p \quad \textnormal{for all} \quad \alpha \in \underline{\Phi}^+.$$
\end{Def*}

We shall need several representation-theoretic results.  We begin with some combinatorial lemmas.

\begin{Lem*}
\label{p-small-J}
Let $\lambda \in X^*(\underline{\bfT})_+$ be $p$-small.  Fix $J \subset \Delta$, and assume $\mu\in X^*(\underline{\bfT})$ is a weight of $L(\lambda)$ which is dominant with respect to $\underline{\Phi}_J^+$.  Then $\mu$ is $p$-small with respect to $\underline{\Phi}_J^+$.
\end{Lem*}

\begin{proof}
Define
$$C' := \{\eta \in X^*(\underline{\bfT})\otimes_{\bbZ} \bbR: \langle\eta + \rho, \alpha^\vee\rangle \leq p \quad \textnormal{for all} \quad \alpha \in \underline{\Phi} \}.$$
The set $C'$ is convex (being the intersection of half-spaces), and contains $\lambda$ by assumption.  

\begin{Claim}
The set $C'$ contains the $W$-orbit of $\lambda$.
\end{Claim}

\begin{proof}[Proof of claim]
We note first that the set $C'$ may equivalently be described as
\begin{equation}
\label{C'positive}
C' = \{\eta \in X^*(\underline{\bfT})\otimes_{\bbZ} \bbR: -p - \hgt(\alpha^\vee) \leq \langle\eta, \alpha^\vee\rangle \leq p - \hgt(\alpha^\vee) \quad \textnormal{for all} \quad \alpha \in \underline{\Phi}^+ \}, 
\end{equation}
where, by abuse of notation, we let $\hgt:\underline{\Phi}^{\vee} \longrightarrow \bbZ$ denote the height function relative to $\underline{\Delta}^\vee$.

Now fix $w \in W$ and $\alpha \in \underline{\Phi}^{+}$, and suppose first that $w^{-1}(\alpha^{\vee}) \in \underline{\Phi}^{-,\vee}$.  This implies $\hgt(\alpha^\vee) - \hgt(w^{-1}(\alpha^\vee)) > 0$, so that 
\begin{eqnarray*}
-p - \hgt(\alpha^\vee) & < & -p - \hgt(\alpha^\vee) + \left(\hgt(\alpha^\vee) - \hgt(w^{-1}(\alpha^\vee))\right) \\
 & = &  -p - \hgt(w^{-1}(\alpha^\vee)) \\
 & \leq & \langle \lambda, w^{-1}(\alpha^\vee)\rangle \\
 & \leq & 0,
\end{eqnarray*}
where the two non-strict inequalities follow from the fact that $\lambda$ is dominant and $p$-small.  Since $\langle \lambda, w^{-1}(\alpha^\vee)\rangle = \langle w(\lambda), \alpha^\vee\rangle$, we see that $w(\lambda)$ satisfies the inequalities \eqref{C'positive} when $w^{-1}(\alpha^\vee) \in \underline{\Phi}^{-,\vee}$ (recall that $p>h+1$).

On the other hand, suppose $w^{-1}(\alpha^\vee) \in \underline{\Phi}^{+,\vee}$, so that $0 \leq \langle \lambda, w^{-1}(\alpha^\vee)\rangle = \langle w(\lambda), \alpha^\vee\rangle$ by dominance of $\lambda$.  To verify the other inequality, assume by contradiction that $\langle w(\lambda), \alpha^\vee\rangle > p - \hgt(\alpha^\vee)$.  Using that $\lambda$ is $p$-small, we get
$$p - \hgt(\alpha^\vee) < \langle w(\lambda), \alpha^\vee\rangle = \langle \lambda, w^{-1}(\alpha^\vee)\rangle \leq p - \hgt(w^{-1}(\alpha^\vee))$$
which implies $\hgt(w^{-1}(\alpha^\vee)) < \hgt(\alpha^\vee)$.  In particular $w^{-1}(\alpha^\vee)$ is not equal to $\alpha_0^\vee$, the highest coroot of the irreducible component of $\underline{\Phi}^\vee$ to which $\alpha^\vee$ belongs.  Thus, there exists a sequence of (not necessarily distinct) simple coroots $\alpha_1^\vee, \ldots, \alpha_n^\vee \in \underline{\Delta}^\vee$ (with $n \geq 1$) such that 
\begin{itemize}[$\diamond$]
\item $w^{-1}(\alpha^\vee) + \sum_{i = 1}^j \alpha_i^\vee \in \underline{\Phi}^{+,\vee}$ for each $j = 1, \ldots, n$, and
\item $w^{-1}(\alpha^\vee) + \sum_{i = 1}^n \alpha_i^\vee$ is equal to $\alpha^\vee_0$.
\end{itemize}
Again using dominance of $\lambda$, we have that $\sum_{i = 1}^n \langle \lambda, \alpha_i^\vee\rangle \geq 0$, which implies
$$p - \hgt(\alpha^\vee) < \langle \lambda, w^{-1}(\alpha^\vee)\rangle + \sum_{i = 1}^n \langle \lambda, \alpha_i^\vee\rangle = \langle \lambda, \alpha_0^\vee \rangle \leq p - \hgt(\alpha_0^\vee)$$
where in the last inequality we have used $p$-smallness of $\lambda$.  The above inequality then gives $\hgt(\alpha_0^\vee) < \hgt(\alpha^\vee)$, contradicting the maximality of $\alpha_0^\vee$.  Thus $w(\lambda)$ also satisfies the inequalities \eqref{C'positive} when $w^{-1}(\alpha^\vee) \in \underline{\Phi}^{+,\vee}$, and the claim is verified.
\end{proof}

We now continue with the proof of the lemma.  Since $\mu$ is a weight of $L(\lambda)$, it is contained in the convex hull of the set $\{w(\lambda)\}_{w\in W}$.  By convexity of $C'$ and the above claim, we get that $\mu \in C'$.  For $\alpha \in \underline{J}$, we have $\langle \rho, \alpha^\vee \rangle = 1 = \langle \rho_M, \alpha^\vee\rangle$ (where $\rho_M := \frac{1}{2}\sum_{\alpha \in \underline{\Phi}_J^+}\alpha$), which implies 
$$\langle \mu + \rho_M, \alpha^\vee \rangle =  \langle \mu + \rho, \alpha^\vee\rangle \leq p$$
for all $\alpha \in \underline{\Phi}_J^+$.  Thus $\mu$ is $p$-small with respect to $\underline{\Phi}_J^+$.  
\end{proof}

\begin{Cor*}
\label{algrepscor}
Let $\lambda \in X^*(\underline{\bfT})_+$ be $p$-small.  Fix $J \subset \Delta$, and let $\bfM$ denote the associated Levi subgroup of 
$\bfG$.
\begin{enumerate}[(a)]
\item\label{algrepscor-a} The $\underline{\bfM}_{k_F}$-representation $L(\lambda)|_{\underline{\bfM}_{k_F}}$ is semisimple.
\item\label{algrepscor-b} The $\bfG(k_F)$-representation $L(\lambda)$ is (absolutely) irreducible.
\item\label{algrepscor-c} The $\bfM(k_F)$-representation $L(\lambda)|_{\bfM(k_F)}$ is semisimple.
\end{enumerate}
\end{Cor*}

\begin{proof}
\begin{enumerate}[(a)]
\item Suppose $L_J(\mu)$ and $L_J(\mu')$ are two irreducible subquotients of $L(\lambda)|_{\underline{\bfM}_{k_F}}$, where $\mu,\mu'\in X^*(\underline{\bfT})$ are dominant with respect to $\underline{\Phi}^+_J$.  Since both $\mu$ and $\mu'$ are also weights of $L(\lambda)$, Lemma \ref{p-small-J} implies that they are both $p$-small relative to $\underline{\Phi}^+_J$.  Therefore, we get
$$\Ext^1_{\underline{\bfM}_{k_F}}\left(L_J(\mu), L_J(\mu')\right) = 0;$$
if $\mu' \neq \mu$, this follows from the Linkage Principle (\cite[Ch. II, Cor. 6.17]{jantzen}), while if $\mu' = \mu$ it follows from \cite[Ch. II, \S 2.12, Eqn. (1)]{jantzen}.  Thus $L(\lambda)|_{\underline{\bfM}_{k_F}}$ is semisimple as a representation of the algebraic group $\underline{\bfM}_{k_F}$. 
 
\item When the derived subgroup of $\underline{\bfG}$ is simply connected, \cite[Lem. 9.2.4]{GHS} implies that $L(\lambda)\otimes_{k_F}\overline{k_F}$ is irreducible as a representation of $\underline{\bfG}(\bbF_p) \cong \bfG(k_F)$, and hence the same is true of the representation $L(\lambda)$.  (Note that the conditions in \cite[Hyp. 9.1.1]{GHS} that the center of $\underline{\bfG}$ is connected and that $\underline{\bfG}$ possesses a twisting element is not used in the proof of the cited result.)  In general, choose a $z$-extension 
$$1 \longrightarrow \widetilde{\bfZ} \longrightarrow \widetilde{\bfG} \longrightarrow \bfG_{k_F} \longrightarrow 1,$$
where $\widetilde{\bfG}$ is a split connected reductive group over $k_F$ with simply connected derived subgroup, and $\widetilde{\bfZ}$ is a split central torus (cf. \cite[p. 387]{K86}).  
%Note that $\bfG_{k_F}$ is split over $k_F$ by definition, so that the construction of loc. cit. indeed gives split groups
Restricting scalars from $k_F$ to $\bbF_p$, we obtain an exact sequence 
\begin{equation}
\label{z-ext-res'd}
1 \longrightarrow \underline{\widetilde{\bfZ}} \longrightarrow \underline{\widetilde{\bfG}} \longrightarrow \underline{\bfG}\longrightarrow 1,
\end{equation}
where $\underline{\widetilde{\bfG}} := \textnormal{Res}_{k_F/\bbF_p}(\widetilde{\bfG})$ and $\underline{\widetilde{\bfZ}} := \textnormal{Res}_{k_F/\bbF_p}(\widetilde{\bfZ})$. We may then view $L(\lambda)$ as an irreducible algebraic representation of $\underline{\widetilde{\bfG}}_{k_F}$ on which 
$\underline{\widetilde{\bfZ}}_{k_F}$ acts trivially.  More precisely, 
%the surjection $\underline{\widetilde{\bfG}}_{k_F} \longrightarrow \underline{\bfG}_{k_F}$ induces a bijection on roots, and therefore, 
as a representation of $\underline{\widetilde{\bfG}}_{k_F}$, the representation $L(\lambda)$ has the form $L(\widetilde{\lambda})$ for the  dominant and $p$-small character $\widetilde{\lambda}$ inflated from $\lambda$. Noting that the derived subgroup of $\underline{\widetilde{\bfG}}$ is simply connected, 
%since after base change to $\overline{\bbF}_p$, the group $\underline{\widetilde{\bfG}}$ becomes isomorphic to a product of copies of 
%$\widetilde{\bfG}$
\cite[Lem. 9.2.4]{GHS} now implies that $L(\widetilde{\lambda})$ is absolutely irreducible as a representation of $\underline{\widetilde{\bfG}}(\bbF_p)$.  Furthermore, by Shapiro's Lemma and Hilbert's Theorem 90, we have $H^1(\bbF_p, \underline{\widetilde{\bfZ}}) = H^1(k_F, \widetilde{\bfZ}) = 0$, and therefore taking Galois invariants of the exact sequence \eqref{z-ext-res'd} gives
$$1 \longrightarrow \underline{\widetilde{\bfZ}}(\bbF_p) \longrightarrow \underline{\widetilde{\bfG}}(\bbF_p) \longrightarrow \underline{\bfG}(\bbF_p) \longrightarrow 1.$$
Since $\underline{\widetilde{\bfZ}}(\bbF_p)$ acts trivially on $L(\widetilde{\lambda})$, we see that the action of $\underline{\widetilde{\bfG}}(\bbF_p)$ descends to an absolutely irreducible action of $\underline{\bfG}(\bbF_p) = \bfG(k_F)$ on $L(\lambda)$.

\item By part (a), the algebraic $\underline{\bfM}_{k_F}$-representation $L(\lambda)|_{\underline{\bfM}_{k_F}}$ takes the form
$$L(\lambda)|_{\underline{\bfM}_{k_F}} = \bigoplus_{\mu \in \mathcal{S}} L_J(\mu)$$
where $\mathcal{S}$ is some finite set of weights (possibly with multiplicity) which are dominant and $p$-small with respect to $\underline{\Phi}_J^+$.  Restricting further to $\bfM(k_F)=\underline{\bfM}(\bbF_p)\subset \underline{\bfM}(k_F)$, we obtain
$$L(\lambda)|_{\bfM(k_F)} = \bigoplus_{\mu \in \mathcal{S}} L_J(\mu)|_{\bfM(k_F)}.$$ 
Applying part (b) to $\bfM$ in place of $\bfG$, we see that each $L_J(\mu)|_{\bfM(k_F)}$ is an irreducible $\bfM(k_F)$-representation, which gives the claim.
\end{enumerate}
\end{proof}

\begin{Lem*}
\label{wt-string-lemma}
Suppose $\lambda \in X^*(\underline{\bfT})_+$ is $p$-small.  Then any weight string in $L(\lambda)$ has length at most $p$.
\end{Lem*}

\begin{proof}
Any weight string in $L(\lambda)$ has the form
$$\mu,~ \mu - \alpha,~ \mu - 2\alpha,~ \ldots,~ \mu - r\alpha,$$
where $\mu \in X^*(\underline{\bfT})$, $\alpha \in \underline{\Phi}$, and $r \geq 0$.  Furthermore, each of the characters appearing above are weights of $L(\lambda)$, while $\mu + \alpha$ and $\mu - (r + 1)\alpha$ are \emph{not} weights of $L(\lambda)$, and the integer $r$ is given by $\langle \mu, \alpha^\vee\rangle$ (see \cite[\S 21.3]{humphreys:GTM}).  Acting by the Weyl group $W$, we may moreover assume $\mu \in X^*(\underline{\bfT})_+$.

If $r = \langle \mu,\alpha^\vee\rangle = 0$, then the weight string has length 1 and the lemma holds.  Therefore, we may assume $r = \langle \mu, \alpha^\vee\rangle > 0$.  In particular, this implies that $\alpha \in \underline{\Phi}^+$.  Let $\alpha_0^\vee$ denote the highest coroot of the irreducible component of $\underline{\Phi}^\vee$ to which $\alpha^\vee$ belongs.  As in the proof of Lemma \ref{p-small-J}, we let $\alpha_1^\vee, \ldots, \alpha_n^\vee \in \underline{\Delta}^\vee$ denote a sequence of simple coroots such that 
\begin{itemize}[$\diamond$]
\item $\alpha^\vee + \sum_{i = 1}^j \alpha_i^\vee \in \underline{\Phi}^{+,\vee}$ for each $j = 1, \ldots, n$, and
\item $\alpha^\vee + \sum_{i = 1}^n \alpha_i^\vee$ is equal to $\alpha^\vee_0$
\end{itemize}
(we allow $\alpha^\vee = \alpha_0^\vee$ here).  The dominance of $\mu$ now gives $\langle \mu, \alpha_i^\vee \rangle \geq 0$, which implies
\begin{equation}
\label{wt-string-1}
\langle \mu, \alpha_0^\vee \rangle = \langle \mu, \alpha^\vee \rangle + \sum_{i = 1}^n\langle \mu, \alpha_i^\vee \rangle \geq \langle \mu, \alpha^\vee \rangle.
\end{equation}

Next, by \cite[Ch. VI, \S 1.8, Prop. 25(ii)]{bourbaki:Lie4-6}, the coroot $\alpha_0^\vee$ lies in the closure of the Weyl chamber of $X_*(\underline{\bfT})\otimes_{\bbZ}\bbR$ defined by $\underline{\Delta}$.  Therefore, we have $\langle \beta, \alpha_0^\vee\rangle \geq 0$ for all $\beta \in \underline{\Delta}$.  Since $\mu$ is a weight of $L(\lambda)$, it lies below $\lambda$ in the partial order defined by $\underline{\Delta}$, so we may write $\lambda = \mu + \sum_{\beta \in \underline{\Delta}} n_\beta \beta$, where $n_\beta \in \bbZ_{\geq 0}$.  Combining these two facts shows that
\begin{equation}
\label{wt-string-2}
\langle \lambda, \alpha_0^\vee \rangle = \langle \mu, \alpha_0^\vee \rangle + \sum_{\beta \in \underline{\Delta}} n_\beta \langle \beta, \alpha_0^\vee\rangle \geq \langle \mu, \alpha_0^\vee \rangle.
\end{equation}

We now conclude.  The length of the weight string in the first paragraph above is given by $r + 1 = \langle \mu, \alpha^\vee \rangle + 1$, and by $p$-smallness of $\lambda$ we obtain
$$\langle \mu, \alpha^\vee \rangle + 1\stackrel{\eqref{wt-string-1}}{\leq} \langle \mu, \alpha_0^\vee \rangle + 1 \stackrel{\eqref{wt-string-2}}{\leq}\langle \lambda, \alpha_0^\vee \rangle + 1 \leq p - \hgt(\alpha_0^\vee) + 1 \leq p.$$

\end{proof}

\begin{Pt*}

Now we proceed to the construction of an action of the Lie algebra $\fn_\omega\otimes_{\bbF_p}k_F$ on an appropriate graded module associated to $L(\lambda)$.  To start with, let us define $\bfG'$ as the Weil restriction from $\cO_F$ to $\bbZ_p$ of $\bfG$:
$$\bfG' := \textnormal{Res}_{\cO_F/\bbZ_p}(\bfG).$$
We use analogous primed notation $\bfM', \bfN'$, etc., for the Weil restrictions of $\bfM, \bfN$, etc.  We set $N_0' := \bfN'(\cO_F)$, so that
\begin{equation}
\label{N0'-prod}
N_0' = \bfN(\cO_F \otimes_{\bbZ_p} \cO_F) \cong \prod_{\varsigma:\cO_F \rightarrow \cO_F}(\bfN\times_{\cO_F,\varsigma}\cO_F)(\cO_F)
\end{equation}
(note that since $\cO_F$ is unramified over $\bbZ_p$, we may identify the field isomorphisms $k_F \longrightarrow k_F$ over $\bbF_p$ with ring automorphisms $\cO_F \longrightarrow \cO_F$ over $\bbZ_p$).  We use similar notation $M_0', G_0'$, etc.  Analogously to the isomorphism \eqref{unip-isom}, we have an isomorphism of $\cO_F$-schemes
$$\prod_{\beta\in \underline{\Phi}^+ - \underline{\Phi}_J^+} \bfU_{\beta} \stackrel{\sim}{\longrightarrow} \bfN'_{\cO_F},$$
inducing a homeomorphism
$$\prod_{\beta\in \underline{\Phi}^+ - \underline{\Phi}_J^+} \bfU_{\beta}(\cO_F) \stackrel{\sim}{\longrightarrow} N_0'.$$
\end{Pt*}

\begin{Pt*}
\label{N_0'-sect}
Let us now fix $\alpha \in \Phi^+ - \Phi^+_J$ (without an underline!), and let $\alpha_\varsigma \in \underline{\Phi}^+ - \underline{\Phi}^+_J$ denote the root of $\underline{\bfT}_{k_F}$ which is equal to $\alpha$ in embedding $\varsigma$.  Using the inclusion $N_0 \longhookrightarrow N_0'$ and the above homeomorphism, the element $u_\alpha(x) \in N_0$, for $x \in \cO_F$, may then be written as
\begin{equation}
\label{unip-res-scalars}
u_{\alpha}(x) = \prod_{\varsigma:\cO_F \rightarrow \cO_F} u_{\alpha_\varsigma}(\varsigma(x)).
\end{equation}
The $p$-valuation $\omega: N_0 \longrightarrow \bbR_{> 0} \cup \{\infty\}$ induces a $p$-valuation on $N_0'$, given by the formula
\begin{eqnarray}
N_0' \cong \prod_{\varsigma:\cO_F\rightarrow \cO_F}(\bfN\times_{\cO_F,\varsigma}\cO_F)(\cO_F) & \longrightarrow & \bbR_{>0} \cup \{\infty\} \notag \\
\prod_{\varsigma}n_\varsigma & \longmapsto & \min_{\varsigma}\{\omega(n_\varsigma)\} \label{omega-N0'}
\end{eqnarray}
(see \cite[\S\S III.3.1.7.3, II.1.1.4, I.2.1.8]{lazard}).  We denote this $p$-valuation by $\omega$; note that the restriction of the function \eqref{omega-N0'} to $N_0$ is equal to the function $\omega$ defined in \eqref{omega-def}, so there is no conflict in notation.  In particular, the $p$-valuation $\omega$ evaluated on the left-hand side of \eqref{unip-res-scalars} (using equation \eqref{omega-def}) is equal to $\omega$ evaluated on the right-hand side of \eqref{unip-res-scalars} (using equation \eqref{omega-N0'}).  Additionally, by definition we have $\omega(m'n'm'^{-1}) = \omega(n')$ for all $n' \in N_0', m' \in M_0'$.  Finally, we note that we may construct $\cO_F\llbracket N_0'\rrbracket$ and the induced filtration $w$ exactly as for $\cO_F\llbracket N_0 \rrbracket$, and that the restriction of $w$ from $\cO_F\llbracket N_0'\rrbracket$ to $\cO_F\llbracket N_0\rrbracket$ is equal to the valuation $w$ on $\cO_F\llbracket N_0\rrbracket$ previously defined (see \cite[Cor. III.2.3.5]{lazard}).  We may therefore construct the ring filtration $\cO_F\llbracket N_0'\rrbracket_\bullet$ (and its mod $p$ analog), which satisfies
$$\cO_F\llbracket N_0'\rrbracket_\nu \cap \cO_F\llbracket N_0 \rrbracket = \cO_F\llbracket N_0 \rrbracket_\nu.$$
\end{Pt*}

\begin{Pt*}
Let us now fix a $p$-small $\lambda \in X^*(\underline{\bfT})_+$, and let $L(\lambda)$ denote the associated representation of $\underline{\bfG}(k_F)$.  By restricting the action to $\underline{\bfP}(k_F)$ and inflating to $P_0' = M_0'N_0'$, we view $L(\lambda)$ as a smooth $k_F\llbracket N_0' \rrbracket$-module with a compatible action of $M_0'$.  We recall that the $p$-valuation $\omega$ on $N_0'$ is valued in $\frac{1}{h'}\bbZ \cup \{\infty\}$, and the same is true of the valuation $w$ on $\cO_F\llbracket N_0' \rrbracket$.  We then define a decreasing filtration $\Fil_\bullet(L(\lambda))$ on $L(\lambda)$ indexed by $\frac{1}{h'}\bbZ$ as follows: if $\nu \in \frac{1}{h'}\bbZ_{\geq 0}$, we set $\Fil_\nu(L(\lambda)) = 0$, while if $\nu \in \frac{1}{h'}\bbZ_{< 0}$, we set
$$\Fil_\nu(L(\lambda)) := L(\lambda)[k_F\llbracket N_0'\rrbracket_{-\nu}] = \left\{v \in L(\lambda): \xi \cdot v = 0 \quad \textnormal{for all}\quad \xi \in k_F\llbracket N_0'\rrbracket_{-\nu}\right\}.$$ 
This defines a filtration which is compatible with the decreasing filtration $k_F\llbracket N_0'\rrbracket_\bullet$ on $k_F\llbracket N_0'\rrbracket$.  Furthermore, since the $p$-valuation $\omega$ is invariant by $M_0'$-conjugation, the filtration $k_F\llbracket N_0' \rrbracket_\bullet$ is stable by $M_0'$-conjugation, which shows that $\Fil_\bullet(L(\lambda))$ is a filtration by $M_0'$-submodules.

By passing to the associated graded module of $L(\lambda)$ relative to $\Fil_\bullet(L(\lambda))$, we obtain the space
$$\gr(L(\lambda)) := \bigoplus_{\nu \in \bbR} \gr_{\nu}(L(\lambda)) := \bigoplus_{\nu \in \bbR}\Fil_{\nu}(L(\lambda))/\Fil_{\nu + 1/h'}(L(\lambda)),$$
which is a graded $M_0'$-representation over the graded ring $\gr(k_F\llbracket N_0'\rrbracket)$.  In what follows, we will restrict the actions to the smaller groups, and consider $\gr(L(\lambda))$ as a graded $M_0$-representation over the graded ring $\gr(k_F\llbracket N_0\rrbracket)$.  In particular, by Corollary \ref{lazardcor}, $\gr(L(\lambda))$ has the structure of a graded module over the $k_F$-Lie algebra $\fn_\omega\otimes_{\bbF_p}k_F$: if $\overline{u_\alpha(x_i)}\otimes 1$ is a basis vector for $\fn_\omega\otimes_{\bbF_p}k_F$ as in \eqref{n-omega-basis}, and $v\in \Fil_\nu(L(\lambda))$ with image $\overline{v} \in \gr_{\nu}(L(\lambda))$, then
$$\left(\overline{u_\alpha(x_i)}\otimes 1\right)\star \overline{v} = \overline{(u_\alpha(x_i)\cdot v - v)}  \in \gr_{\nu + \omega(u_\alpha(x_i))}(L(\lambda))$$
(where we use the notation $\star$ to denote the Lie algebra action).  
\end{Pt*}

\begin{Pt*}
On the other hand, denoting by $\fn|_{\bbF_p}$ the Lie algebra $\fn$ considered as an $\bbF_p$-Lie algebra, we may construct an action of the Lie algebra $\fn|_{\bbF_p}\otimes_{\bbF_p} k_F$ on $L(\lambda)$.  First note that by \cite[Cor. A.7.6]{CGP}, we can identify $\fn|_{\bbF_p}$ as follows:
$$\fn|_{\bbF_p} = \textnormal{Lie}(\bfN)(k_F)|_{\bbF_p} \cong \textnormal{Lie}\left(\textnormal{Res}_{\cO_F/\bbZ_p}(\bfN)\right)(\bbF_p) \cong \textnormal{Lie}\left(\textnormal{Res}_{k_F/\bbF_p}(\bfN_{k_F})\right)(\bbF_p) = \textnormal{Lie}(\underline{\bfN})(\bbF_p).$$
(In the penultimate step we have used that $F/\bbQ_p$ is unramified.)   Thus $\fn|_{\bbF_p}\otimes_{\bbF_p}k_F$ is the Lie algebra of $\underline{\bfN}_{k_F}$ (\cite[\S II.4.1.4]{demazuregabriel}).

%For future reference, we note the decomposition:
%$$\fn|_{\bbF_p} \otimes_{\bbF_p} k_F \cong \textnormal{Lie}(\underline{\bfN})(\bbF_p)\otimes_{\bbF_p} k_F \cong \bigoplus_{\varsigma:k_F \rightarrow k_F} \textnormal{Lie}(\bfN)(k_F) \otimes_{k_F,\varsigma} k_F = \bigoplus_{\varsigma:k_F \rightarrow k_F} \fn \otimes_{k_F,\varsigma} k_F.$$

Then, given $n \in \fn|_{\bbF_p}\otimes_{\bbF_p}k_F = \textnormal{Lie}(\underline{\bfN}_{k_F})(k_F) = \ker(\underline{\bfN}(k_F[\varepsilon]/(\varepsilon^2)) \longrightarrow \underline{\bfN}(k_F))$, the action of $n$ on $L(\lambda)\otimes_{k_F} k_F[\varepsilon]/(\varepsilon^2) = L(\lambda) \oplus L(\lambda)\varepsilon$ is of the form 
\begin{equation}
\label{defofXn}
n\cdot (v + v'\varepsilon) = v + (v' + X_n(v))\varepsilon
\end{equation}
for some $X_n \in \End_{k_F}(L(\lambda))$ (see \cite[Prop. II.4.2.2]{demazuregabriel}).  We then define the Lie algebra action of $n \in \fn|_{\bbF_p}\otimes_{\bbF_p}k_F$ on $v \in L(\lambda)$ by 
\begin{equation}
\label{lie-alg-action-algebraic}
n\star v := X_n(v) \in L(\lambda).
\end{equation}
\end{Pt*}

\begin{Pt*}
Finally, we compare the two Lie algebra actions constructed above.

Let us define a $M_0$-equivariant isomorphism $\gamma: L(\lambda) \stackrel{\sim}{\longrightarrow} \gr(L(\lambda))$ as follows.  Since the $M_0$-representation $L(\lambda)|_{\bfM(k_F)}$ is semisimple by Corollary \ref{algrepscor}, and since the filtration $\Fil_\bullet(L(\lambda))$ is $M_0$-stable, the short exact sequence of $M_0$-representations
$$0 \longrightarrow \Fil_{\nu + 1/h'}(L(\lambda)) \longrightarrow \Fil_\nu(L(\lambda)) \longrightarrow \textnormal{gr}_\nu(L(\lambda)) \longrightarrow 0$$
splits.  By taking the direct sum over all $\nu$ of the splittings $\gr_\nu(L(\lambda)) \longrightarrow \Fil_\nu(L(\lambda)) \longhookrightarrow L(\lambda)$, we obtain an $M_0$-equivariant isomorphism $\gr(L(\lambda)) \stackrel{\sim}{\longrightarrow} L(\lambda)$, and we define $\gamma$ to be the inverse of this map.  
\end{Pt*}

\begin{Lem*}
\label{lazard-vs-alg-lie}
Let $\lambda \in X^*(\underline{\bfT})_+$ be $p$-small.
%(both notions with respect to $\underline{\Phi}^+$)  
Then the map $\gamma$ intertwines the actions of $\fn|_{\bbF_p}\otimes_{\bbF_p}k_F$ and $\fn_\omega \otimes_{\bbF_p} k_F$; more precisely, we have
$$\gamma(n \star v) = \ff(n) \star \gamma(v)$$
where $n \in \fn|_{\bbF_p}\otimes_{\bbF_p}k_F, v \in L(\lambda)$, and $\ff$ is the (base change of the) isomorphism \eqref{def-of-ff}.
\end{Lem*}

\begin{proof}
Recall that we may identify $\underline{\Phi}$ with $[k_F:\bbF_p]$ copies of $\Phi$, indexed by the field homomorphisms $\varsigma: k_F \longrightarrow k_F$ over $\bbF_p$.  Let $\beta \in \underline{\Phi}^+ - \underline{\Phi}^+_J$ and $x \in k_F$.  Then 
$$u_\beta(x\varepsilon) \in \ker\left(\underline{\bfN}(k_F[\varepsilon]/(\varepsilon^2)) \longrightarrow \underline{\bfN}(k_F)\right),$$
and we let $X_\beta$ denote the endomorphism $X_{u_\beta(\varepsilon)}$ defined in equation \eqref{defofXn} above.  According to \cite[Ch. II, \S 1.19, Eqn. (6)]{jantzen}, the action of $u_\beta(x\varepsilon)$ on $v \in L(\lambda) \subset L(\lambda) \oplus L(\lambda)\varepsilon = L(\lambda) \otimes_{k_F} k_F[\varepsilon]/(\varepsilon^2)$ is given by
$$u_{\beta}(x\varepsilon)\cdot v = v + x  X_\beta(v)\varepsilon \in L(\lambda) \oplus L(\lambda)\varepsilon.$$
Using the same cited formula, the action of $u_\beta(x)$ on $v \in L(\lambda)$ is given by
\begin{eqnarray*}
u_\beta(x)\cdot v & = & v + x X_\beta(v) +  \frac{x^2}{2!}X_\beta^2(v) + \frac{x^3}{3!}X_\beta^3(v) + \ldots \\
 & = & \sum_{n \geq 0}\frac{x^n}{n!}X_\beta^n(v) \in L(\lambda). 
\end{eqnarray*}
To see that this expression is well-defined, suppose $\mu \in X^*(\underline{\bfT})$ and $v \in L(\lambda)_\mu$ is a non-zero weight vector.  Then $X_\beta^n(v) \in L(\lambda)_{\mu + n\beta}$ (\cite[Ch. II, \S 1.19, Eqn. (5)]{jantzen}).  By Lemma \ref{wt-string-lemma}, the $p$-smallness assumption guarantees that $X_\beta^p(v) \in L(\lambda)_{\mu + p\beta} = 0$.  Hence, $X_\beta^p = 0$ on $L(\lambda)$, and we may write
\begin{equation}
\label{exponential}
u_\beta(x)\cdot v = \sum_{n = 0}^{p - 1} \frac{x^n}{n!}X_\beta^n(v).
\end{equation}
Furthermore, since $(u_\beta(x) - 1)$ commutes with $X_{\beta}$, we see that $(u_\beta(x) - 1)^p$ acts by 0 on $L(\lambda)$.  Thus, we may invert equation \eqref{exponential} to obtain
\begin{eqnarray*}
xX_{\beta}(v) & = & (u_\beta(x) - 1)\cdot v - \frac{1}{2}(u_\beta(x) - 1)^2\cdot v + \frac{1}{3}(u_\beta(x) - 1)^3\cdot v - \ldots \\
 & = & \sum_{n = 1}^{p - 1}\frac{(-1)^{n + 1}}{n}(u_\beta(x) - 1)^n \cdot v.
\end{eqnarray*}

We may now calculate the action of $X_\beta$ on the filtration $\Fil_\bullet(L(\lambda))$.  For $n \geq 0$ and $x \in \cO_F$, the element $(u_\beta(x) - 1)^n \in \cO_F\llbracket N_0' \rrbracket$ has valuation $w((u_\beta(x) - 1)^n) = n\omega(u_\beta(x))$ by \cite[Thm. III.2.3.3, eqn. III.2.3.8.8]{lazard}.  This implies that if $v \in \Fil_{\nu}(L(\lambda))$, then $(u_\beta(x) - 1)^n\cdot v \in \Fil_{\nu + n\omega(u_\beta(x))}(L(\lambda))$.  From the previous equation we therefore see that
\begin{equation}
\label{lie-alg-gr-act-1}
\overline{x}X_\beta(v)  =  \sum_{n = 1}^{p - 1}\frac{(-1)^{n + 1}}{n}(u_\beta(x) - 1)^n \cdot v \in \Fil_{\nu + \omega(u_\beta(x))}(L(\lambda)). 
\end{equation}

Let us now fix $\alpha \in \Phi^+ - \Phi^+_J$ (without an underline!) and $x_i$ as in the proof of Lemma \ref{psatN0}, so that $u_\alpha(\overline{x_i}\varepsilon) \in \fn|_{\bbF_p}$ and $u_{\alpha}(\overline{x_i}\varepsilon) \otimes 1 \in \fn|_{\bbF_p} \otimes_{\bbF_p}k_F$ is a basis vector.  Letting $\alpha_\varsigma \in \underline{\Phi}^+ - \underline{\Phi}^+_J$ denote the root of $\underline{\bfT}_{k_F}$ which is equal to $\alpha$ in embedding $\varsigma$, then by equation \eqref{unip-res-scalars} we may write
\begin{equation}
\label{unip-res-scalars-modp}
u_{\alpha}(\overline{x_i}\varepsilon) \otimes 1 = \prod_{\varsigma:k_F \rightarrow k_F} u_{\alpha_\varsigma}(\varsigma(\overline{x_i})\varepsilon).
\end{equation}
Note that the terms on the right hand side commute pairwise, as do the elements $X_{\alpha_\varsigma}$.  Combining equations \eqref{defofXn} and \eqref{unip-res-scalars-modp} shows that we have
\begin{equation}
\label{lie-alg-comps}
X_{u_\alpha(\overline{x_i}\varepsilon) \otimes 1} = \sum_{\varsigma: k_F \rightarrow k_F}\varsigma(\overline{x_i})X_{\alpha_\varsigma}.
\end{equation}
Applying equation \eqref{unip-res-scalars} once again, we have
\begin{equation}
\label{unip-res-scalars-2}
u_\alpha(\overline{x_i}) = \prod_{\varsigma:k_F \rightarrow k_F}u_{\alpha_\varsigma}(\varsigma(\overline{x_i}))
\end{equation}
as elements of $\underline{\bfN}(k_F)$.  Thus, if $v \in \Fil_\nu(L(\lambda))$, we combine the above equations to obtain
\begin{eqnarray*}
(u_\alpha(\overline{x_i}) - 1)\cdot v & \stackrel{\eqref{unip-res-scalars-2}}{=} & \left(\prod_{\varsigma:k_F \rightarrow k_F}u_{\alpha_\varsigma}(\varsigma(\overline{x_i}))\right)\cdot v - v\\
& \stackrel{\eqref{exponential}}{=} & \left(\prod_{\varsigma:k_F \rightarrow k_F}~\sum_{n = 0}^{p - 1}\frac{\varsigma(\overline{x_i})^n}{n!}X_{\alpha_\varsigma}^n\right)(v)  - v\\
& \stackrel{\eqref{lie-alg-gr-act-1}}{\in} & \sum_{\varsigma:k_F \rightarrow k_F} \varsigma(\overline{x_i})X_{\alpha_\varsigma}(v) + \Fil_{\nu + \omega(u_\alpha(x_i)) + 1/h'}(L(\lambda))\\
& \stackrel{\eqref{lie-alg-comps}}{=} & X_{u_{\alpha}(\overline{x_i}\varepsilon) \otimes 1}(v) + \Fil_{\nu + \omega(u_\alpha(x_i)) + 1/h'}(L(\lambda)),
\end{eqnarray*}
where we have used the fact that $\omega(u_\alpha(x_i)) = \omega(u_{\alpha_\varsigma}(\varsigma(x_i)))$ for every $\varsigma$.  By equation \eqref{lie-alg-action-algebraic}, the element $X_{u_\alpha(\overline{x_i}\varepsilon)\otimes 1}(v)$ is exactly the Lie algebra action of $u_{\alpha}(\overline{x_i}\varepsilon)\otimes 1$ on $v$.  Thus, we finally get
\begin{eqnarray*}
\gamma\left((u_\alpha(\overline{x_i}\varepsilon)\otimes 1)\star v\right) & = & \overline{X_{u_\alpha(\overline{x_i}\varepsilon)\otimes1}(v)} \\
 & = & \overline{(u_{\alpha}(x_i) - 1)\cdot v} \\
 & = & \left(\overline{u_\alpha(x_i)}\otimes 1\right)\star \overline{v}\\
 & = & \ff(u_{\alpha}(\overline{x_i}\varepsilon)\otimes 1)\star \gamma(v).
\end{eqnarray*}
\end{proof}

\subsection{Cohomology of $\fn$} \label{coh-n}

We now calculate the Lie algebra cohomology of $\fn|_{\bbF_p} \otimes_{\bbF_p}k_F$ with coefficients in a representation with $p$-small highest weight.

Before stating the result, we set some notation.  Recall that $W = W(\underline{\bfG},\underline{\bfT})$ denotes the Weyl group of $\underline{\bfG}_{k_F}$ relative to $\underline{\bfT}_{k_F}$.  Given $J \subset \Delta$ with associated Levi subgroup $\bfM$, by a slight abuse of notation we let ${}^JW$ denote the set of minimal length coset representatives for $W(\underline{\bfM},\underline{\bfT})$ in $W$.  Using the decomposition \eqref{resGdecomp}, the set ${}^JW$ decomposes as $[k_F:\bbF_p]$ copies of ${}^JW(\bfG,\bfT)$, the set of minimal length coset representatives for $W(\bfM,\bfT)$ in $W(\bfG,\bfT)$.  Finally, given $w \in W$ and $\lambda \in X^*(\underline{\bfT})$, we define the dot action by
$$w\cdot \lambda := w(\lambda + \rho) - \rho.$$

\begin{Th*}
    \label{kostant}
    Suppose Assumption \ref{assn1} holds, and let $\lambda \in X^*(\underline{\bfT})_+$ be $p$-small.  Let $H^i(\fn|_{\bbF_p} \otimes_{\bbF_p}k_F, L(\lambda))$ denote the $k_F$-linear Lie algebra cohomology with coefficients in the representation $L(\lambda)$.  Then we have an isomorphism of $\underline{\bfM}(k_F)$-representations over $k_F$:
    $$H^i\left(\fn|_{\bbF_p}\otimes_{\bbF_p} k_F,~L(\lambda)\right) \cong \bigoplus_{\substack{w\in {}^JW\\ \ell(w) = i}}L_J(w\cdot \lambda)\ .$$
\end{Th*}

We note that the characters $w\cdot \lambda$ with $w \in {}^JW$ are dominant and $p$-small with respect to $\underline{\Phi}^+_J$, and furthermore that such a character encodes a ``Galois twist.''  For example, if we take $\bfM = \bfT$, $\lambda = 0$ and $i = 1$, then the characters of $\bfT(k_F) = \underline{\bfT}(\bbF_p)$ appearing in $H^1(\fn|_{\bbF_p}\otimes_{\bbF_p}k_F ,L(0))$ are exactly
\begin{eqnarray*}
\bfT(k_F) & \longrightarrow & k_F^\times\\
t & \longmapsto & \alpha(t)^{-p^j}
\end{eqnarray*}
for $\alpha \in \Delta$ and $0 \leq j < [k_F:\bbF_p] = [F:\bbQ_p]$.

\begin{proof}
We prove the theorem in several stages.

\textit{Step 0.}  We first claim that it suffices to prove the claim after base-changing to $\overline{k_F}$.  Indeed, suppose we have an isomorphism of $\underline{\bfM}(\overline{k_F})$-representations over $\overline{k_F}$
\begin{equation}
\label{kostant-base-change-to-alg-closed}
H^i\left(\fn|_{\bbF_p}\otimes_{\bbF_p} \overline{k_F},~L(\lambda) \otimes_{k_F} \overline{k_F}\right) \cong \bigoplus_{\substack{w\in {}^JW\\ \ell(w) = i}}L_J(w\cdot \lambda)\otimes_{k_F} \overline{k_F}.
\end{equation}
By restricting to $\underline{\bfM}(k_F)$, we see the above as an isomorphism of $\underline{\bfM}(k_F)$-representations over $\overline{k_F}$.  Now, by Corollary \ref{algrepscor}\ref{algrepscor-b}, the $\underline{\bfM}(k_F)$-representation $\bigoplus_{w \in {}^JW, \ell(w) = i} L_J(w\cdot \lambda)$ is semisimple, and by \cite[\S 12.1, Prop. 1]{bourbaki:algch8} and \eqref{kostant-base-change-to-alg-closed}, the $\underline{\bfM}(k_F)$-representation $H^i(\fn|_{\bbF_p}\otimes_{\bbF_p} k_F,L(\lambda))$ is also semisimple.  For any $m\in \underline{\bfM}(k_F)$, the characteristic polynomials of $m$ on these two semisimple representations agree (again using \eqref{kostant-base-change-to-alg-closed}), and we conclude that they must be isomorphic by the Brauer--Nesbitt theorem (see \cite[Thm. 30.16]{curtisreiner}).

\textit{Step 1.}  Suppose that $\bfG_{k_F}$ is absolutely simple and simply connected.  Using \eqref{resGdecomp}, we have decompositions
\begin{eqnarray*}
\underline{\bfG}_{k_F} & \cong & \prod_{\varsigma:k_F \rightarrow k_F}\bfG_{k_F}\times_{k_F,\varsigma}k_F =: \prod_{\varsigma:k_F \rightarrow k_F}\bfG_{\varsigma},\\
\underline{\bfM}_{k_F} & \cong & \prod_{\varsigma:k_F \rightarrow k_F}\bfM_{k_F}\times_{k_F,\varsigma}k_F =: \prod_{\varsigma:k_F \rightarrow k_F}\bfM_{\varsigma},\\
 \underline{\bfT}_{k_F} & \cong & \prod_{\varsigma:k_F \rightarrow k_F}\bfT_{k_F}\times_{k_F,\varsigma}k_F =: \prod_{\varsigma:k_F \rightarrow k_F}\bfT_{\varsigma},\\
\fn|_{\bbF_p} \otimes_{\bbF_p} k_F & \cong &  \bigoplus_{\varsigma:k_F \rightarrow k_F} \fn \otimes_{k_F,\varsigma} k_F =: \bigoplus_{\varsigma:k_F \rightarrow k_F} \fn_{\varsigma}, \\
 L(\lambda) & \cong & \bigotimes_{\varsigma:k_F\rightarrow k_F}L(\lambda_{\varsigma}),
\end{eqnarray*}
where $\lambda = (\lambda_{\varsigma})_{\varsigma} \in X^*(\underline{\bfT}) = \bigoplus_{\varsigma:k_F \rightarrow k_F} X^*(\bfT_\varsigma)$.

Recall (from, e.g., \cite[Cor. 7.7.3]{weibel}) that the $\overline{k_F}$-linear cohomology $H^i(\fn|_{\bbF_p}\otimes_{\bbF_p}\overline{k_F}, L(\lambda)\otimes_{k_F} \overline{k_F})$ may be computed as the $i^{\textnormal{th}}$ cohomology of the Chevalley--Eilenberg complex
$$\Hom_{\overline{k_F}}\left(\sideset{}{^\bullet_{\overline{k_F}}}{\bigwedge}(\fn|_{\bbF_p}\otimes_{\bbF_p}\overline{k_F}),~L(\lambda)\otimes_{k_F} \overline{k_F}\right).$$  
Thus, using the K\"unneth formula, we obtain
\begin{eqnarray}
 H^i\left(\fn|_{\bbF_p} \otimes_{\bbF_p} \overline{k_F},~L(\lambda)\otimes_{k_F} \overline{k_F}\right)  & \cong & H^i\left(\bigoplus_{\varsigma:k_F \rightarrow k_F} \fn_\varsigma\otimes_{k_F} \overline{k_F},~\bigotimes_{\varsigma:k_F \rightarrow k_F}L(\lambda_\varsigma)\otimes_{k_F} \overline{k_F}\right) \notag \\ 
    & \cong & \bigoplus_{\sum_{\varsigma}i_\varsigma = i} \bigotimes_{\varsigma}H^{i_\varsigma}\left(\fn_\varsigma\otimes_{k_F} \overline{k_F},~ L(\lambda_\varsigma)\otimes_{k_F} \overline{k_F}\right), \label{liealgkunneth}
\end{eqnarray}
where the tensor product in the last line is taken over $\overline{k_F}$.  By \cite[Thm. 4.2.1]{ugavigre}, applied to the simple simply connected group 
$\bfG_{\varsigma,\overline{k_F}}$, the space $H^{i_\varsigma}(\fn_{\varsigma}\otimes_{k_F} \overline{k_F},L(\lambda_\varsigma)\otimes_{k_F} \overline{k_F})$ decomposes as 
\begin{equation}
    \label{twistedkostant}
    H^{i_\varsigma}\left(\fn_{\varsigma}\otimes_{k_F} \overline{k_F},~L(\lambda_{\varsigma})\otimes_{k_F} \overline{k_F}\right) \cong \bigoplus_{\substack{w_\varsigma \in {}^{J}W(\bfG_{\varsigma},\bfT_{\varsigma})\\ \ell(w_\varsigma) = i_\varsigma}} L_{J}(w_\varsigma \cdot \lambda_\varsigma) \otimes_{k_F} \overline{k_F}
\end{equation}
as representations of the group $\bfM_\varsigma(\overline{k_F})$.  Substituting \eqref{twistedkostant} into \eqref{liealgkunneth}, we obtain
\begin{eqnarray*}
    H^i\left(\fn|_{\bbF_p}\otimes_{\bbF_p}\overline{k_F},~L(\lambda)\otimes_{k_F} \overline{k_F}\right) & \cong & \bigoplus_{\sum_{\varsigma}i_\varsigma = i} \bigotimes_{\varsigma} \bigoplus_{\substack{w_\varsigma \in {}^{J}W(\bfG_{\varsigma},\bfT_{\varsigma})\\ \ell(w_\varsigma) = i_\varsigma}} L_{J}(w_\varsigma \cdot \lambda_{\varsigma})\otimes_{k_F} \overline{k_F} \\
    & \cong & \bigoplus_{\substack{w = (w_\varsigma)_\varsigma \in {}^JW(\underline{\bfG},\underline{\bfT}) \\ \ell(w) = i}} \bigotimes_{\varsigma} L_{J}(w_\varsigma\cdot \lambda_{\varsigma})\otimes_{k_F} \overline{k_F}.
\end{eqnarray*}
Note that the action of $\underline{\bfM}(\overline{k_F})$ on the tensor product in the last line above is exactly $L_J(w\cdot \lambda)\otimes_{k_F} \overline{k_F}$.  This gives the result in the absolutely simple, simply connected case.

\textit{Step 2.}  Suppose now that $\bfG_{k_F}$ is semisimple and simply connected.  In this case, we have decompositions
$$\bfG_{k_F} \cong \bfG_1\times \bfG_2\times \cdots \times \bfG_r,$$
where each $\bfG_i$ is split over $k_F$, absolutely simple, and simply connected (this follows from \cite[Thm. 5.1.17]{conrad:sga3} by considering cocharacter lattices).  We obtain analogous decompositions of groups
$$\bfM_{k_F} \cong \bfM_1\times \bfM_2\times \cdots \times \bfM_r, \qquad \bfN_{k_F} \cong \bfN_1\times \bfN_2\times \cdots \times \bfN_r,$$
$$\bfT_{k_F} \cong \bfT_1\times \bfT_2\times \cdots \times \bfT_r, \qquad W \cong W(\underline{\bfG}_1,\underline{\bfT}_1) \times W(\underline{\bfG}_2,\underline{\bfT}_2) \times \cdots \times W(\underline{\bfG}_r,\underline{\bfT}_r),$$
along with a Lie algebra decomposition
$$\fn \cong \fn_1 \oplus \fn_2 \oplus \ldots \oplus \fn_r,$$
and a decomposition of character lattices
$$X^*(\underline{\bfT}) \cong X^*(\underline{\bfT}_1) \oplus X^*(\underline{\bfT}_2) \oplus \ldots \oplus X^*(\underline{\bfT}_r).$$
Therefore, the K\"unneth formula gives a $\underline{\bfM}(\overline{k_F})$-equivariant isomorphism
$$H^i\left(\fn|_{\bbF_p}\otimes_{\bbF_p}\overline{k_F},~L(\lambda)\otimes_{k_F} \overline{k_F}\right) \cong \bigoplus_{\sum_{j = 1}^r i_j = i}~\bigotimes_{j = 1}^r H^{i_j}\left(\fn_j|_{\bbF_p}\otimes_{\bbF_p}\overline{k_F},~L(\lambda_j)\otimes_{k_F} \overline{k_F}\right).$$
The desired claim now follows from Step 1 applied to each $H^{i_j}(\fn_j|_{\bbF_p}\otimes_{\bbF_p}\overline{k_F},~L(\lambda_j)\otimes_{k_F} \overline{k_F})$.

\textit{Step 3.} Next, suppose that $\bfG_{k_F}$ has simply connected derived subgroup $\bfG_{k_F}^{\textnormal{der}}$.  The inclusion 
$\bfG_{k_F}^{\textnormal{der}} \longhookrightarrow \bfG_{k_F}$ induces an equality $\bfN_{k_F} \cap \bfG_{k_F}^{\textnormal{der}} = \bfN_{k_F}$, and we therefore obtain an isomorphism of $(\underline{\bfM} \cap \underline{\bfG}^{\textnormal{der}})(\overline{k_F})$-representations
$$H^i\left(\fn|_{\bbF_p}\otimes_{\bbF_p}\overline{k_F},~L(\lambda)\otimes_{k_F} \overline{k_F}\right)|_{(\underline{\bfM} \cap \underline{\bfG}^{\textnormal{der}})(\overline{k_F})} \cong H^i\left(\fn^{\textnormal{der}}|_{\bbF_p}\otimes_{\bbF_p}\overline{k_F},~L(\lambda|_{\underline{\bfT} \cap \underline{\bfG}^{\textnormal{der}}})\otimes_{k_F} \overline{k_F}\right),$$
where $\fn^{\textnormal{der}}$ denotes the Lie algebra of $\bfN_{k_F}\cap \bfG_{k_F}^{\textnormal{der}}$, and $\underline{\bfT} \cap \underline{\bfG}^{\textnormal{der}}$ is a maximal torus of $\underline{\bfG}^{\textnormal{der}}$, cf. \cite[Prop. 5.3.4]{conrad:sga3} (as above we use underlines to denote scalar restriction from $k_F$ to $\bbF_p$).  By Step 2, we obtain 
$$H^i\left(\fn|_{\bbF_p}\otimes_{\bbF_p}\overline{k_F},~L(\lambda)\otimes_{k_F} \overline{k_F}\right)|_{(\underline{\bfM} \cap \underline{\bfG}^{\textnormal{der}})(\overline{k_F})} \cong \bigoplus_{\substack{w \in {}^JW(\underline{\bfG}^{\textnormal{der}}, \underline{\bfT} \cap \underline{\bfG}^{\textnormal{der}}) \\ \ell(w) = i}} L_J(w\cdot (\lambda|_{\underline{\bfT} \cap \underline{\bfG}^{\textnormal{der}}}))\otimes_{k_F} \overline{k_F}.$$
On the other hand, if we denote by $\bfZ$ the center of $\bfG$, then the action of $\underline{\bfZ}(\overline{k_F})$ on $H^i(\fn|_{\bbF_p}\otimes_{\bbF_p}\overline{k_F},L(\lambda)\otimes_{k_F} \overline{k_F})$ is given by the character $\lambda|_{\underline{\bfZ}(\overline{k_F})}$.  Since $\underline{\bfZ}(\overline{k_F})\underline{\bfG}^{\textnormal{der}}(\overline{k_F}) = \underline{\bfG}(\overline{k_F})$, we conclude that
$$H^i\left(\fn|_{\bbF_p}\otimes_{\bbF_p}\overline{k_F},~L(\lambda)\otimes_{k_F} \overline{k_F}\right) \cong \bigoplus_{\substack{w \in {}^JW \\ \ell(w) = i}} L_J(w\cdot \lambda)\otimes_{k_F} \overline{k_F}$$
as representations of $\underline{\bfM}(\overline{k_F})$ (using the canonical identification of $W$ with $W(\underline{\bfG}^{\textnormal{der}}, \underline{\bfT} \cap \underline{\bfG}^{\textnormal{der}})$).  This gives the claim when $\bfG_{k_F}^{\textnormal{der}}$ is simply connected.

\textit{Step 4.} Finally, suppose $\bfG_{k_F}$ is an arbitrary split connected reductive group, and choose a $z$-extension
$$1 \longrightarrow \widetilde{\bfZ} \longrightarrow \widetilde{\bfG} \stackrel{q}\longrightarrow \bfG_{k_F} \longrightarrow 1,$$
where $\widetilde{\bfG}$ is a split connected reductive group over $k_F$ with simply connected derived subgroup, and $\widetilde{\bfZ}$ is a split central torus (as in the proof of Corollary \ref{algrepscor}\ref{algrepscor-b}). 
Restricting scalars to $\bbF_p$, we obtain a short exact sequence
$$1 \longrightarrow \underline{\widetilde{\bfZ}} \longrightarrow \underline{\widetilde{\bfG}} \stackrel{\underline{q}}{\longrightarrow} \underline{\bfG} \longrightarrow 1.$$
The inflation along $\underline{q}_{k_F}:=\underline{q}\otimes_{\bbF_p}k_F$ of $L(\lambda)$ is the irreducible algebraic representation $L(\widetilde{\lambda})$ of 
$\underline{\widetilde{\bfG}}_{k_F}$ where $\widetilde{\lambda}$ is the inflation of $\lambda$ along $\underline{q}_{k_F}$.  Since $q$ identifies unipotent radicals, and $\widetilde{\lambda}$ is trivial on $\underline{\widetilde{\bfZ}}_{k_F}$, we have an isomorphism of $\underline{q}^{-1}(\underline{\bfM})(\overline{k_F})$-representations
$$H^i\left(\fn|_{\bbF_p}\otimes_{\bbF_p}\overline{k_F},~L(\lambda)\otimes_{k_F} \overline{k_F}\right) \cong H^i\left(\widetilde{\fn}|_{\bbF_p}\otimes_{\bbF_p}\overline{k_F},L(\widetilde{\lambda})\otimes_{k_F} \overline{k_F}\right)$$
where the left hand side is inflated from $\underline{\bfM}(\overline{k_F})$ and $\widetilde{\fn}:=\textnormal{Lie}(q^{-1}(\bfN_{k_F}))(k_F)$ in the right hand side.  Applying Step 3 to $\widetilde{\bfG}$, we get an isomorphism of $\underline{q}^{-1}(\underline{\bfM})(\overline{k_F})$-representations
\begin{equation}
\label{z-ext-H1}
H^i\left(\fn|_{\bbF_p}\otimes_{\bbF_p}\overline{k_F},~L(\lambda)\otimes_{k_F} \overline{k_F}\right) \cong \bigoplus_{\substack{w \in {}^JW(\underline{\widetilde{\bfG}}, \underline{q}^{-1}(\underline{\bfT})) \\ \ell(w) = i}} L_J(w\cdot \widetilde{\lambda})\otimes_{k_F} \overline{k_F}.
\end{equation}
%Since $\underline{\widetilde{\bfZ}}(\overline{k_F})$ acts trivially on the right-hand side of equation \eqref{z-ext-H1}, the action descends to $q^{-1}(\underline{\bfM})(\overline{k_F})/\underline{\widetilde{\bfZ}}(\overline{k_F}) = \underline{\bfM}(\overline{k_F})$ to give an $\underline{\bfM}(\overline{k_F})$-equivariant isomorphism
%$$H^i\left(\fn|_{\bbF_p}\otimes_{\bbF_p}\overline{k_F},~L(\lambda)\otimes_{k_F} \overline{k_F}\right) \cong \bigoplus_{\substack{w \in {}^JW \\ \ell(w) = i}} L_J(w\cdot \lambda)\otimes_{k_F} \overline{k_F}$$
%(using the canonical identification of $W$ with $W(\underline{\widetilde{\bfG}}, q^{-1}(\underline{\bfT}))$). 
Equivalently, we have an isomorphism of $\underline{\bfM}(\overline{k_F})$-representations
$$H^i\left(\fn|_{\bbF_p}\otimes_{\bbF_p}\overline{k_F},~L(\lambda)\otimes_{k_F} \overline{k_F}\right) \cong \bigoplus_{\substack{w \in {}^JW \\ \ell(w) = i}} L_J(w\cdot \lambda)\otimes_{k_F} \overline{k_F}$$
(using the canonical identification of $W$ with $W(\underline{\widetilde{\bfG}}, \underline{q}^{-1}(\underline{\bfT}))$). 
This finishes the final step and concludes the proof.
\end{proof}

\begin{Cor*}
\label{kostant-cor}
Suppose Assumption \ref{assn1} holds, and let $\lambda \in X^*(\underline{\bfT})_+$ be $p$-small.  Then we have an isomorhpism of $k_F$-linear $M_0$-representations
    $$H^i\left(\fn_\omega\otimes_{\bbF_p} k_F,~\gr(L(\lambda))\right) \cong \bigoplus_{\substack{w\in {}^JW\\ \ell(w) = i}}L_J(w\cdot \lambda).$$
\end{Cor*}

\begin{proof}
This follows from Theorem \ref{kostant} and Lemma \ref{lazard-vs-alg-lie}, using that the map $\gamma$ is both $M_0$-equivariant and equivariant for the two Lie algebra actions.
\end{proof}

\subsection{Cohomology of $N_0$} \label{coh-N_0}

Our next goal is to calculate the cohomology of $N_0$ with coefficients in a representation with $p$-small highest weight.  We begin with a few lemmas.

Recall the notation $\bfG' := \textnormal{Res}_{\cO_F/\bbZ_p}(\bfG)$, $\bfM' := \textnormal{Res}_{\cO_F/\bbZ_p}(\bfM)$, $\bfT' := \textnormal{Res}_{\cO_F/\bbZ_p}(\bfT)$, etc. In particular $\bfT'$ is a torus over $\bbZ_p$ which splits over $\cO_F$. Consequently, the abstract group of characters of $\bfT'_{\cO_F}$ identifies canonically with the one of $\bfT'_{F}$ and of $\bfT'_{k_F}=\underline{\bfT}_{k_F}$, which we have simply denoted by $X^*(\underline{\bfT})$.
% \ref{resTdecomp}, \cite[App. B]{conrad:sga3}

\begin{Lem*}
    \label{FpQpdims}
    Suppose Assumption \ref{assn1} holds.  Let $J\subset \Delta$, $\mu \in X^*(\underline{\bfT})$ dominant and $p$-small with respect to 
    $\underline{\Phi}_J^+$, and $L_{J}(\mu)$ (resp., $L_{J,F}(\mu)$) the irreducible algebraic representation of $\bfM'_{k_F} \cong \underline{\bfM}_{k_F}$ (resp., of $\bfM'_F$) of highest weight $\mu$.  Then 
    $$\dim_{k_F}\left(L_{J}(\mu)\right) = \dim_{F}\left(L_{J,F}(\mu)\right).$$
\end{Lem*}

\begin{proof}
    This follows exactly as in \cite[Prop. 1.6]{polotilouine}.  Namely, let $\cO_F(\mu)$ denote the rank 1 $\cO_F$-module on which $\bfT'_{\cO_F}$ acts by the character $\mu$. Then define 
    $$
    H^0(\mu) := \ind_{(\bfB'^- \cap \bfM')_{\cO_F}}^{\bfM'_{\cO_F}}(\cO_F(\mu)),
    $$  
    %Since $H^0(\mu)$ is a submodule of the free $\cO_F$-module $\mathscr{O}(\bfM'_{\cO_F})$
    where $\bfB'^{-} := \textnormal{Res}_{\cO_F/\bbZ_p}(\bfB^-)$ for the Borel subgroup $\bfB^-$ which is opposite to $\bfB$. It is a free $\cO_F$-module, cf.  \cite[\S II.8.8, Eqn. (1)]{jantzen}.

    On the one hand, we have the following sequence of isomorphisms:
    $$H^0(\mu)\otimes_{\cO_F} F \cong \ind_{(\bfB'^- \cap \bfM')_{F}}^{\bfM'_{F}}(F(\mu)) \cong L_{J,F}(\mu).$$
    The first isomorphism comes from flat base change (\cite[\S I.3.5, Eqn. (3)]{jantzen}), while the second is the Borel--Weil--Bott theorem (\cite[Cor. II.5.6]{jantzen}).

    On the other hand, by Kempf's vanishing theorem (\cite[\S II.8.8, Eqn. (2)]{jantzen}) we have 
    $$R^1\ind_{(\bfB'^- \cap \bfM')_{\cO_F}}^{\bfM'_{\cO_F}}(\cO_F(\mu)) = 0.$$
    Thus, applying the universal coefficient theorem (as in \cite[Prop. I.4.18(b)]{jantzen}), and the Borel--Weil--Bott theorem (\cite[Cor. II.5.6]{jantzen}), we get
    \begin{eqnarray*}
    H^0(\mu)\otimes_{\cO_F}k_F & \cong & H^0(\mu)\otimes_{\cO_F}k_F~ \oplus~ \textnormal{Tor}_1^{\cO_F}\left(R^1\ind_{(\bfB'^- \cap \bfM')_{\cO_F}}^{\bfM'_{\cO_F}}(\cO_F(\mu)),~ k_F\right) \\
     & \cong & \ind_{(\bfB'^- \cap \bfM')_{k_F}}^{\bfM'_{k_F}}(k_F(\mu)) \\
     & \cong & L_{J}(\mu).
    \end{eqnarray*}

    Combining the previous two paragraphs with the fact that $H^0(\mu)$ is free over $\cO_F$, we obtain
    $$\dim_{k_F}\left(L_{J}(\mu)\right) = \textnormal{rk}_{\cO_F}\left(H^0(\mu)\right)= \dim_{F}\left(L_{J,F}(\mu)\right).$$
\end{proof}

\begin{Rem*}
\label{FpQpdims-remark}
Suppose $\lambda \in X^*(\underline{\bfT})_+$ is $p$-small, and fix $J \subset \Delta$.  Then for every $w \in {}^JW$, the characters $w\cdot \lambda$ are dominant and $p$-small with respect to $\underline{\Phi}^+_J$ by Lemma \ref{p-small-J}, and therefore
    $$\dim_{k_F}\left(L_{J}(w\cdot \lambda)\right) = \dim_{F}\left(L_{J,F}(w\cdot \lambda)\right).$$
\end{Rem*}

\begin{Lem*} \label{compn:cts-to-lie-alg-Qp}
Suppose Assumption \ref{assn1} holds. Let $V$ be a finite dimensional algebraic representation of the algebraic group $\bfN_{\bbQ_p}'$.  Then we have an isomorphism
\begin{equation*}
    H^i_{\textnormal{cts}}\left(N_0,~V\right) \cong  H^i\left(\fN'_{\bbQ_p},~ V\right)
    \end{equation*}
    where $\fN'_{\bbQ_p} := \textnormal{Lie}(\bfN')(\bbQ_p)$ denotes the Lie algebra of $\bfN_{\bbQ_p}'$.
\end{Lem*}

\begin{proof}
Recall that $(N_0,\omega)$ is a $p$-saturated $p$-valued group by Lemma \ref{psatN0}. As such, it admits a $\bbZ_p$-Lie algebra $\cL^*(N_0) := \cL^*\textnormal{Sat}(\bbZ_p\llbracket N_0\rrbracket)$ in the sense of Lazard \cite[\S\S I.2.2.11, IV.1.3.1]{lazard}. Then, by \cite[Thms. V.2.4.9, V.2.4.10]{lazard}, there is a canonical isomorphism
\begin{equation*}
    H^i_{\textnormal{cts}}\left(N_0,~V\right) \cong  H^i\left(\cL^*(N_0)\otimes_{\bbZ_p}\bbQ_p,~ V\right)^{N_0},
    \end{equation*}
    where the superscript $N_0$ stands for the subspace of $N_0$-invariants. 
    
    Now, since $N_0=\bfN'(\bbZ_p)$ and $(N_0,\omega)$ is $p$-saturated, by \cite[Prop. 4.3.1, Lem. 4.1.3]{HK} there is a natural identification
$$
\cL^*(N_0)\otimes_{\bbZ_p}\bbQ_p \cong \fN' \otimes_{\bbZ_p}\bbQ_p \cong \fN'_{\bbQ_p},
$$
where $\fN' := \textnormal{Lie}(\bfN')(\bbZ_p)$.  Furthermore, this isomorphism is compatible with the natural actions of $N_0$ on $\cL^*(N_0)$ and of $\bfN'$ on $\textnormal{Lie}(\bfN')$; explicitly, it identifies the logarithms attached to $(N_0,\omega)$ with the derivations of $\bfN'$.  In particular, the action of $N_0$ on $H^i(\cL^*(N_0)\otimes_{\bbZ_p}\bbQ_p, V) \cong H^i(\fN'_{\bbQ_p}, V)$ is algebraic (coming from an algebraic action of $\bfN'_{\bbQ_p}$). Thus, if we let $\bfK \subset \bfN'_{\bbQ_p}$ denote the closed subgroup scheme given by the kernel of the $\bfN'_{\bbQ_p}$-action, we see that the kernel for the $N_0$-action is $N_0 \cap \bfK(\bbQ_p)$. On the other hand, by \cite[Thm. V.2.4.10(iii)]{lazard}, this kernel contains an open subgroup of $N_0$ (that is, the $N_0$-action is smooth).  Since such an open subgroup is Zariski dense in $\bfN'_{\bbQ_p}$ (\cite[Lem. 3.2]{platonovrapinchuk}), we conclude that $\bfK = \bfN_{\bbQ_p}'$, which finishes the proof.\footnote{The prototype of this argument goes back to \cite[\S 3]{CW74}, as noted in \cite[Rem. 5.1.3]{HK}.}
\end{proof}

\begin{Rem*}
The proof of \cite[Prop. 4.3.1]{HK} shows that we even have an isomorphism at integral level:
$$\cL^*(N_0) \cong \textnormal{Lie}(\bfN')(\bbZ_p) = \fN'.$$
This isomorphism is moreover compatible with the actions of $N_0$ on $\cL^*(N_0)$ and of $\bfN$ on $\textnormal{Lie}(\bfN')$.
\end{Rem*}

\begin{Lem*}
    \label{lie-leq-grp}
    Suppose Assumption \ref{assn1} holds, and let $\lambda \in X^*(\underline{\bfT})_+$ be $p$-small.  Then
    $$\dim_{k_F}\left(H^i(\fn|_{\bbF_p}\otimes_{\bbF_p}k_F,~L(\lambda))\right) \leq \dim_{k_F}\left(H^i(N_0,~L(\lambda))\right).$$
\end{Lem*}

\begin{proof}
    Our argument is similar to \cite[Pfs. of Thms. 1, 4]{R19b}.

Let $\fN := \textnormal{Lie}(\bfN)(\cO_F)$ denote the $\cO_F$-Lie algebra of $\bfN$, so that $\fN \otimes_{\cO_F}k_F \cong \fn$.  We let $\fN|_{\bbZ_p}$ denote the Lie algebra $\fN$ considered as a $\bbZ_p$-Lie algebra; using \cite[Cor. A.7.6]{CGP}, we have 
    $$\fN|_{\bbZ_p} \cong \textnormal{Lie}(\bfN)(\cO_F)|_{\bbZ_p} \cong \textnormal{Lie}(\bfN')(\bbZ_p),$$
where $\bfN' := \textnormal{Res}_{\cO_F/\bbZ_p}(\bfN)$.  Thus, we have $\fN|_{\bbZ_p}\otimes_{\bbZ_p}k_F \cong \fn|_{\bbF_p}\otimes_{\bbF_p}k_F$.  

We first claim that the algebraic $\underline{\bfG}_{k_F}$-representation $L(\lambda)$ lifts to characteristic 0, that is, that there exists an algebraic representation of $\bfG'_{\cO_F}$ on a finite free $\cO_F$-module whose mod $p$ reduction is $L(\lambda)$.  In fact, this was already constructed in the proof of Lemma \ref{FpQpdims}: we may take $H^0(\lambda) = \ind_{\bfB'^-_{\cO_F}}^{\bfG'_{\cO_F}}(\cO_F(\lambda))$ as our lift.

Next, we claim that the $\cO_F$-linear cohomology $H^i(\fN|_{\bbZ_p} \otimes_{\bbZ_p}\cO_F , H^0(\lambda))$ is free of finite rank over $\cO_F$.  Since $H^i(\fN|_{\bbZ_p}\otimes_{\bbZ_p}\cO_F, H^0(\lambda))$ is finitely generated over $\cO_F$, this is equivalent to showing that the cohomology is $p$-torsion free.  We proceed by descending induction.  The claim is trivially true when $i > \textnormal{rk}_{\cO_F}(\fN|_{\bbZ_p}\otimes_{\bbZ_p}\cO_F) = |\underline{\Phi}^+ - \underline{\Phi}^+_J|$ by the existence of the Chevalley--Eilenberg complex. Suppose that the claim is valid for all degrees greater than $i$.  Then, using the universal coefficient theorem, 
%and the proof of Lemma \ref{FpQpdims}, 
we have 
    \begin{eqnarray*}
        H^i\left(\fN|_{\bbZ_p}\otimes_{\bbZ_p} k_F,~L(\lambda)\right) & \cong & H^i\left(\fN|_{\bbZ_p}\otimes_{\bbZ_p}\cO_F,~H^0(\lambda)\right) \otimes_{\cO_F} k_F \\
        & & ~ \oplus~ \textnormal{Tor}_1^{\cO_F}\left(H^{i + 1}\left(\fN|_{\bbZ_p}\otimes_{\bbZ_p}\cO_F,~H^0(\lambda)\right),~k_F\right) \\
 & \cong & H^i\left(\fN|_{\bbZ_p}\otimes_{\bbZ_p}\cO_F,~H^0(\lambda)\right) \otimes_{\cO_F} k_F \\
  & & \\
        H^i\left(\fN|_{\bbZ_p}\otimes_{\bbZ_p} F,~L_{F}(\lambda)\right) & \cong & H^i\left(\fN|_{\bbZ_p}\otimes_{\bbZ_p}\cO_F,~H^0(\lambda)\right) \otimes_{\cO_F} F \\
        & & ~ \oplus~ \textnormal{Tor}_1^{\cO_F}\left(H^{i + 1}\left(\fN|_{\bbZ_p}\otimes_{\bbZ_p}\cO_F,~H^0(\lambda)\right),~F\right) \\
 & \cong & H^i\left(\fN|_{\bbZ_p}\otimes_{\bbZ_p} \cO_F,~H^0(\lambda)\right) \otimes_{\cO_F} F,
    \end{eqnarray*}
    where on the left-hand side we consider the $k_F$-linear (resp., $F$-linear) cohomology of the $k_F$-Lie algebra $\fN|_{\bbZ_p}\otimes_{\bbZ_p} k_F$ (resp., the $F$-Lie algebra $\fN|_{\bbZ_p}\otimes_{\bbZ_p} F$). 
    %Further, the final isomorphism for each choice of coefficients follows from the inductive assumption.  
    Thus, in order to verify the claim about torsion-freeness, it suffices to show that 
    \begin{equation}
    \label{FpQpLiealg}
        \dim_{k_F}\left(H^i(\fN|_{\bbZ_p}\otimes_{\bbZ_p} k_F,~L(\lambda))\right) = \dim_{F}\left(H^i(\fN|_{\bbZ_p}\otimes_{\bbZ_p} F,~L_F(\lambda))\right).
    \end{equation}
    The above equality essentially follows from \cite[Thm. 4.2.1]{ugavigre}, using an algebraic isomorphism $\overline{F} \cong \bbC$.  Indeed, by using Weil restrictions and proceeding as in the proof of Theorem \ref{kostant}, along with the universal coefficient theorem to base change to algebraically closed fields, it suffices to verify that $\dim_{k_F}(L_{J}(w\cdot \lambda)) = \dim_{F}(L_{J,F}(w\cdot \lambda))$.  This follows from Remark \ref{FpQpdims-remark}.

    Next, we consider the $F$-linear cohomology $H^i(\fN|_{\bbZ_p}\otimes_{\bbZ_p} F, L_F(\lambda))$.  By adjunction, this cohomology group identifies with the $\bbQ_p$-linear cohomology $H^i(\fN|_{\bbZ_p} \otimes_{\bbZ_p}\bbQ_p, L_F(\lambda))$, where we view $L_F(\lambda)$ as a $\bbQ_p$-vector space by restriction.  Thus, Lemma \ref{compn:cts-to-lie-alg-Qp} applies, giving an isomorphism
    \begin{equation}
      \label{compn:lie-alg-to-cts}
    H^i\left(\fN|_{\bbZ_p}\otimes_{\bbZ_p}\bbQ_p,~ L_F(\lambda)\right) \cong H^i_{\textnormal{cts}}\left(N_0,~L_F(\lambda)\right).
    \end{equation}
    
    %Since $\bfN'$ is a smooth group scheme over $\bbZ_p$ with connected generic fiber, \cite[Rmk. 4.3.7]{HKN} implies that we have an isomorphism
    %\begin{equation}
    %\label{compn:lie-alg-to-loc-an}
    %H^i\left(\fN|_{\bbZ_p}\otimes_{\bbZ_p}\bbQ_p,~ L_F(\lambda)\right) \cong H^i_{\textnormal{loc.an.}}\left(N_0,~ L_F(\lambda)\right),
    %\end{equation}
    %where the right-hand side denotes the locally analytic group cohomology (defined using a bar complex consisting of locally analytic functions).  According to \cite[Thm. V.2.3.10]{lazard}, the locally analytic cochains form a subcomplex of the continuous cochains, and the inclusion induces an isomorphism
   % \begin{equation}
    %\label{compn:loc-an-to-cts}
    %H^i_{\textnormal{loc.an.}}\left(N_0,~L_F(\lambda)\right) \cong H^i_{\textnormal{cts}}\left(N_0,~L_F(\lambda)\right)
    %\end{equation}
    %(note that $L_F(\lambda)$ is a ``$\bbZ_p$-module without torsion of finite rank'' in the terminology of \cite{lazard}, cf. \cite[\S V.2.3.2]{lazard}).

    Now let $H^i_{\textnormal{cts}}(N_0,H^0(\lambda))$ denote the continous cohomology of $H^0(\lambda)$ (viewed as a $\bbZ_p$-module).  This cohomology group may be calculated using a quasi-minimal complex \cite[\S V.2.2.2, Eqn. V.2.2.3.1]{lazard}.  Therefore, two more applications of the universal coefficient theorem give
    \begin{eqnarray}
        H^i_{\textnormal{cts}}(N_0,~L_F(\lambda)) & \cong & H^i_{\textnormal{cts}}\left(N_0,~H^0(\lambda)\right)\otimes_{\bbZ_p} \bbQ_p \label{compn:cts-to-Zp}\\
        H^i\left(N_0,~L(\lambda)\right) & \cong & H^i_{\textnormal{cts}}\left(N_0,~H^0(\lambda)\right)\otimes_{\bbZ_p} \bbF_p ~ \oplus~ \textnormal{Tor}_1^{\bbZ_p}\left(H^{i + 1}_{\textnormal{cts}}(N_0,~H^0(\lambda)),~\bbF_p\right). \label{compn:Zp-to-Fp}
    \end{eqnarray}

    Finally, we combine the above isomorphisms, noting that the $\bbQ_p$-vector spaces in \eqref{compn:lie-alg-to-cts}
    %\eqref{compn:lie-alg-to-loc-an}, \eqref{compn:loc-an-to-cts}, and \eqref{compn:cts-to-Zp} 
    (resp., the $\bbF_p$-vector spaces in \eqref{compn:Zp-to-Fp}) naturally have an $F$-linear structure (resp., a $k_F$-linear structure).  Thus, we obtain
        \begin{eqnarray*}
        \dim_{k_F}\left(H^i(\fn|_{\bbF_p}\otimes_{\bbF_p}k_F,~L(\lambda))\right) & = & \dim_{k_F}\left(H^i(\fN|_{\bbZ_p}\otimes_{\bbZ_p} k_F,~L(\lambda))\right) \\
        & \stackrel{\eqref{FpQpLiealg}}{=} & \dim_{F}\left(H^i(\fN|_{\bbZ_p}\otimes_{\bbZ_p} F,~L_F(\lambda))\right) \\
       % & \stackrel{\eqref{compn:lie-alg-to-loc-an}}{=} & \dim_F\left(H^i_{\textnormal{loc.an.}}(N_0,~L_F(\lambda))\right)\\
         & \stackrel{\eqref{compn:lie-alg-to-cts}}{=}  & \dim_{F}\left(H^i_{\textnormal{cts}}(N_0,~L_F(\lambda))\right) \\
         & \leq & \dim_{k_F}\left(H^i(N_0,~L(\lambda))\right),
    \end{eqnarray*}
    where the last inequality follows from equations \eqref{compn:cts-to-Zp} and \eqref{compn:Zp-to-Fp}.  
\end{proof}

We may now prove our main result on the mod $p$ cohomology of $N_0$.

\begin{Th*}
    \label{groupcoh}
    Suppose Assumption \ref{assn1} holds.  Let $\lambda \in X^*(\underline{\bfT})_+$ be $p$-small, and fix a subset $J \subset \Delta$ with associated parabolic subgroup $\bfP = \bfM \bfN$.

	We have a $M_0$-equivariant convergent spectral sequence
        \begin{equation}
        \label{liealgss}
        E_1^{i,j} = \textnormal{gr}_i\left(H^{i + j}(\fn_\omega\otimes_{\bbF_p}k_F,~\gr(L(\lambda)))\right) \Longrightarrow H^{i + j}\left(N_0,~L(\lambda)\right),    
        \end{equation}
        which collapses on the first page.  (We renormalize the gradings on $\fn_\omega$ and $\gr(L(\lambda))$ by multiplying by $h'$, so that they are $\bbZ$-valued.  The grading on the Lie algebra cohomology then comes from the grading on the Chevalley--Eilenberg complex of $\fn_\omega$ and the grading on $\gr(L(\lambda))$.)

	Consequently, we obtain a $M_0$-equivariant isomorphism
        $$\bigoplus_{i \in \bbZ} \textnormal{gr}_i\left(H^n(N_0,~L(\lambda))\right) \cong \bigoplus_{\substack{w \in {}^JW \\ \ell(w) = n}} L_J(w\cdot \lambda),$$
        where the notation is as in Theorem \ref{kostant} and where the grading on the left-hand side comes from the spectral sequence \eqref{liealgss}.  

\end{Th*}

\begin{proof}
The existence of the spectral sequence follows from \cite[Thm. 5.5, Prop. 6.2]{sorensen:hochschild}, while its $M_0$-equivariance follows readily from the constructions in \cite[\S\S 2 -- 5]{sorensen:hochschild}.  We must verify that it collapses on the first page.  By definition of convergence of spectral sequences, we have 
        $$\dim_{k_F}\left(H^n(\fn_\omega \otimes_{\bbF_p}k_F, \gr(L(\lambda)))\right) \geq \dim_{k_F}\left(H^n(N_0,~L(\lambda))\right).$$  
        On the other hand, Lemmas \ref{lazard-vs-alg-lie} and \ref{lie-leq-grp} imply that 
        \begin{eqnarray*}
        \dim_{k_F}\left(H^n(\fn_\omega \otimes_{\bbF_p}k_F, \gr(L(\lambda)))\right) & = & \dim_{k_F}\left(H^n(\fn|_{\bbF_p}\otimes_{\bbF_p}k_F,~L(\lambda))\right) \\
        &  \leq & \dim_{k_F}(H^n(N_0,~L(\lambda))).
        \end{eqnarray*}
        Thus the dimensions are equal, which forces all the differentials on the $E_1$ page to vanish.

	The degeneration of the spectral sequence \eqref{liealgss} and Corollary \ref{kostant-cor} imply that $H^n(N_0,L(\lambda))$ has an $M_0$-stable filtration for which
       \begin{eqnarray*}
        \bigoplus_{i \in \bbZ}\textnormal{gr}_i\left(H^n(N_0,~L(\lambda))\right) & \cong &  \bigoplus_{i \in \bbZ}\textnormal{gr}_i\left(H^n(\fn_\omega\otimes_{\bbF_p}k_F,~\gr(L(\lambda)))\right) \\
        & \cong & H^n\left(\fn_\omega\otimes_{\bbF_p}k_F,~\gr(L(\lambda))\right)\\
        & \cong & \bigoplus_{\substack{w \in {}^JW \\ \ell(w) = n}}L_J(w\cdot \lambda)
        \end{eqnarray*}
        as $M_0$-representations, where the second equivalence comes from the fact that $H^n(\fn_\omega\otimes_{\bbF_p}k_F,\gr(L(\lambda)))$ is semisimple as an $M_0$-representation (recalling Lemma \ref{p-small-J} and Corollary \ref{algrepscor}\ref{algrepscor-b}).  
       
\end{proof}

\subsection{Splitting of the cohomology complex}\label{split:subsection}

In this subsection, we assume the coefficient field $k$ contains the residue field $k_F$ of $F$, and tacitly base-change all representations to $k$.  An application of the universal coefficient theorem shows that the results of the previous section hold with coefficients in $k$.  Our goal will be to examine the bounded complex $\nR H^0(N_0,L(\lambda)) \in \nD(M_0)$; in particular, we give sufficient conditions for this complex to split as a direct sum of its shifted cohomology objects.  

\begin{Pt*}
Fix a standard Levi subgroup $\bfM$ of $\bfG$ with corresponding subset $J \subset \Delta$, and let $\bfC_{\bfM}$ denote the connected center of $\bfM$.  Set $C_{M,0} := \bfC_{\bfM}(\cO_F)$.  We shall be interested in the action of $C_{M,0}$ on the cohomology spaces $H^n(N_0,L(\lambda))$.  To this end, we make the following assumption.  

\begin{Assn*}
\label{assn-centralchar}
If $v,w \in {}^JW$ satisfy $\ell(v) \neq \ell(w)$, then $C_{M,0}$ acts by distinct characters on $L_J(v\cdot \lambda)$ and $L_J(w\cdot \lambda)$.
\end{Assn*}

We remark that this assumption will not be satisfied in general.  (See Remark \ref{dropconds}.)

Below, we will give several sufficient conditions (in terms of $\bfM$ and $\lambda$) for Assumption \ref{assn-centralchar} to hold.  We first record the consequence relevant for our purposes.
\end{Pt*}

\begin{Lem*}
\label{coh-split}
Suppose Assumptions \ref{assn1} and \ref{assn-centralchar} hold.  Then we have 
$$\nR H^0(N_0,L(\lambda)) \cong \bigoplus_{n = 0}^{\dim(N_0)}H^n(N_0, L(\lambda))[-n]$$
in $\nD(M_0)$.  
\end{Lem*}

\begin{proof}
Assumption \ref{assn-centralchar} guarantees that $\Ext_{M_0}^i(L_J(v \cdot \lambda), L_J(w\cdot \lambda)) = 0$ for all $i \in \bbZ$, and all $v, w \in {}^JW$ satisfying $\ell(v) \neq \ell(w)$.  By d\'evissage, Theorem \ref{groupcoh} then implies 
$$\Ext_{M_0}^i\big(H^m(N_0, L(\lambda)),~ H^n(N_0, L(\lambda))\big) = 0$$
for all $i \in \bbZ$ and all $m \neq n$, and thus Corollary \ref{split-cor-a} gives the desired result.
\end{proof}

We now give some conditions for Assumption \ref{assn-centralchar} to hold.

\begin{Lem*}
\label{split:M=T}
Suppose that:
\begin{itemize}[$\diamond$]
\item Assumption \ref{assn1} holds;
\item the center of $\bfG$ is connected;
\item $\bfM = \bfT$ (so that ${}^JW = W$);
\item $\lambda \in X^*(\underline{\bfT})_+$ is $p$-small and for every $\alpha \in \Phi^+$, there exists $\varsigma_0:k_F \longrightarrow k_F$ such that $\langle \lambda + \rho, \alpha_{\varsigma_0}^\vee\rangle \neq p - 1$ (see Subsection \ref{N_0'-sect} for the definition of $\alpha_{\varsigma_0}$).
\end{itemize}
Then Assumption \ref{assn-centralchar} holds.  In particular, we have a splitting
$$\nR H^0(U_0,L(\lambda)) \cong \bigoplus_{n = 0}^{\dim(U_0)}H^n(U_0, L(\lambda))[-n]$$
in $\textnormal{D}(T_0)$.  Moreover, each $H^n(U_0,L(\lambda))$ splits as 
$$H^n(U_0,L(\lambda)) \cong \bigoplus_{\substack{w \in W \\ \ell(w) = n}}k(w\cdot \lambda)$$
in $\textnormal{D}(T_0)^\heartsuit$, and $\nR H^0(U_0,L(\lambda))$ is multiplicity-free.  
\end{Lem*}

\begin{proof}
We will prove the stronger assertion that if $v,w \in W$ satisfy $v \neq w$, then $T_0$ acts by distinct characters on $L_\emptyset(v\cdot \lambda) = k(v\cdot \lambda)$ and $L_{\emptyset}(w\cdot \lambda) = k(w\cdot \lambda)$.  This will imply the stronger claimed splitting statement.

In order to prove the claim, it suffices to show that the stabilizer in $W$ of the (finite-order) character $(\lambda + \rho)|_{T_0}$ is trivial.  According to \cite[Thm. 5.13]{delignelusztig}, this stabilizer is generated by reflections $s_{\alpha_\varsigma}$ ($\alpha \in \Phi^+, \varsigma:k_F \longrightarrow k_F$) such that ``$\alpha_{\varsigma}^\vee$ is orthogonal to the character $(\lambda + \rho)|_{T_0}$.''  This condition is described in \cite[\S 5.9]{delignelusztig}; in particular, this condition holds if and only if 
$$\prod_{\varsigma':k_F \rightarrow k_F} (\varsigma\circ\varsigma')\left(\zeta^{\langle\lambda + \rho, \alpha_{\varsigma'}^\vee\rangle}\right) = 1,$$
where $\zeta$ denotes a fixed generator of $k_F^\times$.  However, the fourth assumption implies that the left-hand side cannot be equal to 1.  Therefore, the stabilizer of $(\lambda + \rho)|_{T_0}$ is trivial.  
\end{proof}

\begin{Rem*}
The fourth assumption of the lemma above can be relaxed somewhat.  Namely, it suffices to check the condition $\langle \lambda + \rho, \alpha_{\varsigma_0}^\vee\rangle \neq p - 1$ only for those $\alpha$ such that $\alpha^\vee$ is either the highest or a second-highest coroot in the irreducible component of $\Phi^\vee$ to which it belongs.  To justify this claim, we show that if $\alpha^\vee$ is not the highest or a second-highest coroot, then $\langle \lambda + \rho, \alpha_\varsigma^\vee\rangle \neq p - 1$ for all $\varsigma : k_F \longrightarrow k_F$.  Suppose by contraposition that $\langle \lambda + \rho, \alpha_\varsigma^\vee\rangle = p - 1$ for some $\varsigma$, and let $\alpha^\vee_0$ denote the highest coroot of $\Phi^\vee_{\alpha^\vee}$, the irreducible component of $\Phi^\vee$ to which $\alpha^\vee$ belongs.  Write 
$$\alpha^\vee = \sum_{\beta^\vee \in \Phi^\vee_{\alpha^\vee} \cap \Delta^\vee}n_{\beta^\vee} \beta^\vee, \qquad \alpha^\vee_{0} = \sum_{\beta^\vee \in \Phi^\vee_{\alpha^\vee} \cap \Delta^\vee}m_{\beta^\vee} \beta^\vee,$$
where $n_{\beta^\vee}, m_{\beta^\vee} \in \bbZ_{\geq 0}$.  By \cite[Ch. VI, \S 1.8, Prop. 25(i)]{bourbaki:Lie4-6}, we have $m_{\beta^\vee} \geq n_{\beta^\vee}$ for all $\beta^\vee \in \Phi^\vee_{\alpha^\vee} \cap \Delta^\vee$.  Therefore, by dominance and $p$-smallness of $\lambda$, we have
\begin{eqnarray*}
\hgt(\alpha_{0}^\vee) - \hgt(\alpha^\vee) & = & \sum_{\beta^\vee \in \Phi^\vee_{\alpha^\vee}\cap \Delta^\vee} (m_{\beta^\vee} - n_{\beta^\vee}) \\
& \leq & \sum_{\beta^\vee \in \Phi^\vee_{\alpha^\vee}\cap \Delta^\vee} (m_{\beta^\vee} - n_{\beta^\vee})\langle \lambda + \rho, \beta_\varsigma^\vee\rangle \\
 &  = & \langle \lambda + \rho, \alpha_{0,\varsigma}^\vee - \alpha^\vee_\varsigma\rangle \\
 & \leq &  p - (p - 1) \\
 & = & 1.
\end{eqnarray*}
We conclude that $\hgt(\alpha^\vee) \geq \hgt(\alpha^\vee_0) - 1$.  
\end{Rem*}
 
 \begin{Rem*}
 \label{dropconds}
 If the connectedness assumption in Lemma \ref{split:M=T} is omitted, then Assumption \ref{assn-centralchar} may fail to hold.  For example, consider the case where $\bfG = \nS\nL_{2/\bbQ_p}$ with $p \geq 5$, $\lambda = \frac{p - 3}{2} \in \bbZ \cong X^*(\bfT)$, and $\bfM = \bfT$.  Then all but the second assumption of Lemma \ref{split:M=T} are satisfied.  We have $\bfC_{\bfM} = \bfT$, and the actions of $T_0 \cong \bbZ_p^\times$ on $H^0(U_0,L(\lambda)) = L_\emptyset(\lambda) = k(\lambda)$ and $H^1(U_0,L(\lambda)) = L_\emptyset(s\cdot \lambda) = k(s\cdot\lambda)$ agree (where $s$ denotes the nontrivial element of the Weyl group).  Thus, Assumption \ref{assn-centralchar} does not hold.

 Similarly, if the fourth assumption of the lemma is omitted, then Assumption \ref{assn-centralchar} may fail to hold.  In this case, take $\bfG = \nG\nL_{2/\bbQ_p}$ with $p \geq 5$, $\lambda = (p - 2,0) \in \bbZ^{\oplus 2} \cong X^*(\bfT)$, and $\bfM = \bfT$.  Then all but the fourth assumption of Lemma \ref{split:M=T} are satisfied.  We have $\bfC_{\bfM} = \bfT$, and the actions of $T_0 \cong (\bbZ_p^\times)^2$ on $H^0(U_0, L(\lambda)) = L_\emptyset(\lambda) = k(\lambda)$ and $H^1(U_0, L(\lambda)) = L_\emptyset(s\cdot \lambda) = k(s\cdot \lambda)$ agree.

 \end{Rem*}

 When $\lambda = 0$, we can drop the connectedness assumption in Lemma \ref{split:M=T}.
 
 \begin{Lem*}
 \label{split:M=T,lambda=0}
 Suppose that:
\begin{itemize}[$\diamond$]
\item Assumption \ref{assn1} holds;
\item $\bfM = \bfT$ (so that ${}^JW = W$);
\item $\lambda = 0$.
\end{itemize}
Then Assumption \ref{assn-centralchar} holds.  In particular, we have a splitting
$$\nR H^0(U_0,k) \cong \bigoplus_{n = 0}^{\dim(U_0)}H^n(U_0, k)[-n]$$
in $\textnormal{D}(T_0)$.  Moreover, each $H^n(U_0,k)$ splits as 
$$H^n(U_0,k) \cong \bigoplus_{\substack{w \in W \\ \ell(w) = n}}k(w\cdot 0)$$
in $\textnormal{D}(T_0)^\heartsuit$, and $\nR H^0(U_0,k)$ is multiplicity-free.  
 \end{Lem*}

 \begin{proof}
 We will prove that if $w \in W$ is nontrivial, then the (finite-order) character $(w\cdot 0)|_{T_0} = (w(\rho) - \rho)|_{T_0}$ is nontrivial. Suppose by contradiction that $(w\cdot 0)|_{T_0}$ is trivial, and let $\beta \in \Delta$ (without an underline).  Then, for each $x \in k_F^\times$, we have
 \begin{equation}
 \label{split:M=T,lambda=0,eqn1}
1 = (w\cdot 0)\left(\beta^\vee(x^{-1})\right) = \prod_{\varsigma:k_F \rightarrow k_F}\varsigma\left(x^{\langle-w(\rho) + \rho, \beta_\varsigma^\vee\rangle}\right).
\end{equation}
Next, note that $\langle-w(\rho) + \rho, \beta_\varsigma^\vee\rangle = \langle\rho, -w^{-1}(\beta_\varsigma^\vee) + \beta_\varsigma^\vee\rangle = \hgt(-w^{-1}(\beta_\varsigma^\vee)) + 1$.  Thus, we get
\begin{equation}
\label{split:M=T,lambda=0,bounds}
-(p - 1) < -h + 2 \leq \langle-w(\rho) + \rho, \beta_\varsigma^\vee\rangle \leq h < p - 1.
 \end{equation}
Combining \eqref{split:M=T,lambda=0,eqn1} and \eqref{split:M=T,lambda=0,bounds}, we deduce
$\langle-w(\rho) + \rho, \beta_\varsigma^\vee\rangle = \hgt(-w^{-1}(\beta_\varsigma^\vee)) + 1 =  0$ for all $\varsigma$.

Now choose $\beta\in \Delta$ for which there exists $\varsigma_0:k_F \longrightarrow k_F$ satisfying $w^{-1}(\beta_{\varsigma_0}^\vee) \in \underline{\Phi}^{-,\vee}$ (such a choice is possible since $w \neq 1$).  The previous paragraph then implies $\hgt(w^{-1}(\beta_{\varsigma_0}^\vee)) = 1$, and we therefore arrive at a contradiction.
\end{proof}

\begin{Rem*}\label{ramified}
    If the ramification condition of Assumption \ref{assn1} in Lemma \ref{split:M=T,lambda=0} is not satisfied, the representations $H^n(U_0, k)$ may fail to be semisimple and multiplicity-free.  For example, consider the group $\bfG = \nG\nL_{2/\bbQ_p(\sqrt{p})}$ over the ramified extension $F = \bbQ_p(\sqrt{p})$, with $\bfU$ the upper triangular unipotent matrices and $\bfT$ the diagonal matrices.  We have $U_0 \cong \cO_F \cong \bbZ_p^{\oplus 2}$, and one can check that, as a $T_0$-representation, the cohomology $H^1(U_0,\bbF_p)$ is a nonsplit extension of $\alpha^{-1}$ by $\alpha^{-1}$, where 
    $$\alpha\left(\begin{pmatrix}
        a & 0 \\ 0 & d
    \end{pmatrix}\right) = ad^{-1}~\textnormal{mod}~\sqrt{p}\cO_F$$
    for $a,d \in \cO_F^\times$.  
\end{Rem*}

 \begin{Pt*}
We now focus on general Levi subgroups $\bfM$ of $\bfG$ (associated to $J \subset \Delta$) and $\lambda = 0$. Let $\Phi = \Phi^{(1)} \sqcup \ldots \sqcup \Phi^{(r)}$ denote the decomposition of $\Phi$ into irreducible components, and for each $1 \leq j \leq r$, let $\alpha_0^{(j)}$ denote the highest root of $\Phi^{(j)}$ (relative to the basis $\Delta \cap \Phi^{(j)}$).  For $1 \leq j \leq r$, we define $\xi_j \in X_*(\bfT_{k_F})$ to be any fixed element satisfying
\begin{itemize}[$\diamond$]
\item $\langle \alpha, \xi_j \rangle = 0$ for all $\alpha \in \Phi^+ - \Phi^{(j), +}$;
\item $\langle \alpha, \xi_j \rangle = 0$ for all $\alpha \in \Phi^{(j),+} \cap \Phi^+_J$;
\item $\langle \alpha, \xi_j \rangle > 0$ for all $\alpha \in \Phi^{(j),+} - (\Phi^{(j),+} \cap \Phi^+_J)$;
\item $\langle \alpha_0^{(j)}, \xi_j \rangle$ is minimal.
\end{itemize}
In particular, since $\langle \alpha, \xi_j\rangle = 0$ for all $\alpha \in J$, we have $\xi_j \in X_*(\bfC_{\bfM, k_F})$ (\cite[Lem. 8.1.8(ii)]{springer}).  Finally, we define
$$h_\bfM := \max_{1 \leq j \leq r} \left\{\langle \alpha_0^{(j)}, \xi_j\rangle \right\}.$$

In the proof below, we shall use the following notation: given $\xi \in X_*(\bfT_{k_F})$ and $\varsigma: k_F \longrightarrow k_F$, we let 
$$\xi_\varsigma \in X_*(\underline{\bfT}) \cong \bigoplus_{\varsigma':k_F \rightarrow k_F} X_*(\bfT_{k_F} \times_{k_F,\varsigma'}k_F)$$ 
denote the element which is equal to $\xi$ in embedding $\varsigma$ and $0$ otherwise.  
\end{Pt*}

\begin{Lem*}
\label{split:ratl-to-alg}
Suppose $p > |\Phi^+ - \Phi^+_J|h_\bfM + 1$, and let $v,w \in {}^JW$.  If $C_{M,0}$ acts by the same character on $L_J(v\cdot 0)$ and $L_J(w\cdot 0)$, then 
$$\langle v\cdot 0, \xi_{j,\varsigma}\rangle = \langle w\cdot 0, \xi_{j,\varsigma}\rangle$$
for all $1 \leq j \leq r$ and all $\varsigma: k_F \longrightarrow k_F$.
 \end{Lem*}
 
 \begin{proof}
Fix $1 \leq j \leq r$.  For $x \in k_F^\times$, the action of $\xi_j(x^{-1}) \in C_{M,0}$ on $L_J(w\cdot 0)$ is given by
\begin{equation}
\label{Mcentralaction}
(w\cdot 0)\left(\xi_j(x^{-1})\right) = \prod_{\varsigma:k_F \rightarrow k_F}\varsigma\left(x^{\langle - w\cdot 0,\xi_{j,\varsigma}\rangle}\right).
\end{equation}
Next, note that 
$$-w\cdot 0 = - w(\rho) + \rho = \sum_{\substack{\alpha \in \underline{\Phi}^+ \\ w^{-1}(\alpha) \in \underline{\Phi}^-}}\alpha.$$
If $\alpha \in X^*(\underline{\bfT}) \cong \bigoplus_{\varsigma':k_F \rightarrow k_F} X^*(\bfT_{k_F} \times_{k_F,\varsigma'}k_F)$ does not lie in the $\varsigma$-component, then $\langle \alpha, \xi_{j,\varsigma} \rangle = 0$.  On the other hand, if $\alpha$ does lie in the $\varsigma$-component, then we obtain
$$0 \leq \langle \alpha, \xi_{j,\varsigma}\rangle \leq \langle \alpha_{0,\varsigma}^{(j)}, \xi_{j,\varsigma}\rangle \leq h_{\bfM}$$
(see \cite[Ch. VI, \S 1.8, Prop. 25(i)]{bourbaki:Lie4-6}).  Since the number of roots $\alpha \in \underline{\Phi}^+$ which satisfy $w^{-1}(\alpha) \in \underline{\Phi}^-$ and which lie in the $\varsigma$-component is bounded above by $|\Phi^+ - \Phi^+_J|$, we get
\begin{equation}
 \label{Mbound}
 0 \leq \langle -w\cdot 0 , \xi_{j,\varsigma}\rangle = \sum_{\substack{\alpha \in \underline{\Phi}^+ \\ w^{-1}(\alpha) \in \underline{\Phi}^-}} \langle \alpha, \xi_{j,\varsigma} \rangle \leq |\Phi^+ - \Phi^+_J|h_{\bfM} < p - 1.
 \end{equation}

Suppose now that $L_J(v\cdot 0)$ and $L_J(w\cdot 0)$ have the same $C_{M,0}$-action.  We get
$$\prod_{\varsigma:k_F \rightarrow k_F}\varsigma\left(x^{\langle - v\cdot 0,\xi_{j,\varsigma}\rangle}\right) = \prod_{\varsigma:k_F \rightarrow k_F}\varsigma\left(x^{\langle - w\cdot 0,\xi_{j,\varsigma}\rangle}\right)$$
for all $x\in k_F^\times$ by equation \eqref{Mcentralaction}, and equation \eqref{Mbound} implies that
$$0 \leq \langle -v\cdot 0, \xi_{j,\varsigma}\rangle < p - 1, \qquad 0 \leq \langle -w\cdot 0, \xi_{j,\varsigma}\rangle < p - 1$$
for all $\varsigma$.  Therefore, by writing the embeddings $\varsigma$ as powers of the arithmetic Frobenius, we see that the exponent of $x$ in \eqref{Mcentralaction} is an integer between $0$ and $p^{[k_F:\bbF_p]} - 2$.  Taking $x$ to be a generator of $k_F^\times$, we conclude that
$$\langle v\cdot 0, \xi_{j,\varsigma}\rangle = \langle w\cdot 0, \xi_{j,\varsigma}\rangle$$
for all $\varsigma$.  
\end{proof}

 We now apply the above discussion to Assumption \ref{assn-centralchar}.  Recall that ${}^JW(\bfG,\bfT)$ denotes the set of minimal length coset representatives for $W(\bfM,\bfT)$ in $W(\bfG,\bfT)$, and that ${}^JW$ decomposes as a product of copies of ${}^JW(\bfG,\bfT)$ indexed by field isomorphisms $\varsigma:k_F \longrightarrow k_F$.

 \begin{Lem*}
 \label{split:bruhat}
  Let $\bfM$ correspond to $J \subset \Delta$, and suppose that:
 \begin{itemize}[$\diamond$]
 \item Assumption \ref{assn1} holds;
\item $p > |\Phi^+ - \Phi^+_J| h_{\bfM} + 1$;
\item $\lambda = 0$;
\item any two elements of ${}^JW(\bfG, \bfT)$ of distinct length are comparable under the right weak Bruhat order of $W(\bfG,\bfT)$ (see the proof below for the definition of this partial order).
\end{itemize}
Then Assumption \ref{assn-centralchar} holds.
 \end{Lem*}
 
 \begin{Rem*}
As examples of sets ${}^JW(\bfG,\bfT)$ which satisfy the fourth assumption above, we may take: 
\begin{itemize}[$\diamond$]
\item any group $\bfG$ of type $\textnormal{A}_n$, and $J = \{\alpha_1, \ldots, \alpha_{n - 1}\}$ or $J = \{\alpha_2, \ldots, \alpha_n\}$; 
\item any group $\bfG$ of type $\textnormal{C}_n$, and $J =  \{\alpha_2, \ldots, \alpha_n\}$;
\item any group $\bfG$ of type $\textnormal{D}_n$, and $J =  \{\alpha_2, \ldots, \alpha_n\}$;
\item any group $\bfG$ of type $\textnormal{G}_2$, and $J = \{\alpha_1\}$ or $J = \{\alpha_2\}$;
\item etc.
\end{itemize}
(See \cite[Planches I -- IX]{bourbaki:Lie4-6} for the customary numbering of roots.)
\end{Rem*}

\begin{proof}[Proof of Lemma \ref{split:bruhat}]
Recall from \cite[Ch. 3]{bjornerbrenti} that the \textit{right weak order} $\preceq$ on the Coxeter group $W(\bfG,\bfT)$ is characterized by $y \preceq z$ if and only if $\ell(y) + \ell(y^{-1}z) = \ell(z)$.  By \cite[Prop. 3.1.3]{bjornerbrenti}, we have that 
\begin{equation}
\label{weak-right-Phi}
y\preceq z \quad \textnormal{if and only if}\quad \Phi_y \subset \Phi_z,
\end{equation}
where 
$$\Phi_y := \Phi^+ \cap y(\Phi^-).$$

We now apply these considerations to central characters.  Suppose that $v,w \in {}^JW$ have distinct lengths, and suppose by contradiction that $L_J(v\cdot 0)$ and $L_J(w\cdot 0)$ have the same action of $C_{M,0}$.  By Lemma \ref{split:ratl-to-alg}, we have 
\begin{equation}
\label{inn-prod-v-w-equal}
\langle v\cdot 0, \xi_{j,\varsigma}\rangle = \langle w\cdot 0, \xi_{j,\varsigma}\rangle
\end{equation}
for all $1\leq j \leq r$ and all $\varsigma:k_F \longrightarrow k_F$.  Since $v$ and $w$ have distinct lengths, there exists $\varsigma$ for which $v_\varsigma$ and $w_\varsigma$ have distinct lengths (as elements of ${}^JW(\bfG,\bfT)$).  Assume without loss of generality that $v_\varsigma \prec w_\varsigma$.  Using the characterization of right weak order in \eqref{weak-right-Phi}, we may write the character $v_\varsigma\cdot 0 - w_\varsigma\cdot 0$ as
$$v_\varsigma\cdot 0 - w_\varsigma\cdot 0 = -\sum_{\alpha \in \Phi_{v_\varsigma}}\alpha + \sum_{\alpha \in \Phi_{w_\varsigma}}\alpha = \sum_{\alpha \in \Phi_{w_\varsigma} - \Phi_{v_\varsigma}}\alpha.$$
Since $v_\varsigma\prec w_\varsigma$, the sum $\sum_{\alpha \in \Phi_{w_\varsigma} - \Phi_{v_\varsigma}}\alpha$ contains a summand in $\Phi^+ - \Phi^+_J$.  Such a summand pairs positively with $\xi_{j,\varsigma}$ (for an appropriately chosen $j$), from which we deduce $\langle v_\varsigma\cdot 0 - w_\varsigma\cdot 0, \xi_{j,\varsigma}\rangle\neq 0$.  However, this contradicts \eqref{inn-prod-v-w-equal}.
\end{proof}

 \begin{Lem*}
 \label{split:abelian}
  Let $\bfM$ correspond to $J \subset \Delta$, and suppose that:
 \begin{itemize}[$\diamond$]
\item Assumption \ref{assn1} holds;
\item $p > |\Phi^+ - \Phi^+_J| h_{\bfM} + 1$;
\item $\lambda = 0$;
\item $\bfN$ is abelian.
\end{itemize}
Then Assumption \ref{assn-centralchar} holds.
 \end{Lem*}

\begin{proof}
The given assumptions imply that $N_0$ is a (topologically) finitely generated, abelian, pro-$p$ group.  From this, we deduce that $N_0 \cong \bbZ_p^{\oplus \dim(\bfN)[F:\bbQ_p]}$ as topological groups, and therefore $N_0$ is uniform.  Hence, by \cite[Thm. 5.1.5]{symondsweigel}, we obtain an $M_0$-equivariant isomorphism
\begin{equation}
\label{abelian:wedge}
H^i(N_0,~k) \cong \sideset{}{_{k}^i}\bigwedge H^1(N_0,~ k).
\end{equation}
Moreover, by Theorem \ref{groupcoh}, the $M_0$-representation $H^1(N_0,k)$ admits a filtration whose graded pieces are 
$$L_J(-\beta_\varsigma), \qquad \textnormal{where}\qquad \beta \in \Delta - J,~ \varsigma:k_F \longrightarrow k_F.$$
In particular, the characters of $C_{M,0}$ appearing in $H^1(N_0,k)$ are the $(-\beta_\varsigma)|_{C_{M,0}}$ for $\beta$ and $\varsigma$ as above, with multiplicity $\dim_{k}(L_J(-\beta_\varsigma))$.

Now suppose $v,w\in {}^JW$ satisfy $\ell(v)\neq \ell(w)$, and that $C_{M,0}$ acts by the same character on $L_J(v\cdot 0)$ and $L_J(w\cdot 0)$.  Since the character $(v\cdot 0)|_{C_{M,0}}$ appears in $H^{\ell(v)}(N_0, k)$ (and analogously for $w$), the decomposition \eqref{abelian:wedge} implies
$$(-v\cdot 0)|_{C_{M,0}} = \sum_{\varsigma:k_F \rightarrow k_F}\sum_{\beta \in \Delta - J} m_{\beta, \varsigma}\beta_{\varsigma}|_{C_{M,0}}, \qquad (-w\cdot 0)|_{C_{M,0}} = \sum_{\varsigma:k_F \rightarrow k_F}\sum_{\beta \in \Delta - J} n_{\beta, \varsigma}\beta_{\varsigma}|_{C_{M,0}},$$
for some choice of coefficients $m_{\beta,\varsigma},~ n_{\beta,\varsigma} \in \bbZ_{\geq 0}$ satisfying $0 \leq m_{\beta,\varsigma},~ n_{\beta,\varsigma} \leq \dim_{k}(L_J(-\beta_\varsigma))$ and
$$\sum_{\varsigma:k_F \rightarrow k_F}\sum_{\beta \in \Delta - J} m_{\beta, \varsigma} = \ell(v),\qquad \sum_{\varsigma:k_F \rightarrow k_F}\sum_{\beta \in \Delta - J} n_{\beta, \varsigma} = \ell(w).$$
Since $v$ and $w$ were assumed to have distinct lengths, there exists $\varsigma$ for which $v_\varsigma$ and $w_\varsigma$ have distinct lengths.  From this, we deduce
$$\sum_{\beta \in \Delta - J} m_{\beta, \varsigma} = \ell(v_\varsigma) \neq \ell(w_\varsigma) = \sum_{\beta \in \Delta - J} n_{\beta, \varsigma}.$$
Consequently, there exists $\beta' \in \Delta - J$ for which we have $m_{\beta', \varsigma} \neq n_{\beta', \varsigma}$.

Now choose $1 \leq j \leq r$ such that $\beta' \in \Phi^{(j)}$.  By the assumption on central characters and Lemma \ref{split:ratl-to-alg}, we have
$$\sum_{\beta \in \Delta - J}m_{\beta,\varsigma}\langle \beta_{\varsigma}, \xi_{j,\varsigma}\rangle = \langle -v\cdot 0, \xi_{j,\varsigma}\rangle = \langle -w\cdot 0, \xi_{j,\varsigma}\rangle = \sum_{\beta \in \Delta - J}n_{\beta,\varsigma}\langle \beta_{\varsigma}, \xi_{j,\varsigma}\rangle.$$
By \cite[Lem. 2.2]{RRS}, we have $J \cap \Phi^{(j)} = (\Delta \cap \Phi^{(j)}) - \{\beta'\}$, so that by definition of $\xi_j$ the preceding equation reduces to 
$$m_{\beta',\varsigma}\langle \beta'_{\varsigma}, \xi_{j,\varsigma}\rangle = n_{\beta',\varsigma}\langle \beta'_{\varsigma}, \xi_{j,\varsigma}\rangle.$$
However, since $\langle \beta'_{\varsigma}, \xi_{j,\varsigma}\rangle > 0$ this contradicts $m_{\beta', \varsigma} \neq n_{\beta', \varsigma}$.  
\end{proof}

\subsection{Application to principal series}\label{psr:subsection}

We now give an application of the above results to the cohomology of principal series representations.  

Let $k$ be an arbitrary field of characteristic $p$, and let $\chi:T \longrightarrow k^\times$ denote a smooth character.  Such a character is trivial on the maximal pro-$p$ subgroup of $T$.  Thus, if we let $\chi_0$ denote the restriction $\chi|_{T_0}$, then $\chi_0$ factors through the subgroup $\bfT(k_F)$.  We use the notation $\chi_0$ to refer to either the character of $T_0$ or of $\bfT(k_F)$.

We inflate the character $\chi$ to a character of $B$ trivial on $U$, and consider the induced representation $\Ind_B^G(\chi)$.  We are interested in the cohomology spaces $H^i(G_0,\Ind_B^G(\chi))$.

\begin{Prop*}
	\label{prin-series-appn}
    Suppose Assumption \ref{assn1} holds.  Let $T_1$ denote the pro-$p$ Sylow subgroup of $T_0$, and let $\ft_1^* := \Hom_{\textnormal{cts}}(T_1,k)$, which is a $k$-vector space of dimension $[F:\bbQ_p]\textnormal{rk}(\bfG)$.
    \begin{enumerate}[(a)]
        \item\label{prin-series-appn:1} Suppose $k$ contains $k_F$, and assume $\chi_0 \neq -w\cdot 0$ for all $w \in W$.  Then 
        $$H^i\left(G_0,\Ind_B^G(\chi)\right) = 0  \qquad \textnormal{for all $i \geq 0$}.$$
        \item\label{prin-series-appn:2} Suppose $k$ contains $k_F$, and assume $\chi_0 = -w\cdot 0$ for some $w \in W$.  Then 
        $$H^i\left(G_0, \Ind_B^G(\chi)\right) \cong \sideset{}{^{i - \ell(w)}_k}{\bigwedge}\ft_1^*.$$
        In particular, $H^i(G_0, \Ind_B^G(\chi))\neq 0$ if and only if $\ell(w) \leq i \leq \ell(w) + [F:\bbQ_p]\textnormal{rk}(\bfG)$.  
        \item\label{prin-series-appn:3} Suppose $k$ is arbitrary, and assume that $\chi$ is unramified (that is, that $\chi_0$ is the trivial character).  Then 
        $$H^i\left(G_0,\Ind_B^G(\chi)\right) \cong \sideset{}{^{i}_k}{\bigwedge}\ft_1^*.$$
        In particular, $H^i(G_0, \Ind_B^G(\chi))\neq 0$ if and only if $0 \leq i \leq [F:\bbQ_p]\textnormal{rk}(\bfG)$.  
    \end{enumerate}
\end{Prop*}

\begin{proof}
    Combining the Iwasawa decomposition $G = BG_0$ and the Mackey decomposition gives
    $$\Ind_B^G(\chi)|_{G_0} \cong \Ind_{B_0}^{G_0}(\chi_0)$$
    (where we inflate $\chi_0$ to a character of $B_0$ trivial on $U_0$).  An application of Shapiro's lemma then gives
    $$H^i\left(G_0,\Ind_B^G(\chi)\right) \cong  H^i\left(G_0, \Ind_{B_0}^{G_0}(\chi_0)\right) \cong H^i(B_0,\chi_0).$$
    Using the decomposition $B_0 = T_0 \ltimes U_0$, we will calculate the latter cohomology using the Hochschild--Serre spectral sequence:
    \begin{equation}
        \label{hochschildserre}
        E_2^{i,j} = H^i\left(T_0,H^j(U_0,\chi_0)\right) \cong H^i\left(T_0,\chi_0 \otimes_k H^j(U_0,k)\right) \Longrightarrow H^{i + j}(B_0, \chi_0).  
    \end{equation}

    We first examine the $T_0$ cohomology above.  Suppose $V$ is a smooth representation of $T_0$ over $k$ on which $T_1$ acts trivially.  The Hochschild--Serre spectral sequence associated to the normal subgroup $T_1$ then gives 
    $$E_2^{i,j} = H^i\left(\bfT(k_F),H^j(T_1,V)\right) \cong H^i\left(\bfT(k_F),V \otimes_k H^j(T_1,k)\right) \Longrightarrow H^{i + j}(T_0,V).$$
    Since $\bfT(k_F)$ is finite of order coprime to $p$, the spectral sequence collapses to give
    $$\left(V \otimes_k H^j(T_1,k)\right)^{\bfT(k_F)} \cong H^j(T_0,V).$$
    Note furthermore that the action of $\bfT(k_F)$ on $H^j(T_1,k)$ is trivial, from which we get
    \begin{equation}
        \label{T0coh}
        V^{\bfT(k_F)}\otimes_k H^j(T_1,k) \cong H^j(T_0,V).
    \end{equation}
    Thus $H^j(T_0,V)\neq 0$ if and only if $V^{\bfT(k_F)}\neq 0$ and $H^j(T_1,k) \neq 0$.

    We now apply the previous paragraph to the spectral sequence \eqref{hochschildserre}:
    \begin{enumerate}[(a)]
        \item Suppose first that $k$ contains $k_F$, and that $\chi_0 \neq -w\cdot 0$ for all $w\in W$.  By Lemma \ref{split:M=T,lambda=0}, for all $j \geq 0$, the semisimple representation $\chi_0\otimes_k H^j(U_0, k)$ does not contain the trivial character of $\bfT(k_F)$.  Consequently, equation \eqref{T0coh} applied to $V = \chi_0\otimes_k H^j(U_0, k)$ implies $H^i(T_0,\chi_0 \otimes_k H^j(U_0,k)) = 0$ for all $i$ and $j$, which gives part (a) by invoking the spectral sequence \eqref{hochschildserre}.
        \item Suppose next that $k$ contains $k_F$, and that $\chi_0 = -w\cdot 0$ for some fixed $w\in W$.  By the multiplicity-freeness result in Lemma \ref{split:M=T,lambda=0}, we obtain that 
        $$\left(\chi_0\otimes_k H^j(U_0,k)\right)^{\bfT(k_F)} \neq 0 \qquad \textnormal{if and only if} \qquad j = \ell(w),$$ 
        in which case the dimension is 1.  Applying equation \eqref{T0coh} to $V = \chi_0\otimes_k H^j(U_0, k)$ shows that 
        $$H^i\left(T_0,\chi_0\otimes_k H^j(U_0,k)\right) \neq 0 \qquad \textnormal{implies}\qquad j = \ell(w).$$
        Therefore, the spectral sequence \eqref{hochschildserre} collapses to give
        $$H^i(B_0,\chi_0) \cong H^{i - \ell(w)}\left(T_0,\chi_0\otimes_kH^{\ell(w)}(U_0,k)\right) \cong H^{i - \ell(w)}(T_1,k).$$
        \item Finally, suppose that $k$ is arbitrary, and that $\chi_0$ is trivial.  A straightforward argument using the universal coefficient theorem, Lemma \ref{split:M=T,lambda=0}, and \cite[\S 12.1, Prop. 1]{bourbaki:algch8} shows that $\bigoplus_{n \geq 0}H^n(U_0, k)$ is multiplicity-free and semisimple.  Since the trivial $T_0$-character appears in $H^0(U_0,k)$, we deduce
        $$H^j(U_0,k)^{\bfT(k_F)} \neq 0 \qquad \textnormal{if and only if} \qquad j = 0,$$ 
        in which case the dimension is 1.  Applying equation \eqref{T0coh} to $V = H^j(U_0, k)$ shows that
        $$H^i\left(T_0,H^j(U_0,k)\right) \neq 0 \qquad \textnormal{implies}\qquad j = 0.$$
        Therefore, the spectral sequence \eqref{hochschildserre} collapses to give
        $$H^i(B_0,\chi_0) \cong H^{i}\left(T_0,H^{0}(U_0,k)\right) \cong H^{i}(T_1,k).$$
    \end{enumerate}

    To finish the proof, it suffices to calculate $H^i(T_1,k)$ for an arbitrary field $k$ of characteristic $p$.  The group $T_1$ is isomorphic to $(1 + p\cO_F)^{\textnormal{rk}(\bfG)}$; since $1 + p\cO_F$ is torsion-free (this follows from Assumption \ref{assn1}), we get that 
    $$T_1 \cong \bbZ_p^{\oplus [F:\bbQ_p]\textnormal{rk}(\bfG)}.$$
    In particular, $T_1$ is a uniform pro-$p$ group of dimension $[F:\bbQ_p]\textnormal{rk}(\bfG)$.  By a result of Lazard (see \cite[Thm. 5.1.5]{symondsweigel}) we have 
    $$H^i(T_1,k) \cong \sideset{}{^i_{k}}{\bigwedge}\ft_1^*.$$
    This finishes the proof.
\end{proof}

\section{The left adjoint of parabolic induction and orthogonal decompositions} \label{LN_0:section}

In this section, we use the results of the previous section to describe the left adjoint of parabolic induction on certain compactly induced representations.

We keep the assumptions of the previous section.  In particular, Assumption \ref{assn1} remains in effect, and we let $J \subset \Delta$ denote a subset with corresponding parabolic subgroup $\bfP = \bfM\bfN$.  We fix a $p$-small character $\lambda \in X^*(\underline{\bfT})_+$ throughout.

\subsection{The left adjoint at compact level}

\bigskip

In \cite{H22}, Heyer constructs and studies the left adjoint $L(N,-)$ of the parabolic induction functor $\Ind_{P}^{G}:\nD(M)\lra \nD(G)$, for an arbitrary coefficient field $k$ of characteristic $p$. In particular, at compact level, the following holds:

\begin{Th*} \emph{\textbf{(Heyer, \cite{H22})}} 
\label{Heyer}
\begin{enumerate}[(a)]
\item\label{Heyer-a} Let $\cK\subset G$ be a compact and torsion-free (in particular, pro-$p$) subgroup, and assume that $P \cap \cK = (M \cap \cK)(N \cap \cK)$. Then the parabolic induction functor $\Ind_{P\cap \cK}^{\cK}:\nD(M\cap \cK)\lra \nD(\cK)$ admits a left adjoint $L(N\cap \cK,-)$; moreover
$$
L(N\cap \cK,-) = \textnormal{R}H^0\left(N\cap \cK,-\right)[\dim(N\cap \cK)].
$$
\item\label{Heyer-b} The parabolic induction functor $\Ind_{P_0}^{G_0}:\nD(M_0)\lra \nD(G_0)$ admits a left adjoint $L(N_0,-)$.
\end{enumerate}
\end{Th*}

\begin{proof}
Point \ref{Heyer-a} is \cite[Cor. 3.1.8, Prop. 3.1.10]{H22}, while point \ref{Heyer-b} follows from \cite[Thm. 3.2.3]{H22}.
\end{proof}

In point \ref{Heyer-b}, the subgroup $G_0 \subset G$ is compact but not necessarily torsion-free.  However, for this subgroup, we will adapt the proof of point \ref{Heyer-a} and use Theorem \ref{groupcoh} to prove the following:

\begin{Prop*} \label{formula-spherical}
Assume Assumption \ref{assn1} holds. Then
$$
L(N_0,-) = \textnormal{R}H^0\left(N_0,-\right)[\dim(N_0)] \otimes_k (N_{k_F/\bbF_p}(\overline{2\rho}\cdot \overline{2\rho_M}^{-1})).
$$
\end{Prop*}

In order to prove Proposition \ref{formula-spherical}, we will need the following lemma.

\begin{Lem*}\label{rig-comp}
Let $\cK$ be a compact $p$-torsion-free $p$-adic Lie group. Then $\nD(\cK)$ is rigidly-compactly generated (in the sense of \cite[Def. 2.7]{BDS}).
\end{Lem*}

\begin{proof}
If $\cK$ is a pro-$p$-group, then $\nD(\cK)$ is rigidly-compactly generated by \cite[Prop. 2.3.19]{H22}, in which case the full subcategory of compact objects is the strictly full saturated triangulated subcategory $\lan 1_{\cK}\lran$ generated by the trivial representation. 

If $\cK$ is not necessarily pro-$p$, one can consider a pro-$p$-Sylow subgroup $\cP$ (which is an \emph{open} subgroup of $\cK$). Then $\Ind_{\cP}^{\cK}(1_{\cP})$ is a compact generator of $\nD(\cK)$ by \cite[Lem. 4, Prop. 6]{S15}. Moreover, we claim that it is a rigid object, i.e., that the canonical map
$$
\uRHom_{\cK}(\Ind_{\cP}^{\cK}(1_{\cP}),1_{\cK})\otimes_k Z\lra \uRHom_{\cK}(\Ind_{\cP}^{\cK}(1_{\cP}),Z)
$$
is an isomorphism for all $Z\in \nD(\cK)$ (recall that we use $\uRHom_{\cK}$ to denote the internal Hom in $\nD(\cK)$, cf. \cite[\S 3]{SS22}) . Since the canonical map 
$$
\Ind_{\cP}^{\cK}\left(\uRHom_{\cP}(X,Y|_{\cP})\right) \lra \uRHom_{\cK}(\Ind_{\cP}^{\cK}(X),Y)
$$
is an isomorphism for all $X \in \nD(\cP),Y\in \nD(\cK)$ by \cite[Lem. 2.2.3, \S 2.2.1]{heyer:geometrical}, the claim holds if and only if the canonical (projection)
map  
$$
\Ind_{\cP}^{\cK}(1_{\cP})\otimes_k Z \lra \Ind_{\cP}^{\cK}(1_{\cP} \otimes_k Z|_{\cP})
$$
is an isomorphism. But this isomorphism is readily checked if $Z\in\nD(\cK)^{\heartsuit}$, and hence follows for any $Z$ since all the differentials involved are trivial (see again \cite[Lem. 2.2.3]{heyer:geometrical}).

Finally, by the general arguments from the proof of \cite[Prop. 2.3.19]{H22}, the category $\nD(\cK)$ is rigidly-compactly generated, with full subcategory of compact objects equal to $\lan \Ind_{\cP}^{\cK}(1_{\cP})\lran$.
\end{proof}

We may now prove the above proposition.

\begin{proof}[Proof of Proposition \ref{formula-spherical}.]
First, we claim that the functor  
	$$\textnormal{R}H^0(N_0,-): \nD(P_0) \longrightarrow \nD(M_0)$$
	preserves compact objects.  
Let $\textnormal{pr}:\bfG(\cO_F) \longtwoheadrightarrow \bfG(k_F)$ denote the projection map, and define the pro-$p$-Iwahori subgroup by $I_1 := \textnormal{pr}^{-1}\left(\bfU(k_F)\right)$.  Thus $I_1$ is a pro-$p$-Sylow subgroup of $G_0$; by the Assumption \ref{assn1}, it is torsion-free, i.e.,  $G_0$ is $p$-torsion free.  Further, set 
	$I_{P,1} :=  I_1 \cap P_0$; by the Iwahori decomposition relative to $P$, the subgroup $I_{P,1}$ is a pro-$p$-Sylow of $P_0$.
	%Since each closed subgroup of $G_0$ is $p$-torsion-free, Lemma \ref{rig-comp} implies that the relevant derived categories are rigidly-compactly generated.  We may therefore use the results of \cite{BDS} exactly in the same way as in \cite[\S 3.1]{H22}.  
	Then by the proof of Lemma \ref{rig-comp}, the full subcategory of compact objects of $\nD(P_0)$ is $\langle \Ind_{I_{P,1}}^{P_0}(1_{I_{P,1}})\rangle$.  Hence, to prove the claim, it suffices to show that $\textnormal{R}H^0(N_0,\Ind_{I_{P,1}}^{P_0}(1_{I_{P,1}}))$ is a compact object of $\nD(M_0)$.  This can be checked at pro-$p$-level as follows. The factorization $P_0=M_0N_0$ induces the factorization $I_{P,1}=I_{M,1}N_0$ with $I_{M,1} := I_1 \cap M_0$. Hence we have the following diagram of functors between abelian categories 
	\begin{center}
	\begin{tikzcd}[column sep=8ex]
	\textnormal{Rep}(I_{P,1}) \ar[r, "\Ind_{I_{P,1}}^{P_0}"] \ar[d, "H^0(N_0{,}-)"'] & \textnormal{Rep}(P_0) \ar[d, "H^0(N_0{,}-)"] \\
	\textnormal{Rep}(I_{M,1})  \ar[r, "\Ind_{I_{M,1}}^{M_0}"] & \textnormal{Rep}(M_0)\ ;
	\end{tikzcd}
	\end{center}
	it is commutative (up to natural isomorphism), as can be checked using the factorizations above and the related Mackey decompositions (\cite[Thm. 1.1]{yamamoto}).  Then, since the two induction functors are exact and since the restriction $\textnormal{Rep}(P_0) \lra \textnormal{Rep}(I_{P,1})$ is exact, the diagram of functors 
	\begin{center}
	\begin{tikzcd}[column sep = 8ex]
    \nD(I_{P,1}) \arrow[r, "\mathrm{Ind}_{I_{P,1}}^{P_0}"] \arrow[d, "\mathrm{R}H^0(N_0{,}-)"']
    & \nD(P_0) \arrow[d, "\mathrm{R}H^0(N_0{,}-)"] \\
    \nD(I_{M,1}) \arrow[r, "\mathrm{Ind}_{I_{M,1}}^{M_0}"]
    & \nD(M_0)
\end{tikzcd}
	\end{center}
	is commutative (up to a natural isomorphism). Now, the functor $\Ind_{I_{M,1}}^{M_0}:\nD(I_{M,1})  \lra \nD(M_0)$ preserves compact objects because it is left adjoint to the restriction functor which commutes with direct sums. Thus it suffices to check that $\textnormal{R}H^0(N_0, 1_{I_{P,1}}) \in \nD(I_{M,1})$ is a compact object, which is the content of \cite[Prop. 3.1.6]{H22}.
	
	Next, we claim that there exists an object $\omega_{P_0/M_0}\in\nD(P_0)$ such that the functor 
	$$\textnormal{R}H^0(N_0,\omega_{P_0/M_0} \otimes_k -): \nD(P_0) \longrightarrow \nD(M_0)$$
	is left adjoint to the inflation functor\footnote{More precisely, the functor $\textnormal{R}H^0(N_0,-):\nD(P_0) \lra \nD(M_0)$ admits a right adjoint and $\omega_{P_0/M_0}$ is the value of this right adjoint on $1_{M_0}$.}. Indeed, since any closed subgroup of $G_0$ is $p$-torsion-free, Lemma \ref{rig-comp} allows us to use the results of \cite{BDS} as in \cite[\S 3.1]{H22}, with the above fact that $\textnormal{R}H^0(N_0,-)$ preserves compact objects as an input.  One can then compose this left adjoint with the restriction functor
	$$\Res^{G_0}_{P_0}:\nD(G_0) \longrightarrow \nD(P_0)$$
	to obtain a left adjoint to $\Ind_{P_0}^{G_0}:\nD(M_0) \lra \nD(G_0)$.
	%\footnote{In particular, this gives a direct construction of this left adjoint considered \emph{(b)} of \ref{Heyer}.}.
	
	Finally, it remains to show that
	$$\omega_{P_0/M_0} = N_{k_F/\bbF_p}(\overline{2\rho}\cdot \overline{2\rho_M}^{-1})[\dim(N_0)]$$
	(with coefficients of the character extended to $k$).  Since we deal with groups which are not pro-$p$, for this last point we need the notion of Poincaré group in the sense of \cite[Def. 3.7.1]{NSW}; then any closed subgroup $\cK$ of $G_0$ is a Poincar\'e group at $p$ of dimension $\dim(\cK)$ (we defer the justification of this claim until after the present proof).  In particular, 
$$
\varinjlim_{\substack{\cK'\vartriangleleft \cK\\ \textnormal{open}}} H^i(\cK',~ k)^\vee 
= 
\left\{ \begin{array}{ll}
0 & \textrm{if $0\leq i <\dim(\cK)$,}\\
k & \textrm{if $i=\dim(\cK)$,}
\end{array} \right.
$$
	where the maps in the direct limit are the $k$-linear duals of the corestriction maps;  the first equality comes from the definition, while the second one comes from \cite[Cor. 3.4.7]{NSW}. With these inputs, the proof of \cite[Prop. 3.1.10]{H22} applies to show that 
	$\omega_{P_0/M_0} \cong \delta[\dim(N_0)]$ for some smooth character $\delta: P_0 \lra k^\times$.  In particular $\delta$ is trivial on the pro-$p$ group $N_0$. To compute its value on $M_0$, note that the left adjoint property of $\textnormal{R}H^0(N_0,\omega_{P_0/M_0}\otimes-)$, and the fact that it is concentrated in degrees $\leq 0$, imply that 
	$$H^{\dim(N_0)}(N_0,-)\otimes_k \delta:\textnormal{Rep}(P_0)\lra\textnormal{Rep}(M_0)$$ is left adjoint to inflation on the heart abelian categories. Hence this left adjoint has to coincide with the functor of $N_0$-coinvariants. In particular
	$H^{\dim(N_0)}(N_0,k)\otimes_k\delta$ is the trivial representation of $M_0$. But we have computed $H^{\dim(N_0)}(N_0,k)$ in Theorem \ref{groupcoh} (taking $\lambda:=0$).  Namely, assuming that $k$ contains $k_F$, this cohomology space is given by
	$$\bigoplus_{i \in \bbZ}\textnormal{gr}_i(H^{\dim(N_0)}(N_0,~ k)) \cong \bigoplus_{\substack{w \in {}^JW\\ \ell(w) = \dim(N_0)}}L_J(w\cdot 0).$$  
	By \cite[\S 1.10, Prop.]{humphreys-book}, there is a unique element of ${}^JW$ of length $\dim(N_0)$, namely $w_{\bfM, \circ}w_\circ$, where $w_\circ$ (resp., $w_{\bfM,\circ}$) denotes the longest element of  $W$ (resp., $W(\underline{\bfM},\underline{\bfT})$).  Thus, we get $H^{\dim(N_0)}(N_0,k) \cong L_J(w_{\bfM,\circ}w_{\circ}\cdot 0)$, and a calculation shows that the character $w_{\bfM,\circ}w_{\circ}\cdot 0$ is equal to $N_{k_F/\bbF_p}(\overline{2\rho}^{-1}\cdot \overline{2\rho_M})$. In particular, this character has a model over any coefficient field $k$ of characteristic $p$.  This gives the result.
\end{proof}

\begin{Rem*}
We can also deduce Proposition \ref{formula-spherical} from results of \cite{H22} and \cite{heyer:geometrical}.  Indeed, \cite[Cor. 3.4.20]{H22} identifies the left adjoint $L(N_0,-)$ as $\nR H^0(N_0,-)$ up to a twist (see also Remark 3.4.3 of \textit{op. cit.}).  One can then explicitly identify this character using \cite[Prop. 2.3.4, Lem. 2.3.5, Prop. 2.3.9, Lem. 4.1.6]{heyer:geometrical}.  We would like to thank Heyer for pointing this out.
\end{Rem*}

The following result was used in the proof of Proposition \ref{formula-spherical}.

\begin{Lem*}
Suppose $\cK$ is a compact $p$-torsion-free $p$-adic Lie group.  Then $\cK$ is a Poincar\'e group at $p$ of dimension $\dim(\cK)$, in the sense of \cite[Def. 3.7.1]{NSW}.
\end{Lem*}

\begin{proof}
If $\cP$ denotes a pro-$p$-Sylow subgroup of $\cK$, then by \cite[Cor. (1)]{serre:cohdim} and \cite[Thm. V.2.5.8]{lazard}, $\cP$ is a Poincar\'e group of dimension $\dim(\cP) = \dim(\cK)$, in the sense of \cite[\S V.2.5.7]{lazard} or \cite[\S I.4.5]{serre:galoiscoh}.  By \cite[Prop. 3.7.6 and Def. 3.7.1]{NSW}, we see that $\cP$ is a Poincar\'e group at $p$ of dimension $\dim(\cK)$, in the sense of \cite[Ch. III, \S 7]{NSW}.  In particular, the dualizing module of $\cP$ is isomorphic (as an abelian group) to $\bbQ_p/\bbZ_p$.  By \cite[Ch. III, \S 7, Exer. 1]{NSW}, we obtain that $\cK$ is a duality group at $p$ of dimension $\dim(\cK)$ (this uses the fact that $\cd_p(\cK) = \dim(\cK) < \infty$, by \cite[Cor. (1)]{serre:cohdim}).  Since $\cP$ is open in $\cK$, the set of open normal subgroups of $\cP$ is cofinal in the set of open normal subgroups of $\cK$, and we see that the restriction to $\cP$ of the dualizing module of $\cK$ is isomorphic to the dualizing module of $\cP$ (compare \cite[\S I.3.5, Prop. 18]{serre:galoiscoh}).  Hence, the dualizing module of $\cK$ is isomorphic to $\bbQ_p/\bbZ_p$, and the claim follows.
\end{proof}

\subsection{Orthogonal decomposition for small weights}

We assume from this point onwards that the coefficient field $k$ contains $k_F$.  

\begin{Pt*}\label{comput-N_0-coh}
Let $\lambda\in X^*(\underline{\bfT})_+$ be a $p$-small character of $\bfT$ with respect to $\underline{\Phi}^+$, and let $L(\lambda)$ be the irreducible algebraic representation of $\underline{\bfG}$ of highest weight $\lambda$. Then, by Theorem \ref{groupcoh}, the object 
$$
\textnormal{R}H^0(N_0,~L(\lambda))\quad\in \ \nD(M_0)
$$
is concentrated in degrees $[0,\dim(N_0)]$, and its $n^{\textnormal{th}}$ cohomology $H^n(N_0,L(\lambda))\in \nD(M_0)^{\heartsuit}$ is finite-dimensional, with Jordan--H\"older constituents given by the highest weight representations
$$
L_J(w\cdot \lambda),\quad w\in {}^{J}W,\quad \ell(w)=n.
$$
\end{Pt*}

\begin{Th*}\label{comput-L(N_0,L(lambda))}
Suppose Assumptions \ref{assn1} and \ref{assn-centralchar} hold. Let 
$$
L(N_0,-):\nD(G_0) \lra \nD(M_0)
$$
be the left adjoint of parabolic induction at spherical level. Set $L^{ n}(N_0,-):=h^nL(N_0,-)$ and $L^{ \leq n}(N_0,-):=\tau^{\leq n}L(N_0,-)$ for $n\in\bbZ$. 

Then $L^n(N_0,L(\lambda)) = 0$ for $n\notin [-\dim(N_0),0]$.  Moreover, the $M_0$-representation $L^n(N_0,L(\lambda))$ is finite-dimensional with Jordan--H\"older constituents 
$$
L_J(w\cdot \lambda)\otimes_k N_{k_F/\bbF_p}(\overline{2\rho}\cdot \overline{2\rho_M}^{-1}),\quad w\in {}^{J}W,\quad\ell(w)=n+\dim(N_0).
$$
In particular:
\begin{enumerate}[(a)]
\item For $n = -\dim(N_0)$, we have $L^{\leq -\dim(N_0)}(N_0,L(\lambda)) \cong L^{-\dim(N_0)}(N_0,L(\lambda))[\dim(N_0)]$ and
\begin{eqnarray*}
L^{-\dim(N_0)}(N_0,~ L(\lambda)) & \cong & L(\lambda)^{N_0}\otimes_k N_{k_F/\bbF_p}(\overline{2\rho}\cdot \overline{2\rho_M}^{-1})\\
 & \cong & H^0(N_0,~L(\lambda))\otimes_k N_{k_F/\bbF_p}(\overline{2\rho}\cdot \overline{2\rho_M}^{-1})\\
 & \cong & L_J(\lambda)\otimes_k N_{k_F/\bbF_p}(\overline{2\rho}\cdot \overline{2\rho_M}^{-1}).
\end{eqnarray*}
\item For $n = 0$, we have $L^{\leq 0}(N_0,L(\lambda)) \cong L(N_0,L(\lambda))$ and
\begin{eqnarray*}
L^0(N_0,~L(\lambda)) & \cong & L(\lambda)_{N_0}\\
 & \cong & H^{\dim(N_0)}(N_0,~L(\lambda))\otimes_k N_{k_F/\bbF_p}(\overline{2\rho}\cdot \overline{2\rho_M}^{-1})\\
 & \cong & L_J(w_{\bfM,\circ}w_{\circ}\cdot \lambda)\otimes_k N_{k_F/\bbF_p}(\overline{2\rho}\cdot \overline{2\rho_M}^{-1}).
\end{eqnarray*}
\end{enumerate}
\end{Th*}

\begin{proof}
This follows from combining Proposition \ref{formula-spherical} with the discussion of $\textnormal{R}H^0\left(N_0,-\right)$ in Section \ref{comput-N_0-coh}.
\end{proof}

\begin{Cor*}\label{decomp-L(N_0)}
Suppose Assumptions \ref{assn1} and \ref{assn-centralchar} hold. If $m,n \in \bbZ$ are unequal, then the following orthogonality relations hold: 
$$
\textnormal{RHom}_{M_0}\left(L^m(N_0,L(\lambda)),L^n(N_0,L(\lambda))\right)=0.
$$
In particular $L(N_0,L(\lambda))\in \nD(M_0)$ decomposes canonically as
$$
L(N_0,~L(\lambda))\cong\bigoplus_{n=-\dim(N_0)}^0L^n(N_0,L(\lambda))[-n].
$$
\end{Cor*}

The proof of the above corollary will require an extra ingredient.  Let us first note the following lemma. Recall that $\nD(M_0)$ is a closed symmetric monoidal category; in particular it is endowed with the contravariant \emph{smooth dual} endofunctor $\cS := \uRHom_{M_0}(-,k)$, cf. \cite[\S 3]{SS22}.

\begin{Lem*}\label{smooth-dual}
Let $\nD^{\rm b}_{\fin}(M_0)$ be the full subcategory of the bounded derived category whose objects have finite dimensional cohomologies.
\begin{enumerate}[(a)]
\item\label{smooth-dual-a} The contravariant endofunctor $\cS$ of $\nD(M_0)$ restricts to an endofunctor of $\nD^{\rm b}_{\fin}(M_0)$, which then is $t$-exact.
\item\label{smooth-dual-b} The resulting contravariant endofunctor of $\nD^{\rm b}_{\fin}(M_0)^{\heartsuit}$ sends a finite dimensional smooth representation $V$ of $M_0$ to its contragredient $V^*$.
\item\label{smooth-dual-c} For all $X\in \nD^{\rm b}_{\fin}(M_0)$ and $Y\in\Rep(M_0)$, the canonical map
$$
\cS(X)\otimes_k Y=\uRHom_{M_0}(X,k)\otimes_k Y\lra \uRHom_{M_0}(X,Y)
$$
is an isomorphism.
\end{enumerate}
\end{Lem*}

\begin{proof}
If $V\in \nD(M_0)$ is concentrated in degree $0$ and finite-dimensional, it follows directly from the definition of $\cS$ that $\cS(V)$ is the contragredient $V^*$ placed in degree $0$. Then the lemma follows by dévissage.
\end{proof}

We may now prove the above corollary.

\begin{proof}[Proof of Corollary \ref{decomp-L(N_0)}]
Consider $L^m(N_0,L(\lambda))$ as an object of $\nD^{\rm b}_{\fin}(M_0)$ concentrated in degree $0$. Then, by \cite[Prop. 3.1]{SS22} and Lemma \ref{smooth-dual}, we have
\begin{eqnarray*}
\textnormal{RHom}_{M_0}\big(L^m(N_0,L(\lambda)),~L^n(N_0,L(\lambda))\big) & \cong & \RHom_{M_0}\big(k,~ \uRHom_{M_0}(L^m(N_0,L(\lambda)),~L^n(N_0,L(\lambda)))\big) \\
 & \cong & \RHom_{M_0}\big(k,~ \cS(L^m(N_0,L(\lambda)))\otimes_k L^n(N_0,L(\lambda))\big) \\
 & \cong & \RHom_{M_0}\big(k,~ L^m(N_0,L(\lambda))^*\otimes_k L^n(N_0,L(\lambda))\big) \\
& \cong &  \textnormal{R}H^0\big(M_0,~L^m(N_0,L(\lambda))^*\otimes L^n(N_0,L(\lambda))\big) 
\end{eqnarray*}
as objects of $\nD(k)$.  By Theorem \ref{comput-L(N_0,L(lambda))}, the $M_0$-representation $L^m(N_0,L(\lambda))^*\otimes L^n(N_0,L(\lambda))$ admits a filtration whose graded pieces are
$$
L_J(v\cdot \lambda)^*\otimes L_J(w\cdot \lambda),
$$
where 
$$(v,w)\in ({}^{J}W)^2,\quad\ell(v)=m+\dim(N_0),\quad\ell(w)=n+\dim(N_0).$$
Now $C_{M,0}$ acts on $L_J(v\cdot \lambda)^*\otimes L_J(w\cdot \lambda)$ by the character $(w\cdot \lambda-v\cdot\lambda)|_{C_{M,0}}$, which is non-trivial by Assumption \ref{assn-centralchar} (since $m\neq n$). Since $\textnormal{R}H^0(C_{M,0},\chi)=0$ for any non-trivial smooth character $\chi:C_{M,0} \lra k^{\times}$, we get
$$
\textnormal{R}H^0\left(C_{M,0},~L^m(N_0,L(\lambda))^*\otimes L^n(N_0,L(\lambda))|_{C_{M,0}}\right) = 0
$$
by dévissage. Hence
$$
\textnormal{R}H^0\left(M_0,~L^m(N_0,L(\lambda))^*\otimes L^n(N_0,L(\lambda))\right)=0
$$
since $\textnormal{R}H^0(M_0,-)=\textnormal{R}H^0(M_0/C_{M,0},-)\circ \textnormal{R}H^0(C_{M,0},(-)|_{C_{M,0}})$.

Finally, the canonical decomposition of $L(N_0,L(\lambda))\in \nD(M_0)$ follows from the orthogonality relations, cf. \ref{decomp-can}.
\end{proof}

\begin{Rem*} \label{comult-compute}
In the situation of Corollary \ref{decomp-L(N_0)}, let $\mu\in X^*(\underline{\bfT})$ be a character of $\underline{\bfT}$ which is dominant with respect to $\underline{\Phi}_J^+$, and let $L_J(\mu)$ be the irreducible algebraic representation of $\underline{\bfM}$ of highest weight $\mu$. 
\begin{enumerate}[(a)]
\item If $ \mu|_{C_{M,0}}\notin\{ (w\cdot \lambda) |_{C_{M,0}}:\ w\in {}^{J}W\}$, then 
$$
\textnormal{RHom}_{M_0}\left(L(N_0,L(\lambda)),~L_J(\mu)\otimes_k N_{k_F/\bbF_p} (\overline{2\rho}\cdot \overline{2\rho_M}^{-1})\right)=0.
$$
\item If $ \mu|_{C_{M,0}} = (w\cdot \lambda) |_{C_{M,0}}$ for some $w\in {}^{J}W$, then 
\begin{flushleft}
$\displaystyle{\textnormal{RHom}_{M_0}\left(L(N_0,L(\lambda)),~L_J(\mu)\otimes_k N_{k_F/\bbF_p} (\overline{2\rho}\cdot \overline{2\rho_M}^{-1})\right)}$
\end{flushleft}
\begin{flushright}
$ \cong \textnormal{RHom}_{M_0}\left(L^{\ell(w)-\dim(N_0)}(N_0,L(\lambda)),~L_J(\mu)\otimes_k N_{k_F/\bbF_p} (\overline{2\rho}\cdot \overline{2\rho_M}^{-1})\right).$
\end{flushright}
Moreover, if $w$ is the only element in ${}^JW$ satisfying $\mu|_{C_{M,0}} = (w\cdot \lambda)|_{C_{M,0}}$, then we obtain
\begin{flushleft}
$ \textnormal{RHom}_{M_0}\left(L^{\ell(w)-\dim(N_0)}(N_0,L(\lambda)),~L_J(\mu)\otimes_k N_{k_F/\bbF_p} (\overline{2\rho}\cdot \overline{2\rho_M}^{-1})\right)$
\end{flushleft}
\begin{flushright}
$ \cong   \textnormal{RHom}_{M_0}\left(L_J(w\cdot\lambda),~L_J(\mu)\right).$
\end{flushright}
\end{enumerate}
The last isomorphism is canonical only if the $M_0$-representation $H^{\ell(w)}(N_0,L(\lambda))$ is irreducible, for example this is always the case for $w\in\{1,w_{\bfM,\circ}w_{\circ}\}$; in general, the isomorphism depends on the choice of a realization of $L_J(w\cdot\lambda)$ as a subquotient of the $M_0$-representation $H^{\ell(w)}(N_0,L(\lambda))$. The proof is similar to the one of Corollary \ref{decomp-L(N_0)}.
\end{Rem*}

\section{The Satake morphism} \label{Satake:section}

We now use the previous results on orthogonal decompositions to define a derived version of the Satake morphism.  We continue to assume that $k$ contains $k_F$.

\subsection{The construction}

\begin{Pt*}

Following Heyer \cite[\S 4.3]{H22}, we use the left adjoint 
$$
L(N,-):\nD(G)\lra  \nD(M)
$$ 
of the parabolic induction functor $\Ind_P^G:\nD(M) \lra \nD(G)$ to define a derived Satake morphism with coefficients in $k$.  Recall from \cite[Thm. 4.1.1]{H22} that the adjunction $L(N,-) \dashv \Ind_P^G$ holds in the form of a natural isomorphism
$$
\iota_{X,Y}:\RHom_M(L(N,X),Y) \stackrel{\sim}{\lra} \RHom_G(X,\Ind_P^G(Y))
$$
in $\nD(k)$, for all $X\in \nD(G)$ and $Y\in \nD(M)$. Denoting by $\eta_X:X\lra \Ind_P^G(L(N,X))$ the unit, the map $\iota_{X,Y}$ is defined as the composition
\begin{center}
\begin{tikzcd}
\RHom_M(L(N,X),Y) \ar[rr, "\iota_{X,Y}", "\sim"'] \ar[dr, "\Ind_P^G"'] & & \RHom_G(X,\Ind_P^G(Y)) \\
&\RHom_G(\Ind_P^G(L(N,X)),\Ind_P^G(Y)) \ar[ur, "\RHom_G(\eta_X{,}\Ind_P^G(Y))"'] & 
\end{tikzcd}
\end{center}
Here the map $\Ind_P^G$ is well-defined since $\Ind_P^G:\Rep(M) \lra \Rep(G)$ is exact.
\end{Pt*}

\begin{Const*}
\label{pre-Sat}
Given $X\in \nD(G)$, we denote by
$$
L(N,-)_X:\RHom_G(X,X) \lra \RHom_M(L(N,X),L(N,X))\quad
$$
the composition
\begin{center}
\begin{tikzcd}
\RHom_G(X,X) \ar[rr, "L(N{,}-)_X"] \ar[dr, "\RHom_G(X{,}\eta_X)"'] & &  \RHom_M(L(N,X),L(N,X))  \\
&\RHom_G\left(X,\Ind_P^G(L(N,X))\right) \ar[ur, "\iota_{X,L(N,X)}^{-1}"', "\mathclap{\sim}" sloped] & 
\end{tikzcd}
\end{center}
Moreover, we denote by 
$$
h^{\bullet}L(N,-)_X:\Ext^{\bullet}_G(X,X) \lra \Ext^{\bullet}_M(L(N,X),L(N,X))
$$
the morphism of graded $k$-vector spaces induced by $L(N,-)_X$ on cohomologies (that is, we write $\Ext^\bullet_G(X,X)$ for $\bigoplus_{i \in \bbZ}\Ext^i_G(X,X)$, etc.).
\end{Const*}

\begin{Lem*}
The map $h^{\bullet}L(N,-)_X$ is a morphism of graded $k$-algebras for the Yoneda product.
\end{Lem*}

\begin{proof}
By definition, we have the following commutative diagram of graded $k$-vector spaces:
\begin{center}
\begin{tikzcd}[column sep=3em, row sep=3em]
    E^{\bullet}:=\Hom_{\nD(G)}(X,X[\bullet]) \arrow[d] \arrow[r, "h^{\bullet}L(N{,}-)_X"] \arrow[d, "e^{\bullet}"']
    & F^{\bullet}:=\Hom_{\nD(M)}(L(N,X),L(N,X)[\bullet]) \arrow[dl, "h^{\bullet}\iota_{X,L(N,X)}"', "\sim"' sloped] \arrow[d, "\Ind_P^G"] \\
    M^{\bullet}:=\Hom_{\nD(G)}\left(X,\Ind_P^G (L(N,X))[\bullet]\right)
    & G^{\bullet}:=\Hom_{\nD(G)}\left(\Ind_P^G (L(N,X)),\Ind_P^G (L(N,X))[\bullet]\right) \arrow[l, "g^{\bullet}"']
\end{tikzcd}
\end{center}
where
$$
e^{\bullet}:=\Hom_{\nD(G)}(X,\eta_X[\bullet])\quad\textrm{and}\quad g^{\bullet}:=\Hom_{\nD(G)}(\eta_X,\Ind_P^G(L(N,X))[\bullet]).
$$
Thus, we have that:
\begin{itemize}[$\diamond$]
\item the graded $k$-vector space $M^{\bullet}$ is a graded bimodule over $(F^{\bullet},E^{\bullet})$ (where the $F^\bullet$-action is given by first applying the functor $\Ind_P^G$);
\item the map $e^{\bullet}$ is a morphism of graded right $E^{\bullet}$-modules, which sends the unit of $E^{\bullet}$ to $\eta_X$; 
\item the map $h^{\bullet}\iota_{X,L(N,X)}$ is an isomorphism of graded left  $F^{\bullet}$-modules, which sends the unit of $F^{\bullet}$ to $\eta_X$.
\end{itemize}
The lemma now follows from a diagram chase.
\end{proof}

\begin{Pt*}\label{com-cpct}
Next, by \cite[\S 4.3]{H22} there is a natural isomorphism of functors from $\nD(G_0)$ to $\nD(M)$
\begin{equation*}
%\label{isom-left-adj}
L(N,-) \circ \cind_{G_0}^G \stackrel{\sim}{\Longrightarrow} \cind_{M_0}^M\circ L(N_0,-),
\end{equation*}
because $G_0\subset G$ is an open compact subgroup satisfying the Iwasawa decomposition $G=PG_0$ and $P_0=N_0M_0$.  Using this isomorphism, we make the following definition.
\end{Pt*}

\begin{Def*}\label{Sat}
Given $X_0\in \nD(G_0)$, we let
$$\sS_{X_0}:\RHom_G\left(\cind_{G_0}^G(X_0),~\cind_{G_0}^G(X_0)\right) \lra \RHom_M\left(\cind_{M_0}^M(L(N_0,X_0)),~\cind_{M_0}^M(L(N_0,X_0))\right)$$
be the morphism $L(N,-)_X$ of Section \ref{pre-Sat} associated to $X:=\cind_{G_0}^G(X_0)$, composed with the isomorphism of Section \ref{com-cpct} evaluated on $X_0$; it is a morphism in $\nD(k)$.

Furthermore, we let
$$h^{\bullet}\sS_{X_0}:\Ext^{\bullet}_G\left(\cind_{G_0}^G(X_0),~\cind_{G_0}^G(X_0)\right) \lra \Ext^{\bullet}_M\left(\cind_{M_0}^M(L(N_0,X_0)),~\cind_{M_0}^M(L(N_0,X_0))\right) $$
be the morphism of graded $k$-vector spaces induced by $\sS_{X_0}$ on cohomologies; it is a morphism of graded $k$-algebras.
\end{Def*}

\subsection{Satake morphisms for $p$-small weights}

We now show how the Satake map $\sS_{X_0}$ decomposes when we take $X_0 = L(\lambda)$.

\begin{Th*}\label{decomp-L(N)}
Suppose Assumptions \ref{assn1} and \ref{assn-centralchar} hold. Let 
$$
L(N,-):\nD(G) \lra \nD(M)
$$
be the left adjoint of parabolic induction, and set $L^{ n}(N,-):=h^nL(N,-)$ for $n\in\bbZ$.  If $m,n \in \bbZ$ are unequal, then the following orthogonality relations hold: 
$$
\textnormal{RHom}_{M}\left(L^m\big(N,\cind_{G_0}^G(L(\lambda))\big),~L^n\big(N,\cind_{G_0}^G(L(\lambda))\big)\right) = 0.
$$
In particular $L(N,\cind_{G_0}^G(L(\lambda)))\in \nD(M)$ decomposes as
$$
L\big(N,\cind_{G_0}^G(L(\lambda))\big) \cong \bigoplus_{n=-\dim(N)}^0L^n\big(N,\cind_{G_0}^G(L(\lambda))\big)[-n],
$$
and $\RHom_{M}(L(N,\cind_{G_0}^G(L(\lambda))),~L(N,\cind_{G_0}^G(L(\lambda))))\in\nD(k)$ decomposes as
\begin{flushleft}
$\displaystyle{\RHom_{M}\left(L(N,\cind_{G_0}^G(L(\lambda))),~L(N,\cind_{G_0}^G(L(\lambda)))\right)}$
\end{flushleft} 
\begin{flushright}
$\displaystyle{\cong \bigoplus_{n=-\dim(N)}^0\RHom_{M}\left(L^n(N,\cind_{G_0}^G(L(\lambda))),~L^n(N,\cind_{G_0}^GL(\lambda)))\right).}$
\end{flushright}
\end{Th*}

\begin{proof}
Since $\cind_{M_0}^M:\Rep(M_0) \lra \Rep(M)$ is exact, we have $L^m(N,\cind_{G_0}^G(L(\lambda))) \cong \cind_{M_0}^M(L^m(N_0,L(\lambda)))$ by the discussion in Section \ref{com-cpct}.  Hence, by Frobenius reciprocity, we obtain
\begin{flushleft}
$\displaystyle{\textnormal{RHom}_{M}\left(L^m\big(N,\cind_{G_0}^G(L(\lambda))\big),~L^n\big(N,\cind_{G_0}^G(L(\lambda))\big)\right) }$
\end{flushleft}
\begin{flushright}
$\displaystyle{\cong\textnormal{RHom}_{M_0}\left(L^m(N_0,L(\lambda)),~L^n\big(N,\cind_{G_0}^G(L(\lambda))\big)|_{M_0}\right)}$
\end{flushright}
(recall that $(-)|_{M_0}: \Rep(M) \lra \Rep(M_0)$ is exact). By \cite[Prop. 3.1]{SS22} and Lemma \ref{smooth-dual}, this gives
\begin{flushleft}
$\displaystyle{\textnormal{RHom}_{M}\left(L^m\big(N,\cind_{G_0}^G(L(\lambda))\big),~L^n\big(N,\cind_{G_0}^G(L(\lambda))\big)\right)}$ 
\end{flushleft}
\begin{flushright}
$\displaystyle{\cong \textnormal{R}H^0\left(M_0,~L^m(N_0,L(\lambda))^*\otimes L^n\big(N,\cind_{G_0}^G(L(\lambda))\big)|_{M_0}\right)
}$
\end{flushright}
(see the proof of Corollary \ref{decomp-L(N_0)}).  On the other hand, by Theorem \ref{comput-L(N_0,L(lambda))} the $M_0$-representation $L^n(N,\cind_{G_0}^G(L(\lambda)))|_{M_0} \cong \cind_{M_0}^M(L^n(N_0,L(\lambda)))|_{M_0}$ admits a filtration whose graded pieces are
$$
\left(\cind_{M_0}^M (L_J(w\cdot \lambda))\otimes_k N_{k_F/\bbF_p} (\overline{2\rho}\cdot \overline{2\rho_M}^{-1})\right)  |_{M_0},
$$
where $w\in {}^{J}W$ satsifies $\ell(w)=n+\dim(N_0)$.  Recall that $C_{M,0}$ is central in $M$.  It follows from the Mackey formula that $C_{M,0}$ acts on the $w^{\textnormal{th}}$ graded piece by the character $(w\cdot \lambda)|_{C_{M,0}}\otimes_k N_{k_F/\bbF_p} (\overline{2\rho}\cdot \overline{2\rho_M}^{-1})|_{C_{M,0}}$. 

From this point one concludes exactly as in the proof of Corollary \ref{decomp-L(N_0)}, further noting that $\textnormal{R}H^0(C_{M,0},\bigoplus_I\chi) \cong \bigoplus_I\textnormal{R}H^0(C_{M,0},\chi)$ for a non-necessarily finite set $I$, since $C_{M,0}$ is compact.
\end{proof}

\begin{Cor*}
Suppose Assumptions \ref{assn1} and \ref{assn-centralchar} hold. Then the Satake morphism in $\nD(k)$
\begin{flushleft}
$\displaystyle{ \sS_{L(\lambda)}:\RHom_G\left(\cind_{G_0}^G(L(\lambda)),~\cind_{G_0}^G(L(\lambda))\right) }$
\end{flushleft}
\begin{flushright}
$\displaystyle{  \lra\RHom_M\left(\cind_{M_0}^M(L(N_0,L(\lambda))),~\cind_{M_0}^M(L(N_0,L(\lambda)))\right) }$
\end{flushright}
decomposes canonically as a family 
\begin{flushleft}
$\displaystyle{ \sS_{L(\lambda),n}:\RHom_G\left(\cind_{G_0}^G(L(\lambda)),~\cind_{G_0}^G(L(\lambda))\right) }$
\end{flushleft}
\begin{flushright}
$\displaystyle{  \lra\RHom_M\left(\cind_{M_0}^M(L^n(N_0,L(\lambda))),~\cind_{M_0}^M(L^n(N_0,L(\lambda)))\right) }$
\end{flushright}
indexed by $n\in [-\dim(N),0]$. Accordingly, the morphism of graded $k$-algebras
\begin{flushleft}
$\displaystyle{ h^{\bullet}\sS_{L(\lambda)}:\Ext_G^{\bullet}\left(\cind_{G_0}^G(L(\lambda)),~\cind_{G_0}^G(L(\lambda))\right) }$
\end{flushleft}
\begin{flushright}
$\displaystyle{ \lra\Ext_M^{\bullet}\left(\cind_{M_0}^M(L(N_0,L(\lambda))),~\cind_{M_0}^M(L(N_0,L(\lambda)))\right) }$
\end{flushright}
decomposes canonically as a family of graded $k$-algebras
\begin{flushleft}
$\displaystyle{ h^{\bullet}\sS_{L(\lambda),n}:\Ext_G^{\bullet}\left(\cind_{G_0}^G(L(\lambda)),~\cind_{G_0}^G(L(\lambda))\right) }$
\end{flushleft}
\begin{flushright}
$\displaystyle{ \lra\Ext_M^{\bullet}\left(\cind_{M_0}^M(L^n(N_0,L(\lambda))),~\cind_{M_0}^M(L^n(N_0,L(\lambda)))\right) }$
\end{flushright}
indexed by $n\in [-\dim(N),0]$.
\end{Cor*}

\begin{Rem*}
\label{link-Heyer}
The degree $0$ part $\{h^0\sS_{L(\lambda),n}\}_{n\in  [-\dim(N),0]}$ of the family $\{h^{\bullet}\sS_{L(\lambda),n}\}_{n\in  [-\dim(N),0]}$ is the family of morphisms of $k$-algebras considered by Heyer in \cite[\S 4.3, Def.]{H22} (for the $G_0$-representation $L(\lambda)$). In particular, the morphism of $k$-algebras 
$$
h^0\sS_{L(\lambda),0}:\End_G\left(\cind_{G_0}^G(L(\lambda))\right)\lra \End_M\left(\cind_{M_0}^M(L(\lambda)_{N_0})\right)
$$ 
is the Satake morphism originally defined by Herzig in \cite{herzig:satake},\cite{herzig:inventiones}, and generalized by Henniart-Vignéras in \cite{HV12},\cite{HV15}; see \cite[Thm. 4.3.2]{H22}.
\end{Rem*}

\begin{Rem*}
\label{link-Ronchetti}
Suppose that Assumptions \ref{assn1} and \ref{assn-centralchar} hold.  In the discussion below, we take $\lambda = 0$, so that when $\bfM = \bfT$, Assumption \ref{assn-centralchar} is automatically satisfied (by Lemma \ref{split:M=T,lambda=0}).  

We examine the Satake morphisms introduced above.  The representation $L(\lambda)$ is equal to the trivial $G_0$-representation $1_{G_0}$, and we will show that the $n=0$ component $h^{\bullet}\sS_{1_{G_0},0}$ of the family $\{h^{\bullet}\sS_{1_{G_0},n}\}_{n\in  [-\dim(N),0]}$,  is the morphism of graded $k$-algebras constructed by Ronchetti in \cite[\S 6 Def. 13]{R19b}\footnote{More precisely, the morphism in \cite[\S 6 Def. 13]{R19b} is only defined in $\Ext$-degrees $\bullet \leq 1$, and the morphism $h^{\leq 1}\sS_{1_{G_0},0}$ is equal to it.}. Indeed, by the above construction we have the commutative diagram:

\begin{center}
\adjustbox{scale=.68,center}{
\begin{tikzcd}[column sep=3em, row sep=3em]
    & \RHom_M\left(\cind_{M_0}^M(1_{M_0}),~\cind_{M_0}^M(1_{M_0})\right) \arrow[dr, "\mathclap{\sim}" sloped] & \\
 \RHom_G\left(\cind_{G_0}^G(1_{G_0}),~\cind_{G_0}^G(1_{G_0})\right)  \ar[dr] \ar[r, "\sS_{1_{G_0}}"] \ar[ur, "\sS_{1_{G_0},0}"]  & \RHom_M\left(L(N,\cind_{G_0}^G(1_{G_0})),~L(N,\cind_{G_0}^G(1_{G_0}))\right) \ar[r] \ar[d, "\rotatebox{90}{$\sim$}"', "\iota_{\cind_{G_0}^G(1_{G_0}), L(N,\cind_{G_0}^G(1_{G_0}))}"] & \RHom_M\left(L(N,\cind_{G_0}^G(1_{G_0})),~\cind_{M_0}^M(1_{M_0})\right) \ar[d, "\rotatebox{90}{$\sim$}"', "\iota_{\cind_{G_0}^G(1_{G_0}), \cind_{M_0}^M(1_{M_0}) }"]\\
    & \RHom_G\left(\cind_{G_0}^G(1_{G_0}),~\Ind_P^G(L(N,\cind_{G_0}^G(1_{G_0})))\right) \ar[r] & \RHom_G\left(\cind_{G_0}^G(1_{G_0}),~\Ind_P^G(\cind_{M_0}^M(1_{M_0}))\right)
\end{tikzcd}
}
\end{center}
To justify the commutativity, note that we have a canonical augmentation
\begin{equation}
\label{augmentation}
L(N,\cind_{G_0}^G(1_{G_0})) \lra \tau^{\geq 0}L(N,\cind_{G_0}^G(1_{G_0})) \cong  L^0(N,\cind_{G_0}^G(1_{G_0})) \cong \cind_{M_0}^M(1_{M_0}).
\end{equation}
This map induces the commutativity of the lower-right square, and gives the diagonal isomorphism (by invoking Theorem \ref{decomp-L(N)}).  The lower-left triangle commutes by definition of $\sS_{1_{G_0}}$, while the upper triangle commutes by Theorem \ref{decomp-L(N)}.

Hence $\sS_{1_{G_0},0}$ fits into the commutative diagram  
\begin{center}
\begin{tikzcd}[column sep=3em, row sep=3em]
    \RHom_G\left(\cind_{G_0}^G(1_{G_0}),~\cind_{G_0}^G(1_{G_0})\right) \ar[dr] \ar[r, "\sS_{1_{G_0},0}"]    & \RHom_M\left(\cind_{M_0}^M(1_{M_0}),~\cind_{M_0}^M(1_{M_0})\right) \arrow[d, "\rotatebox{90}{$\sim$}"] \\

    & \RHom_G\left(\cind_{G_0}^G(1_{G_0}),~\Ind_P^G(\cind_{M_0}^M(1_{M_0}))\right)
\end{tikzcd}
\end{center}
Here the isomorphism comes from the two outer isomorphisms of the previous diagram, while the diagonal map is induced from the composition 
$$\cind_{G_0}^G(1_{G_0}) \longrightarrow \Ind_P^G(L(N,\cind_{G_0}^G(1_{G_0}))) \longrightarrow \Ind_P^G(\cind_{M_0}^M(1_{M_0})),$$
where the first map is the unit $\eta_{\cind_{G_0}^G(1_{G_0})}$, and the second is induced by the augmentation map \eqref{augmentation}.  Thus, we see that
$$
h^{\bullet}\sS_{1_{G_0},0}:\Ext_G^{\bullet}\left(\cind_{G_0}^G(1_{G_0}),~\cind_{G_0}^G(1_{G_0})\right )\lra \Ext_M^{\bullet}\left(\cind_{M_0}^M(1_{M_0}),~\cind_{M_0}^M(1_{M_0})\right)
$$ 
is obtained by evaluating the natural right action of $\Ext_G^{\bullet}(\cind_{G_0}^G(1_{G_0}),\cind_{G_0}^G(1_{G_0}))$ on the canonical intertwiner $\cind_{G_0}^G(1_{G_0})\lra\Ind_P^G(\cind_{M_0}^M(1_{M_0}))$ constructed above, which is a degree $0$ element of the graded bimodule $\Ext_G^{\bullet}(\cind_{G_0}^G(1_{G_0}),\Ind_P^G(\cind_{M_0}^M(1_{M_0})))$ of universal unramified parabolic induction.  This is precisely the construction from \cite[\S 2]{HV12}.
\end{Rem*}

\begin{Rem*}
\label{unr-p-s}
The vertical isomorphism from the previous diagram
$$
\RHom_M\left(\cind_{M_0}^M(1_{M_0}),~\cind_{M_0}^M(1_{M_0})\right) \stackrel{\sim}{\lra} \RHom_G\left(\cind_{G_0}^G(1_{G_0}),~ \Ind_P^G(\cind_{M_0}^M(1_{M_0}))\right)
$$
appearing in Remark \ref{link-Ronchetti} is interesting in its own right, since the right-hand side computes the $G_0$-cohomology of the \emph{universal} unramified parabolic induction $\Ind_P^G(\cind_{M_0}^M(1_{M_0}))$. It ultimately relies on the $N_0$-cohomology computation of Section \ref{comput-N_0-coh}. 
\end{Rem*}

\appendix

\section{Orthogonality and canonical splitting}

In this appendix we recall some facts about splitting of objects in derived categories.  The proofs are standard, but we include them for the sake of completeness.

Let $\cC$ be an abelian category having enough injectives. Denote by $\nD^{\rm b}(\cC)$ its bounded derived category. 

\begin{Pt*}
Recall the truncation functors $\tau^{\geq n}:\nD^{\rm b}(\cC)\lra\nD^{\rm b}(\cC)$ for $n\in\bbZ$.  For each $n\in\bbZ$ and $X\in \nD^{\rm b}(\cC)$, we have the distinguished triangle
\begin{equation}\label{trunc-triangle-a}
(h^nX)[-n] \lra \tau^{\geq n}X \lra \tau^{\geq n+1}X \lra ((h^nX)[-n] )[1].
\end{equation}
\end{Pt*}

\begin{Lem*}\label{split-lemma-a}
Let $X\in \nD^{\rm b}(\cC)$.  Assume $\Ext_{\cC}^{i-n}(h^iX,h^nX)=\Ext_{\cC}^{i-n+1}(h^iX,h^nX)=0$ for some $n\in\bbZ$ and all $i>n$. Then the morphism 
$$
\tau^{\geq n}X \lra \tau^{\geq n+1}X
$$ 
admits a section, which is unique. Together with $(h^nX)[-n] \lra \tau^{\geq n}X$, it defines an isomorphism
$$
\alpha_{\tau^{\geq n}X}: (h^nX)[-n]\oplus\tau^{\geq n+1}X \stackrel{\sim}{\lra} \tau^{\geq n}X.
$$
Also, the morphism
$$
(h^nX)[-n] \lra \tau^{\geq n}X
$$
admits a retraction, which is unique. Together with $\tau^{\geq n}X\lra \tau^{\geq n+1}X$, it defines an isomorphism
$$
\beta_{\tau^{\geq n}X}:  \tau^{\geq n}X \stackrel{\sim}{\lra} (h^nX)[-n]\oplus\tau^{\geq n+1}X.
$$
The composition
\begin{center}
\begin{tikzcd}[column sep = 10ex]
    (h^nX)[-n]\oplus\tau^{\geq n+1}X \arrow[r, "\alpha_{\tau^{\geq n}X}", "\sim"'] 
    & \tau^{\geq n}X \arrow[r, "\beta_{\tau^{\geq n}X}", "\sim"']
    & (h^nX)[-n]\oplus\tau^{\geq n+1}X
\end{tikzcd}
\end{center}
is the identity, i.e. $\alpha_{\tau^{\geq n}X}$ and $\beta_{\tau^{\geq n}X}$ are inverse one to the other.
\end{Lem*}

\begin{proof}
Applying $\Hom_{\nD^{\rm b}(\cC)}(\tau^{\geq n+1}X,-)$ to the triangle (\ref{trunc-triangle-a}) we get the exact sequence of abelian groups
\begin{flushleft}
$\displaystyle{\Hom_{\nD^{\rm b}(\cC)}(\tau^{\geq n+1}X,(h^nX)[-n])\lra \Hom_{\nD^{\rm b}(\cC)}(\tau^{\geq n+1}X,\tau^{\geq n}X) \lra \Hom_{\nD^{\rm b}(\cC)}(\tau^{\geq n+1}X,\tau^{\geq n+1}X) }$
\end{flushleft}
\begin{flushright}
$ \displaystyle{\lra \Hom_{\nD^{\rm b}(\cC)}(\tau^{\geq n+1}X,((h^nX)[-n] )[1]).}$
\end{flushright}
The first term vanishes because 
$$
\Hom_{\nD^{\rm b}(\cC)}((h^iX)[-i],(h^nX)[-n])=\Ext_{\cC}^{i-n}(h^iX,h^nX)=0
$$
for all $i\geq n+1$ by assumption. The fourth term vanishes because 
$$
\Hom_{\nD^{\rm b}(\cC)}((h^iX)[-i],((h^nX)[-n] )[1])=\Ext_{\cC}^{i-n+1}(h^iX,h^nX)=0
$$
for all $i\geq n+1$ by assumption. Hence the identity of $\tau^{\geq n+1}X$ lifts uniquely to a morphism $\tau^{\geq n+1}X\lra\tau^{\geq n}X$, as desired.

Similarly, applying $\Hom_{\nD^{\rm b}(\cC)}(-,(h^nX)[-n])$ to the triangle (\ref{trunc-triangle-a}) we get the exact sequence of abelian groups
\begin{flushleft}
$\displaystyle{\Hom_{\nD^{\rm b}(\cC)}(\tau^{\geq n+1}X,(h^nX)[-n])\lra \Hom_{\nD^{\rm b}(\cC)}(\tau^{\geq n}X,(h^nX)[-n]) \lra \Hom_{\nD^{\rm b}(\cC)}((h^nX)[-n],(h^nX)[-n]) }$
\end{flushleft}
\begin{flushright}
$\displaystyle{\hfill \lra \Hom_{\nD^{\rm b}(\cC)}(\tau^{\geq n+1}X[-1],(h^nX)[-n]).}$
\end{flushright}
The first term vanishes because 
$$
\Hom_{\nD^{\rm b}(\cC)}((h^iX)[-i],(h^nX)[-n])=\Ext_{\cC}^{i-n}(h^iX,h^nX)=0
$$
for all $i\geq n+1$ by assumption. The fourth term vanishes because 
$$
\Hom_{\nD^{\rm b}(\cC)}((h^iX)[-i][-1],(h^nX)[-n])=\Ext_{\cC}^{i-n+1}(h^iX,h^nX)=0
$$
for all $i\geq n+1$ by assumption. Hence the identity of $(h^nX)[-n]$ lifts uniquely to a morphism $\tau^{\geq n}X\lra(h^nX)[-n]$, as desired.

Finally, we have
$$
(\beta_{\tau^{\geq n}X}\circ \alpha_{\tau^{\geq n}X})|_{(h^nX)[-n]}=(\Id_{(h^nX)[-n]},0)
$$
by construction, and 
$$
(\beta_{\tau^{\geq n}X}\circ \alpha_{\tau^{\geq n}X})|_{\tau^{\geq n+1}X}=(0,\Id_{\tau^{\geq n+1}X})
$$
by construction and because $\Hom_{\nD^{\rm b}(\cC)}(\tau^{\geq n+1}X,(h^nX)[-n])=0$.
\end{proof}

\begin{Cor*}
\label{split-cor-a}
Let $X\in\nD^{\rm b}(\cC)$ such that $\Ext_{\cC}^{i-n}(h^iX,h^nX)=\Ext_{\cC}^{i-n+1}(h^iX,h^nX)=0$ for all $n,i\in\bbZ$ with $i>n$. Then the family $\{\alpha_{\tau^{\geq n}X}\}_{n\in\bbZ}$ assemble to an isomorphism
$$
\alpha_{\geq,X}:\bigoplus_n (h^nX)[-n] \stackrel{\sim}{\lra} X,
$$
the family $\{\beta_{\tau^{\geq n}X}\}_{n\in\bbZ}$ assemble to an isomorphism
$$
\beta_{\geq,X}: X \stackrel{\sim}{\lra} \bigoplus_n (h^nX)[-n],
$$
and $\alpha_{\geq,X}$ and $\beta_{\geq,X}$ are inverse one to the other.
\end{Cor*}

\begin{Pt*}
Recall the truncation functors $\tau^{\leq n}:\nD^{\rm b}(\cC)\lra\nD^{\rm b}(\cC)$ for $n\in\bbZ$. For each $n\in\bbZ$ and $X\in \nD^{\rm b}(\cC)$, we have the distinguished triangle
\begin{equation}\label{trunc-triangle-b}
\tau^{\leq n-1}X \lra \tau^{\leq n}X \lra (h^nX)[-n] \lra (\tau^{\leq n-1}X)[1].
\end{equation}
\end{Pt*}

The following results follow in a completely analogous manner to Lemma \ref{split-lemma-a} and Corollary \ref{split-cor-a}.

\begin{Lem*}\label{split-lemma-b}
Let $X\in \nD^{\rm b}(\cC)$.  Assume $\Ext_{\cC}^{n-i}(h^nX,h^iX)=\Ext_{\cC}^{n-i+1}(h^nX,h^iX)=0$ for some $n\in\bbZ$ and all $i<n$. Then 
the morphism
$$
\tau^{\leq n}X\lra (h^nX)[-n] 
$$
admits a section, which is unique. Together with $\tau^{\leq n-1}X\lra \tau^{\leq n}X$, it defines an isomorphism
$$
\alpha_{\tau^{\leq n}X}:  (h^nX)[-n]\oplus\tau^{\leq n-1}X \stackrel{\sim}{\lra} \tau^{\leq n}X.
$$
Also, the morphism 
$$
\tau^{\leq n-1}X \lra \tau^{\leq n}X
$$ 
admits a retraction, which is unique. Together with $\tau^{\leq n}X \lra (h^nX)[-n]$, it defines an isomorphism
$$
\beta_{ \tau^{\leq n}X}:\tau^{\leq n}X \stackrel{\sim}{\lra} \tau^{\leq n-1}X\oplus (h^nX)[-n].
$$
The composition
\begin{center}
\begin{tikzcd}[column sep = 10ex]
    (h^nX)[-n]\oplus\tau^{\leq n-1}X \arrow[r, "\alpha_{\tau^{\leq n}X}" , "\sim"']
    & \tau^{\leq n}X \arrow[r, "\beta_{\tau^{\leq n}X}" , "\sim"']
    & (h^nX)[-n]\oplus\tau^{\leq n-1}X
\end{tikzcd}
\end{center}
is the identity, i.e. $\alpha_{\tau^{\leq n}X}$ and $\beta_{\tau^{\leq n}X}$ are inverse one to the other.
\end{Lem*}

\begin{Cor*}
Let $X\in\nD^{\rm b}(\cC)$ such that $\Ext_{\cC}^{n-i}(h^nX,h^iX)=\Ext_{\cC}^{n-i+1}(h^nX,h^iX)=0$ for all $n,i\in\bbZ$ with $i<n$. Then the family $\{\alpha_{\tau^{\leq n}X}\}_{n\in\bbZ}$ assemble to an isomorphism
$$
\alpha_{\leq,X}:\bigoplus_n (h^nX)[-n] \stackrel{\sim}{\lra} X,
$$
the family $\{\beta_{\tau^{\leq n}X}\}_{n\in\bbZ}$ assemble to an isomorphism
$$
\beta_{\leq,X}: X \stackrel{\sim}{\lra} \bigoplus_n (h^nX)[-n],
$$
and $\alpha_{\leq,X}$ and $\beta_{\leq,X}$ are inverse one to the other.
\end{Cor*}

%\begin{Lem}
%Let $a,b\in\bbZ$ with $a<b$ and $b-a=1$, and let $X\in \nD^{[a,b]}(\cC)$. Assume $\Ext_{\cC}^1(h^bX,h^aX)=\Ext_{\cC}^2(h^bX,h^aX)=0$, so that
%$$
%\xymatrix{
%(h^aX)[-a]\oplus (h^bX)[-b] \ar[r]^<<<<<{\alpha_{X,a}}_>>>>>{\sim} & X  \ar[r]^<<<<<{\beta_{X,b}}_>>>>>{\sim} & (h^aX)[-a]\oplus (h^bX)[-b]
%}
%$$
%are defined. Then $\beta_{X,b}\circ \alpha_{X,a}=\Id$, i.e. $\alpha_{X,a}$ and $\beta_{X,b}$ are inverse one to the other.
%\end{Lem}

%\begin{proof}
%Indeed, 
%$$
%(\beta_{X,b}\circ \alpha_{X,a})|_{(h^aX)[-a]}=(\Id_{(h^aX)[-a]},0)
%$$
%by construction, and 
%$$
%(\beta_{X,b}\circ \alpha_{X,a})|_{(h^bX)[-b]}=(0,\Id_{(h^bX)[-b]})
%$$
%by construction and because $\Hom_{\nD^{\rm b}(\cC)}((h^bX)[-b],(h^aX)[-a])=\Ext_{\cC}^1(h^bX,h^aX)=0$.
%\end{proof}

\begin{Lem*}\label{comp-geq-leq}
Let $X\in \nD^{\rm b}(\cC)$. Assume $\Ext_{\cC}^{b-a}(h^bX,h^aX)=\Ext_{\cC}^{b-a+1}(h^bX,h^aX)=0$ for all $a,b\in\bbZ$ with $a<b$. Then 
$$
\alpha_{\geq,X}=\alpha_{\leq,X}\quad\textrm{and}\quad\beta_{\geq,X}=\beta_{\leq,X}.
$$
\end{Lem*}

\begin{proof}
Fix $n\in\bbZ$, and consider the commutative diagram of canonical morphisms
\begin{center}
\begin{tikzcd}[column sep=2.5em, row sep=2.5em]
    & X \arrow[dr, "q"] & \\
    \tau^{\leq n}X \arrow[ur, "i"] \arrow[dr, "p"']
    & & \tau^{\geq n}X \\
    & (h^nX)[-n] \arrow[ur, "j"'] &
\end{tikzcd}
\end{center}
The morphism $q$ is split by the composition of sections 
$$
s: \tau^{\geq n}X \lra \tau^{\geq n-1}X \lra \tau^{\geq n - 2}X \lra \ldots \lra  \tau^{\geq m}X = X
$$ 
(for $m \ll 0$) from Lemma \ref{split-lemma-a}, and by definition $\alpha_{\geq,X}|_{(h^nX)[-n]}:(h^nX)[-n]\lra X$ is $s\circ j$. Since $(h^nX)[-n]$ is concentrated in degrees $\leq n$, the latter factors uniquely through $i$, i.e., there exists a unique $f:(h^nX)[-n]\lra \tau^{\leq n}X$ such that 
$\alpha_{\geq,X}|_{(h^nX)[-n]}=i\circ f$. Then:
$$
j\circ p \circ f = q\circ i\circ f = q \circ s\circ j =j.
$$
But $j$ admits a retraction by Lemma \ref{split-lemma-a}, hence $p\circ f=\Id_{(h^nX)[-n]}$. Consequently, the morphism $f$ has to coincide with the unique section of $p$ from Lemma \ref{split-lemma-b}. Then by definition $\alpha_{\leq,X}|_{(h^nX)[-n]}:(h^nX)[-n]\lra X$ is $i\circ f$. We have thus obtained that 
$\alpha_{\geq,X}$ and $\alpha_{\leq,X}$ agree on the direct summand $(h^nX)[-n]$.

It follows that $\alpha_{\geq,X}=\alpha_{\leq,X}$, and hence that $\beta_{\geq,X}=\beta_{\leq,X}$ by passing to the inverse.
\end{proof}

\begin{Pt*}
In the situation of Lemma \ref{comp-geq-leq}, we set
$$
\alpha_X:=\alpha_{\geq,X}=\alpha_{\leq,X}\quad\textrm{and}\quad\beta_X:=\beta_{\geq,X}=\beta_{\leq,X}.
$$
They are isomorphisms, which are inverse one to the other.
\end{Pt*}

\begin{Lem*}\label{decomp-can}
Let $X\in \nD^{\rm b}(\cC)$. Assume $\Ext_{\cC}^i(h^mX,h^nX)=0$ for all $i,m,n\in\bbZ$ with $m\neq n$; in particular
$$
X\cong \bigoplus_{n \in \bbZ}(h^nX)[-n]
$$
using $\alpha_X=\beta_X^{-1}$. Then
$$
\RHom_{\cC}(X,X)\cong \bigoplus_{m,n\in \bbZ}\RHom_{\cC}(h^mX,h^nX)[m-n] \cong \bigoplus_{n\in \bbZ}\RHom_{\cC}(h^nX,h^nX)
$$
in the derived category of abelian groups. Moreover, the induced isomorphism 
$$
\Ext_{\cC}^{\bullet}(X,X)\cong \bigoplus_{n\in \bbZ}\Ext_{\cC}^{\bullet}(h^nX,h^nX)
$$
in the category of graded abelian groups is an isomorphism of graded rings for the Yoneda product.
\end{Lem*}

\begin{proof}
This is clear.
\end{proof}

%\section{Conflict of interest statement}
%
%On behalf of all authors, Cédric Pépin states that there is no conflict of interest.

%\nocite{*}
\bibliographystyle{amsalpha}
\bibliography{refs}

\providecommand{\bysame}{\leavevmode\hbox to3em{\hrulefill}\thinspace}
\providecommand{\MR}{\relax\ifhmode\unskip\space\fi MR }
\providecommand{\MRhref}[2]{%
  \href{http://www.ams.org/mathscinet-getitem?mr=#1}{#2}
}
\providecommand{\href}[2]{#2}
\begin{thebibliography}{oGVAG09}

\bibitem[BB05]{bjornerbrenti}
Anders Bj\"{o}rner and Francesco Brenti, \emph{Combinatorics of {C}oxeter
  groups}, Graduate Texts in Mathematics, vol. 231, Springer, New York, 2005.
  \MR{2133266}

\bibitem[BDS16]{BDS}
Paul Balmer, Ivo Dell'Ambrogio, and Beren Sanders, \emph{Grothendieck-{N}eeman
  duality and the {W}irthm\"{u}ller isomorphism}, Compos. Math. \textbf{152}
  (2016), no.~8, 1740--1776. \MR{3542492}

\bibitem[Bou81]{bourbaki:Lie4-6}
Nicolas Bourbaki, \emph{\'{E}l\'{e}ments de math\'{e}matique}, Masson, Paris,
  1981, Groupes et alg\`ebres de Lie. Chapitres 4, 5 et 6. [Lie groups and Lie
  algebras. Chapters 4, 5 and 6]. \MR{647314}

\bibitem[Bou12]{bourbaki:algch8}
\bysame, \emph{\'{E}l\'{e}ments de math\'{e}matique. {A}lg\`ebre. {C}hapitre 8.
  {M}odules et anneaux semi-simples}, Springer, Berlin, 2012, Second revised
  edition of the 1958 edition [MR0098114]. \MR{3027127}

\bibitem[Car79]{C79}
P.~Cartier, \emph{Representations of {$p$}-adic groups: a survey}, Automorphic
  forms, representations and {$L$}-functions ({P}roc. {S}ympos. {P}ure {M}ath.,
  {O}regon {S}tate {U}niv., {C}orvallis, {O}re., 1977), {P}art 1, Proc. Sympos.
  Pure Math., XXXIII, Amer. Math. Soc., Providence, R.I., 1979, pp.~111--155.
  \MR{546593}

\bibitem[CGP15]{CGP}
Brian Conrad, Ofer Gabber, and Gopal Prasad, \emph{Pseudo-reductive groups},
  second ed., New Mathematical Monographs, vol.~26, Cambridge University Press,
  Cambridge, 2015. \MR{3362817}

\bibitem[Con14]{conrad:sga3}
Brian Conrad, \emph{Reductive group schemes}, Autour des sch\'{e}mas en
  groupes. {V}ol. {I}, Panor. Synth\`eses, vol. 42/43, Soc. Math. France,
  Paris, 2014, pp.~93--444. \MR{3362641}

\bibitem[Con15]{conrad:nonsplitZ}
\bysame, \emph{Non-split reductive groups over {${\bf Z}$}}, Autours des
  sch\'{e}mas en groupes. {V}ol. {II}, Panor. Synth\`eses, vol.~46, Soc. Math.
  France, Paris, 2015, pp.~193--253. \MR{3525597}

\bibitem[CR62]{curtisreiner}
Charles~W. Curtis and Irving Reiner, \emph{Representation theory of finite
  groups and associative algebras}, Pure and Applied Mathematics, Vol. XI,
  Interscience Publishers (a division of John Wiley \& Sons, Inc.), New
  York-London, 1962. \MR{144979}

\bibitem[CW74]{CW74}
W.~Casselman and D.~Wigner, \emph{Continuous cohomology and a conjecture of
  {S}erre's}, Invent. Math. \textbf{25} (1974), 199--211. \MR{352333}

\bibitem[DG70]{demazuregabriel}
Michel Demazure and Pierre Gabriel, \emph{Groupes alg\'{e}briques. {T}ome {I}:
  {G}\'{e}om\'{e}trie alg\'{e}brique, g\'{e}n\'{e}ralit\'{e}s, groupes
  commutatifs}, Masson \& Cie, \'{E}diteurs, Paris; North-Holland Publishing
  Co., Amsterdam, 1970, Avec un appendice {{\i}t Corps de classes local} par
  Michiel Hazewinkel. \MR{302656}

\bibitem[DL76]{delignelusztig}
P.~Deligne and G.~Lusztig, \emph{Representations of reductive groups over
  finite fields}, Ann. of Math. (2) \textbf{103} (1976), no.~1, 103--161.
  \MR{393266}

\bibitem[GHS18]{GHS}
Toby Gee, Florian Herzig, and David Savitt, \emph{General {S}erre weight
  conjectures}, J. Eur. Math. Soc. (JEMS) \textbf{20} (2018), no.~12,
  2859--2949. \MR{3871496}

\bibitem[GK14]{grosseklonne:spherical}
Elmar Gro{\ss}e-Kl\"{o}nne, \emph{On the universal module of {$p$}-adic
  spherical {H}ecke algebras}, Amer. J. Math. \textbf{136} (2014), no.~3,
  599--652. \MR{3214272}

\bibitem[Her11a]{herzig:inventiones}
Florian Herzig, \emph{The classification of irreducible admissible mod {$p$}
  representations of a {$p$}-adic {${\rm GL}_n$}}, Invent. Math. \textbf{186}
  (2011), no.~2, 373--434. \MR{2845621}

\bibitem[Her11b]{herzig:satake}
\bysame, \emph{A {S}atake isomorphism in characteristic {$p$}}, Compos. Math.
  \textbf{147} (2011), no.~1, 263--283. \MR{2771132}

\bibitem[Hey23a]{heyer:geometrical}
Claudius Heyer, \emph{The {G}eometrical {L}emma for {S}mooth {R}epresentations
  in {N}atural {C}haracteristic},
  \href{https://arxiv.org/abs/2303.14721}{https://arxiv.org/abs/2303.14721},
  2023.

\bibitem[Hey23b]{H22}
\bysame, \emph{The left adjoint of derived parabolic induction}, Math. Z.
  \textbf{305} (2023), no.~3, Paper No. 46, 60. \MR{4658617}

\bibitem[HK11]{HK}
Annette Huber and Guido Kings, \emph{A {$p$}-adic analogue of the {B}orel
  regulator and the {B}loch-{K}ato exponential map}, J. Inst. Math. Jussieu
  \textbf{10} (2011), no.~1, 149--190. \MR{2749574}

\bibitem[Hum78]{humphreys:GTM}
James~E. Humphreys, \emph{Introduction to {L}ie algebras and representation
  theory}, Graduate Texts in Mathematics, vol.~9, Springer-Verlag, New
  York-Berlin, 1978, Second printing, revised. \MR{499562}

\bibitem[Hum90]{humphreys-book}
\bysame, \emph{Reflection groups and {C}oxeter groups}, Cambridge Studies in
  Advanced Mathematics, vol.~29, Cambridge University Press, Cambridge, 1990.
  \MR{1066460}

\bibitem[HV12]{HV12}
Guy Henniart and Marie-France Vign\'{e}ras, \emph{Comparison of compact
  induction with parabolic induction}, Pacific J. Math. \textbf{260} (2012),
  no.~2, 457--495. \MR{3001801}

\bibitem[HV15]{HV15}
\bysame, \emph{A {S}atake isomorphism for representations modulo {$p$} of
  reductive groups over local fields}, J. Reine Angew. Math. \textbf{701}
  (2015), 33--75. \MR{3331726}

\bibitem[Jan03]{jantzen}
Jens~Carsten Jantzen, \emph{Representations of algebraic groups}, second ed.,
  Mathematical Surveys and Monographs, vol. 107, American Mathematical Society,
  Providence, RI, 2003. \MR{2015057}

\bibitem[Kon22]{kongsgaard:thesis}
Daniel Kongsgaard, \emph{On the {M}od p {C}ohomology of {P}ro-p {I}wahori
  {S}ubgroups}, ProQuest LLC, Ann Arbor, MI, 2022, Thesis (Ph.D.)--University
  of California, San Diego. \MR{4465058}

\bibitem[Kot86]{K86}
Robert~E. Kottwitz, \emph{Stable trace formula: elliptic singular terms}, Math.
  Ann. \textbf{275} (1986), no.~3, 365--399. \MR{858284}

\bibitem[Laz65]{lazard}
Michel Lazard, \emph{Groupes analytiques {$p$}-adiques}, Inst. Hautes
  \'{E}tudes Sci. Publ. Math. (1965), no.~26, 389--603. \MR{209286}

\bibitem[LS24]{lahirisorensen}
Aranya Lahiri and Claus Sorensen, \emph{Rigid vectors in {$p$}-adic principal
  series representations}, Israel J. Math. \textbf{259} (2024), no.~1,
  427--459. \MR{4732375}

\bibitem[NSW08]{NSW}
J\"{u}rgen Neukirch, Alexander Schmidt, and Kay Wingberg, \emph{Cohomology of
  number fields}, second ed., Grundlehren der mathematischen Wissenschaften
  [Fundamental Principles of Mathematical Sciences], vol. 323, Springer-Verlag,
  Berlin, 2008. \MR{2392026}

\bibitem[oGVAG09]{ugavigre}
University of~Georgia VIGRE Algebra~Group, \emph{On {K}ostant's theorem for
  {L}ie algebra cohomology}, Representation theory, Contemp. Math., vol. 478,
  Amer. Math. Soc., Providence, RI, 2009, University of Georgia VIGRE Algebra
  Group: Irfan Bagci, Brian D. Boe, Leonard Chastkofsky, Benjamin Connell,
  Bobbe J. Cooper, Mee Seong Im, Tyler Kelly, Jonathan R. Kujawa, Wenjing Li,
  Daniel K. Nakano, Kenyon J. Platt, Emilie Wiesner, Caroline B. Wright and
  Benjamin Wyser, pp.~39--60. \MR{2513265}

\bibitem[PR94]{platonovrapinchuk}
Vladimir Platonov and Andrei Rapinchuk, \emph{Algebraic groups and number
  theory}, Pure and Applied Mathematics, vol. 139, Academic Press, Inc.,
  Boston, MA, 1994, Translated from the 1991 Russian original by Rachel Rowen.
  \MR{1278263}

\bibitem[PT02]{polotilouine}
Patrick Polo and Jacques Tilouine, \emph{Bernstein-{G}elfand-{G}elfand
  complexes and cohomology of nilpotent groups over {$\Bbb Z_{(p)}$} for
  representations with {$p$}-small weights}, no. 280, 2002, Cohomology of
  Siegel varieties, pp.~97--135. \MR{1944175}

\bibitem[Ron18]{R19b}
Niccol\`{o} Ronchetti, \emph{A {S}atake homomorphism for the $\bmod \ p$
  derived {H}ecke algebra},
  \href{https://arxiv.org/abs/1808.06512}{https://arxiv.org/abs/1808.06512},
  2018.

\bibitem[Ron20]{R19a}
\bysame, \emph{On the cohomology of integral {$p$}-adic unipotent radicals},
  Comm. Algebra \textbf{48} (2020), no.~10, 4186--4213. \MR{4127114}

\bibitem[RRS92]{RRS}
Roger Richardson, Gerhard R\"{o}hrle, and Robert Steinberg, \emph{Parabolic
  subgroups with abelian unipotent radical}, Invent. Math. \textbf{110} (1992),
  no.~3, 649--671. \MR{1189494}

\bibitem[Sch11]{schneider:padicliegroups}
Peter Schneider, \emph{{$p$}-adic {L}ie groups}, Grundlehren der Mathematischen
  Wissenschaften [Fundamental Principles of Mathematical Sciences], vol. 344,
  Springer, Heidelberg, 2011. \MR{2810332}

\bibitem[Sch15]{S15}
\bysame, \emph{Smooth representations and {H}ecke modules in characteristic
  {$p$}}, Pacific J. Math. \textbf{279} (2015), no.~1-2, 447--464. \MR{3437786}

\bibitem[Ser65]{serre:cohdim}
Jean-Pierre Serre, \emph{Sur la dimension cohomologique des groupes profinis},
  Topology \textbf{3} (1965), 413--420. \MR{0180619}

\bibitem[Ser02]{serre:galoiscoh}
\bysame, \emph{Galois cohomology}, english ed., Springer Monographs in
  Mathematics, Springer-Verlag, Berlin, 2002, Translated from the French by
  Patrick Ion and revised by the author. \MR{1867431}

\bibitem[Sor21]{sorensen:hochschild}
Claus Sorensen, \emph{Hochschild cohomology and {$p$}-adic {L}ie groups},
  M\"{u}nster J. Math. \textbf{14} (2021), no.~1, 101--122. \MR{4300164}

\bibitem[Spr09]{springer}
T.~A. Springer, \emph{Linear algebraic groups}, second ed., Modern
  Birkh\"{a}user Classics, Birkh\"{a}user Boston, Inc., Boston, MA, 2009.
  \MR{2458469}

\bibitem[SS23]{SS22}
Peter Schneider and Claus Sorensen, \emph{Duals and admissibility in natural
  characteristic}, Represent. Theory \textbf{27} (2023), 30--50. \MR{4561086}

\bibitem[SW00]{symondsweigel}
Peter Symonds and Thomas Weigel, \emph{Cohomology of {$p$}-adic analytic
  groups}, New horizons in pro-{$p$} groups, Progr. Math., vol. 184,
  Birkh\"auser Boston, Boston, MA, 2000, pp.~349--410. \MR{1765127}

\bibitem[Wei94]{weibel}
Charles~A. Weibel, \emph{An introduction to homological algebra}, Cambridge
  Studies in Advanced Mathematics, vol.~38, Cambridge University Press,
  Cambridge, 1994. \MR{1269324}

\bibitem[Yam22]{yamamoto}
Yuki Yamamoto, \emph{On {M}ackey {D}ecomposition for locally profinite groups},
  \href{https://arxiv.org/abs/2203.14262}{https://arxiv.org/abs/2203.14262},
  2022.

\end{thebibliography}

\end{document}